\documentclass[oneside,11pt]{amsart}
\usepackage{blkarray}
\usepackage{amssymb}
\usepackage{graphicx,caption,subcaption} 
\usepackage{amsmath}
\usepackage{amsthm}
\usepackage{amsfonts}
\usepackage{setspace}
\usepackage{mathrsfs}
\usepackage{hyperref}
\usepackage{tensor}
\usepackage{tikz}
\usepackage{tikz-cd}
\usepackage{epstopdf}
\usepackage{enumerate}
\usepackage[enableskew]{youngtab}
\usepackage{young}
\usepackage[all,cmtip]{xy}
\usepackage[margin = .8 in]{geometry}
\usepackage{hyperref,color,polynom}

\usepackage{mdframed}
\usepackage{lipsum}

\newcommand{\zz}{\ensuremath{\mathbb{Z}}}

\newcommand{\nn}{\ensuremath{\mathbb{N}}}

\newcommand{\ff}{\ensuremath{\mathbb{F}}}

\newcommand{\aaa}{\ensuremath{\mathbb{A}}}

\newcommand{\bA}{\ensuremath{\textbf{A}}}
\newcommand{\bB}{\ensuremath{\textbf{B}}}
\newcommand{\bC}{\ensuremath{\textbf{C}}}
\newcommand{\bD}{\ensuremath{\textbf{D}}}
\newcommand{\bI}{\ensuremath{\textbf{I}}}
\newcommand{\bSet}{\ensuremath{\textbf{Set}}}
\newcommand{\bAlg}{\ensuremath{\textbf{Alg}}}
\newcommand{\bGrp}{\ensuremath{\textbf{Grp}}}
\newcommand{\bSch}{\ensuremath{\textbf{Sch}}}
\newcommand{\bMod}{\ensuremath{\textbf{Mod}}}
\newcommand{\bCRing}{\ensuremath{\textbf{CRing}}}

\newcommand{\bDis}{\ensuremath{\textbf{Dis}}}
\newcommand{\bThin}{\ensuremath{\textbf{Thin}}}
\newcommand{\bSh}{\ensuremath{\textbf{Sh}}}

\newcommand{\mX}{\ensuremath{\mathcal{X}}}
\newcommand{\mY}{\ensuremath{\mathcal{Y}}}
\newcommand{\mZ}{\ensuremath{\mathcal{Z}}}
\newcommand{\mF}{\ensuremath{\mathcal{F}}}
\newcommand{\mG}{\ensuremath{\mathcal{G}}}
\newcommand{\mU}{\ensuremath{\mathcal{U}}}
\newcommand{\mD}{\ensuremath{\mathcal{D}}}
\newcommand{\mS}{\ensuremath{\mathcal{S}}}
\newcommand{\mP}{\ensuremath{\mathcal{P}}}
\newcommand{\mE}{\ensuremath{\mathcal{E}}}
\newcommand{\mH}{\ensuremath{\mathcal{H}}}

\newcommand{\mV}{\ensuremath{\mathcal{V}}}

\newcommand{\sO}{\ensuremath{\mathscr{O}}}

\newcommand{\fr}{\ensuremath{\mathfrak{r}}}
\newcommand{\fp}{\ensuremath{\mathfrak{p}}}
\newcommand{\fq}{\ensuremath{\mathfrak{q}}}

\newcommand{\fa}{\ensuremath{\mathfrak{a}}}
\newcommand{\fb}{\ensuremath{\mathfrak{b}}}
\newcommand{\fd}{\ensuremath{\mathfrak{d}}}

\newcommand{\free}{\ensuremath{\text{free}}}
\newcommand{\CRP}{\ensuremath{\text{CRP}}}
\newcommand{\loc}{\ensuremath{\text{loc}}}
\newcommand{\Sum}{\ensuremath{\text{Sum}}}
\newcommand{\Hom}{\ensuremath{\text{Hom}}}
\newcommand{\Nat}{\ensuremath{\text{Nat}}}
\newcommand{\op}{\ensuremath{\text{op}}}

\newcommand{\Mat}{\ensuremath{\text{Mat}}}
\newcommand{\eval}{\ensuremath{\text{Eval}}}

\newcommand{\GL}{\ensuremath{\text{GL}}}

\newcommand{\Fl}{\ensuremath{\text{Fl}}}

\newcommand{\id}{\ensuremath{\text{id}}}

\newcommand{\Sub}{\ensuremath{\text{Sub}}}

\newcommand{\Gr}{\ensuremath{\text{Gr}}}

\newcommand{\Perm}{\ensuremath{\text{Perm}}}
\newcommand{\End}{\ensuremath{\text{End}}}
\newcommand{\Spec}{\ensuremath{\text{Spec}}}

\newcommand{\Stab}{\ensuremath{\text{Stab}}}
\newcommand{\Orb}{\ensuremath{\text{Orb}}}

\newcommand{\Supp}{\ensuremath{\text{Supp}}}
\newcommand{\fppf}{\ensuremath{\text{fppf}}}

\newcommand{\fin}{\ensuremath{\text{fin}}}

\newcommand{\im}{\operatorname{im}}
\newcommand{\sh}{\operatorname{sh}}

\newcommand{\Span}{\ensuremath{\text{span}}}

\theoremstyle{definition}

\newmdtheoremenv{frm-thm}{Theorem}
\newmdtheoremenv{frm-def}{Definition}
\newmdtheoremenv{frm-lem}{Lemma}

\newtheorem{definition}{Definition}
\newtheorem{example}{Example}

\newtheorem{remark}{Remark}

\newtheorem{proof techniques}{Proof Techniques}

\newtheorem{lemma}{Lemma}
\newtheorem{corollary}{Corollary}

\newtheorem{note}{Note}

\newtheorem{theorem}{Theorem}
\newtheorem{proposition}{Proposition}

\newtheorem*{theorem*}{Theorem}

\begin{document}

\title{Well-Ordered Flag Spaces as Functors of Points}
\author{Nathaniel Gallup}
\date{}
\maketitle


\begin{abstract}
Using Grothendieck's ``functor of points'' approach to algebraic geometry, we define a new infinite-dimensional algebro-geometric flag space as a $k$-functor (for $k$ a ring) which maps a $k$-algebra $R$ to the set of certain well-ordered chains of submodules of an infinite rank free $R$-module. This generalizes the well known construction of a $k$-functor that is represented by the classical (i.e. finite-dimensional) full flag scheme. We prove that as in the finite-dimensional case, there is an action of a general linear group on our flag space, that the stabilizer of the standard flag is the subgroup $B$ of upper triangular matrices, and that the Bruhat decomposition holds, meaning that our space is covered by the disjoint Schubert cells $\sh(B \sigma B) / B$ indexed by permutations $\sigma$ of an infinite set. Finally, in the case of flags indexed by the ordinal $\omega + 1$, we define an analog of the Bruhat order on this infinite permutation group and prove that when $k$ is a domain, Ehresmann's closure relations still hold, i.e. that the closure $\overline{\sh(B \sigma B) / B}$ is covered by the Schubert cells indexed by permutations smaller than $\sigma$ in the infinite Bruhat order. 
\end{abstract}


\section{Introduction}

The goal of this paper is to introduce a new type of infinite-dimensional algebro-geometric flag space. It began with the observation that the proof of the Bruhat decomposition
\begin{equation*}
\Fl_n(k) = \bigsqcup_{\sigma \in S_n} B \sigma B / B
\end{equation*}
(where $k$ is a field and $B := B(k) \subseteq \GL_n(k)$ is the subgroup of upper triangular matrices, see \cite{lusztig2010} for a thorough review of the topic) still works, with only minor adjustments, if $\Fl_n$ is replaced by the set of certain increasing chains of subspaces in an infinite-dimensional vector space that are indexed by a \textbf{well-ordered}\footnote{This well-ordered condition originates from the notion of transfinite d\'{e}vissage which was introduced by Kaplansky in \cite{kaplansky1958} where he used it to prove the remarkable result that every projective module is a direct sum of countably generated submodules (see also \cite[{058T}]{stacks-project} for a nice treatment). Kaplansky d\'{e}vissage was also used heavily by Raynaud-Gruson \cite{raynaud-gruson1971} and Perry \cite{perry2010} to prove faithfully flat descent of projectivity.} set $I$, and $S_n$ is replaced by the group of permutations on the set of immediate successors $I_s \subseteq I$.

Classically, there are many methods used to ``geometrize'' (i.e. give the structure of a variety on) the set $\text{Fl}_n(k)$ of full flags in $k^n$. One method is to find open affine covering charts. A second is to embed $\text{Fl}_n(k)$ into a large projective space using the Pl\"{u}cker embedding, and show that the image is a Zariski closed set (see \cite{fulton1997} 9.1). Yet a third is to notice that the general linear group $G := \GL_n(k)$ acts transitively on the set of full flags, and the stabilizer of the standard flag is $B$. This gives a set-theoretic bijection $G / B \overset{\sim}{\leftrightarrow} \Fl_n(k)$, and it happens that the action of $B$ on $G$ by right multiplication is particularly nice (the action is free and the coordinate ring of $B$ is finitely generated projective over $k$), so the quotient $G / B$ is a variety as well. All three of these constructions give isomorphic variety structures on $\text{Fl}_k(n)$. Unfortunately, none of these seems to generalize easily to the set of $I$-flags for $I$ infinite. The first method fails because the corresponding charts are no longer affine. The second fails because no exterior power of an infinite dimensional vector space vanishes. And the third fails because the coordinate ring of $B$ is no longer finitely generated. 

This problem of geometrizing a set of flags in a given vector space is an example of a \emph{moduli problem}. A more modern method of solving such problems is Grothendieck's ``functor of points'' approach to algebraic geometry. Applied to our current situation, it gives a natural description of the flag variety $\Fl_n$ as the covariant functor from the category $\bAlg_k$ of algebras over a ring $k$ to the category $\bSet$ of sets which sends a $k$-algebra $R$ to the set of sequences $F_1 \subseteq F_2 \subseteq \ldots \subseteq F_n$ of direct summands of $R^n$ where $F_r$ has (projective) rank $r$. This functor is then represented by the (scheme corresponding to the) classical variety described above in the sense that it is isomorphic to the functor $\Hom_{\bSch/k}(\Spec(-) , \Fl_n(k) )$. 

The functor of points method \textbf{does} seem to translate well to the case of infinite $I$-flags. In Sections \ref{sec: submodule filtrations} and \ref{sec: flags} we give a careful definition of an $I$-flag, and in Section \ref{sec: full flag functor}, using methods developed in Section \ref{sec: the grassmannian}, we define a $k$-functor $\Fl_I$ generalizing that above. Although this functor is no longer representable by a scheme, we will show that at least the following theorem holds. 

\begin{theorem*}
The $k$-functor $\Fl_I$ is a sheaf in the fppf topology. 
\end{theorem*}

Hence we can still study some of its geometric properties. In Section \ref{sec: the bruhat decomposition}, for each permutation $\sigma \in \Perm(I_s)$, we define a subfunctor $X_\sigma^\circ$ called a \emph{Schubert cell} of the flag functor which can be identified with $\sh(B \sigma B) / B \subseteq \GL_{I_s} / B$\footnote{Here $\sh( - )$ denotes the fppf sheafification.} and we prove the following theorem, which we call the \emph{infinite Burhat decomposition}.

\begin{theorem*}
Let $I$ be a well-ordered set. The subfunctors $X^\circ_\sigma$ for $\sigma \in \Perm(I_s)$ are disjoint and cover the $k$-functor $\Fl_I$. 
\end{theorem*}

In the classical case, Ehresmann (\cite{ehresmann1934}, see also \cite{chevalley1994}) proved that the double cosets $B \sigma B$ and the Schubert cells $X_\sigma^\circ = B \sigma B / B$ satisfy the following remarkable closure property. 
\begin{equation}\label{eq: closure of a schubert cell}
\overline{B \sigma B} = \bigsqcup_{\tau \leq \sigma} B \tau B \ \ \ \ \ \ \ \ \overline{ X_\sigma^\circ } = \bigsqcup_{\tau \leq \sigma} X_\tau^\circ
\end{equation}
Here ``$\tau \leq \sigma$'' refers to the \emph{Bruhat order} on the symmetric group $S_n$. In Section \ref{sec: the infinite bruhat order}, for the case of $I = \omega + 1$, we give a definition and several characterizations of an analog of the Bruhat order on $\Perm(\nn)$ which we call the \emph{infinite Burhat order} and in Section \ref{sec: closure of the schubert cells} we prove that the following functorial version of Ehresmann's closure properties holds. 

\begin{theorem*}
Let $k$ be a \textbf{domain} and let $\sigma \in \Perm(\nn)$. Then we have that $\overline{B \sigma B}$ is covered by the disjoint subfunctors $B \tau B$ for $\tau \leq \sigma$ and in the quotient the Schubert subfunctor $\overline{\sh(B \sigma B) / B}$ is covered by the disjoint subfunctors $\sh(B \tau B) / B$ for $\tau \leq \sigma$. 
\end{theorem*}

Our definition of an infinite-dimensional flag space is of course by no means the first. In \cite{dimitrov-penkov2004} (see also \cite{ignatyev-penkov2017}), Ivan Dimitrov and Ivan Penkov defined flag spaces very similar to ours, except that their \emph{generalized flags} are indexed by countable totally ordered sets in which every element has an immediate successor or an immediate predecessor (this, of course, is satisfied by any well-ordered set), and are compatible with a given fixed flag. These spaces are filtered colimits of flag varieties with closed immersions as the inclusions, i.e. they are \emph{ind-varieties}. In \cite{fresse-penkov2015} Penkov and Lucas Fresse proved a Bruhat decomposition for these generalized flag spaces as well as many geometric properties of the Schubert ind-varieties. The approach in this research program is certainly topically different from ours (they do not use functors of points), but their definitions actually do give a fundamentally different object from the one that we study. If we restrict to the case of flags indexed by $\omega + 1$, their generalized flag space consists of maximal ascending chains of subspaces which are \textbf{eventually equal} to a given fixed flag. The group which acts transitively in this case is a filtered colimit of the groups $\GL_n(k)$ which is a proper subgroup of $\GL_\nn(k)$, (whereas our group is all of the latter), and their ``Weyl group'' is the set of permutations of $\nn$ with finite support (whereas ours is all of $\Perm(\nn)$).

Recall that constant rank-$n$ projective $R$-modules are the same as Zariski locally free modules of rank $n$, and that in algebraic geometry a module satisfying the latter definition is often called a \emph{vector bundle}. Thus one can think of the flag functor as assigning to a $k$-algebra $R$ the set of certain chains of nested vector bundles. In \cite{drinfeld2006}, Vladimir Drinfeld surveys his work on infinite-dimensional vector bundles in algebraic geometry. Although he mentions results from Kaplansky \cite{kaplansky1958}, Raynaud-Gruson \cite{raynaud-gruson1971}, and Perry \cite{perry2010}, about discrete infinite-dimensional vector bundles (which is the topic of this paper), his results there are about Tate vector bundles (i.e. those which have locally linearly compact fibers). 

\subsection{Acknowledgements}

The author would like to thank Greg Kuperberg, Eric Babson, Steve Sawin, and Will Sawin for many helpful discussions over the past few years, and Will Sawin in particular for suggesting the functorial algebraic geometry approach from which essentially this entire paper sprung.


\section{Background}

In this section we fix notation and explain some background needed for the rest of the paper. It is quite long and the reader is encouraged to skip to Section \ref{sec: the grassmannian} and refer back when needed. This material will all be well-known to experts and we give proofs only when we could not find an appropriate reference. 

\subsection{Categories}
Categories will be denoted by bold letters, such as $\bC$, and functors will usually be denoted by calligraphic letters, such as $\mF$. If $\bC$ is a locally small category containing objects $A , B$, we denote by $\Hom_\bC(A , B)$ the set of morphisms in $\bC$ from $A$ to $B$. If $\bD$ is another category we denote the category of functors from $\bC$ to $\bD$ by $[\bC , \bD]$. Morphisms in this category are natural transformations. In particular, for any category $\bC$ we have the identity functor $\id_\bC : \bC \to \bC$. If $\mF : \bC \rightleftarrows \bD: \mG$ are functors such that $\mG \circ \mF = \id_\bC$ and $\mF\circ \mG = \id_\bD$, then we say that $\mF$ and $\mG$ are \emph{inverses} and that $\bC$ and $\bD$ are \emph{isomorphic categories}, although this situation is not very common. 

The \emph{opposite category} $\bC^\op$ is obtained from $\bC$ by keeping the objects the same and reversing all the arrows, so that every morphism $f : A \to B$ in $\bC$ corresponds to a unique morphism $f^\op : B \to A$ in $\bC^\op$. In this way, a contravariant functor from $\bC \to \bD$ is exactly the same as a covariant functor $\bC^\op \to \bD$. 


\subsection{Natural Transformations}

If $\mX , \mY : \bC \to \bD$ are functors, we denote the (possibly proper) class of natural transformations from $\mX$ to $\mY$ by $\Nat(\mX , \mY)$. We denote the \emph{identity natural transformation} from a functor $\mF$ to itself by $1_\mF$. 

There are two ways to compose a natural transformation and a functor to get another natural transformation. Suppose that $\mE , \mF : \bB \to \bC$ and $\mG,\mH : \bC \to  \bD$ are functors and $\eta : \mE \to \mF$ and $\epsilon : \mG \to \mH$ are natural transformations. Then we define the natural transformation $\mG \eta : \mG \circ \mE \to \mG \circ \mF$ by specifying for each object $x$ in $\bB$ the morphism $\mG(\eta_x)$. Similarly we define the natural transformation $\epsilon \mG : \mG \circ \mE \to \mH \circ \mE$ by specifying for each object $x$ in $\bB$ the morphism $\epsilon_{\mG(x)}$.

\begin{lemma}\label{lem: precomposition of a natural isomorphism with a functor}
If $\mE , \mF : \bB \to \bC$ and $\mG,\mH : \bC \to  \bD$ are functors and $\eta : \mE \to \mF$ and $\epsilon : \mG \to \mH$ are natural isomorphisms, then $\mG \eta$ and $\epsilon \mG$ are also natural isomorphisms.
\end{lemma}

\begin{proof}
Recall that a natural transformation is an isomorphism if and only if all of its components are isomorphisms. Then the first result follows because functors map isomorphisms to isomorphisms and the second is trivial. 
\end{proof}


\subsection{Fully Faithful Functors}

A functor is called \emph{full} if it is surjective on hom-sets, \emph{faithful} if it is injective on hom-sets, and \emph{fully faithful} if it is both full and faithful. 

\begin{lemma}\label{lem: objects are isomorphic iff their images under a fully faithful functor are isomorphic}
If $\mF : \bC \to \bD$ is fully faithful then a morphism $f  : x \to y$ in $\bC$ is an isomorphism if and only if $\mF(f) : \mF(x) \to \mF(y)$ is an isomorphism in $\bD$. In particular we have that $x \cong y$ if and only if $\mF(x) \cong \mF(y)$. 
\end{lemma}

\begin{lemma}\label{lem: fully faithful functors preserve and reflect commutative squares}
If $\mF : \bC \to \bD$ is fully faithful then a square in $\bC$ commutes if and only if its image in $\bD$ commutes as well. 
\end{lemma}


\subsection{The Pushforward of a Functor}

If $\mF : \bC \to \bD$ is a functor and $\bB$ is another category, then the \emph{pushforward} of $\mF$ is the functor $\mF_* : [\bB , \bC] \to [\bB , \bD]$ which sends $\mG \mapsto \mF \circ \mG$. 

\begin{theorem} \label{thm: pushforwards of fully faithful functors are also fully faithful}
If $\mF: \bC \to \bD$ is fully faithful, then for any category $\bB$ we have that $\mF_* : [\bB , \bC] \to [\bB , \bD]$ is also fully faithful. 
\end{theorem}

\begin{proof}
Let $\mG , \mH : \bB \to \bC$ be functors and $\epsilon : \mF_* \mG \to \mF_* \mH$ a natural transformation. Then for each object $x$ in $\bB$ we have a morphism $\epsilon_x : \mF \mG(x) \to \mF \mH(x)$. Hence because $\mF$ is fully faithful, there exists a unique $\eta_x : \mG(x) \to \mH(x)$ such that $\mF(\eta_x) = \epsilon_x$. To see that $\eta : \mG \to \mH$ is natural, suppose that $f : x \to y$ is a morphism in $\bB$. Consider the following diagrams. 
\begin{center}
\begin{tikzcd}
 \mG(x) \arrow[r, "\mG(f)"] \arrow[d , "\eta_x" '] &  \mG(y) \arrow[d , "\eta_y"] &&&& \mF \mG(x) \arrow[r, "\mF \mG(f)"] \arrow[d , "\epsilon_x = \mF(\eta_x) " '] &  \mF \mG(y) \arrow[d , "\epsilon_y = \mF(\eta_y) "]
\\  \mH(x) \arrow[r, "\mH(f)" '] &  \mH(y)  &&&& \mF \mH(x) \arrow[r, "\mF \mH(f)" '] & \mF \mH(y) 
\end{tikzcd}
\end{center}
The right hand square in $\bD$ commutes, hence because $\mF$ is fully faithful, by Lemma \ref{lem: fully faithful functors preserve and reflect commutative squares} the left hand square in $\bC$ commutes as well. 
\end{proof}


\subsection{Products of Categories}

If $\bC$ and $\bD$ are categories, one can form their \emph{product} $\bC \times \bD$. The objects in this category are ordered pairs $(A , B)$ where $A$ is an object in $\bC$ and $B$ is an object in $\bD$, and a morphism from $(A , B)$ to $(X , Y)$ consists of a pair of morphisms $(f , g)$ where $f : A \to X$ is a morphism in $\bC$ and $g  : B \to Y$ is a morphism in $\bD$. One can take the product of two functors $\mF : \bA \to \bC$ and $\mG : \bB \to \bD$ to obtain a functor $\mF \times \mG : \bA \times \bB \to \bC \times \bD$ which is defined in the obvious way, and furthermore if $\eta : \mF \to \mF'$ and $\epsilon : \mG \to \mG'$ are natural transformations, one can take their product as well to obtain a natural transformation $\eta \times \epsilon : \mF \times \mG \to \mF' \times \mG'$. 

If $\bC$ is locally small, then there is a functor $\Hom_\bC(- , -) : \bC^\op \times \bC \to \bSet$ which sends an object $(A , B)$ to $\Hom_\bC(A , B)$ and morphism $(f^\op , g) : (A , B) \to (X , Y)$ to the function $\Hom_\bC(A , B) \to \Hom_\bC(X , Y)$ which sends $[\varphi : A \to B] \mapsto [g \circ \varphi \circ f : X \to Y]$. 


\subsection{The Yoneda Lemma} 

If $A$ is an object in a category $\bC$, we denote by $h^A: \bC \to \textbf{Set}$ the functor sending an object $B \mapsto \text{Hom}_\bC(A , B)$ and a morphism $\varphi: B \to C$ to the function $[f: B \to C] \mapsto [\varphi \circ f : A \to C]$. We therefore obtain a functor $h^- : \bC^\op \to [\bC, \bSet]$ by sending an object $A \mapsto h^A$ and a morphism $\varphi: A \to B$ to the natural transformation $h^B \to h^A$ defined on $C \in \bC$ by $[f : B \to C] \mapsto [f \circ \varphi : A \to C]$. This functor is named the \emph{(contravariant) Yoneda embedding}, so called because by the ubiquitous \emph{(weak) Yoneda lemma}, it is fully faithful, i.e. the map $\Hom_\bC(C , B) \to \Nat(h^B , h^C)$ induced by the Yoneda embedding is a bijection. One says that a functor $\mX : \bC \to \bSet$ is \emph{representable} if it is isomorphic to $h^A$ for some $A \in \bC$. 

The \emph{(strong) Yoneda lemma} says that for any functor $\mY : \bC \to \bSet$ and any object $A \in \bC$, the map $\Nat(h^A, \mY ) \to \mY(A)$ given by $\Phi \mapsto \Phi_A(1_A)$ is a bijection. Its inverse is obtained by sending $y \in \mY(A)$ to the natural transformation $\Phi^y : h^A \to \mY$ given by $\Phi^y_B( \varphi : A \to B ) = \mY(\varphi)(y)$. 


\subsection{Limits and Colimits}

Let $\bI, \bC$ be categories. A \textbf{diagram of shape $\bI$} in $\bC$ is a functor $\mD : \bI \to \bC$. If $X \in \bC$ is an object, then $\Delta_X : \bI \to \bC$ denotes the \textbf{constant functor} which sends all the objects of $\bI$ to $X$ and all the morphisms of $\bI$ to the identity $1_X$. A \textbf{cone} (resp. \textbf{cocone}) over $\mD$ is a pair $(X , \eta)$ where $X \in \bC$ is an object and $\eta : \Delta_X \to \mD$ (resp. $\eta : \mD \to \Delta_X$) is a natural transformation. A \textbf{morphism} of cones (resp. cocones) $(X , \eta) \to (Y , \rho)$ is a morphism $f : X \to Y$ such that $\rho \circ \Delta_f = \eta$ (resp. $\Delta_f \circ \eta = \rho$), where $\Delta_f : \Delta_X \to \Delta_Y$ is the natural transformation induced by $f$. A cone (resp. cocone) $(X , \eta)$ over $\mD$ is a \textbf{limit of $\mD$} (resp. \textbf{colimit of $\mD$}) if it is a terminal (resp. initial) object in the category of cones over $\mD$.

The following proposition states that limits and colimits in functor categories can be computed component-wise. In other words, suppose that $\bC$ and $\bD$ are categories. For every object $X \in \bC$ we have a functor $\eval_X : [\bC , \bD] \to \bD$ which sends an object $\mF$ to $\mF(X)$ and a morphism (i.e. a natural transformation) $\eta : \mF \to \mG$ to $\eta_X$. 

\begin{proposition}\label{prop: limits and colimits in functor categories}
Suppose that $\mD : \bI \to [\bC, \bD]$ is a diagram of functors and $(\mF , \rho)$ is a cone (resp. cocone) over $\mD$. Then $(\mF(X) , \eval_X  \rho)$ is a cone (resp. cocone) over $\eval_X \circ \mD: \bI \to \bD$, and furthermore $(\mF , \rho)$ is a limit (resp. colimit) of $\mD$ if $(\mF(X) , \eval_X \circ \rho)$ is a limit (resp. colimit) of $\eval_X \circ \mD$ for all objects $X \in \bC$. The converse holds if $\bD$ has all limits of shape $\bI$.
\end{proposition}


\subsection{Filtered Colimits}

A category $\bC$ is \emph{filtered} if every finite diagram in $\bC$ has a cocone. 

\begin{proposition}\label{prop: filtered colimits of sets}
Suppose that $\bI$ is a small filtered category, $\mD : \bI \to \bSet$ is a diagram, $X$ is a set, and $\eta : \mD \to \Delta(X)$ is a natural transformation. Then $(X , \eta)$ is a colimit of $\mD$ if and only if 
\begin{itemize}

\item[(FC1)] The functions $\eta_i$ are injective for all objects $i \in \bI$. 

\item[(FC2)] The functions $\eta_i$ are jointly surjective, i.e. $X = \bigcup_{i \in I} \im \eta_i$. 

\end{itemize}
\end{proposition}

\begin{proof}

We prove only the reverse direction, as that is all we will need. Let $(Y ,\mu)$ be another cocone over $\mD$. We define a function $f: X \to Y$ as follows. Given any $x \in X$, by (FC2) there exists $i \in \bI$ and $s \in \mD(i)$ such that $\eta_i(s) = x$. We define $f(x) = \mu_i(s)$, and we now check that this is well-defined. Suppose that there is $j \in \bI$ and $t \in \mD(j)$ such that $\eta_j(t) = x$. Because $\bI$ is filtered, the diagram $\bDis\{ i , j \} \hookrightarrow \bI$ has a cocone, i.e. there exists an object $k \in \bI$ and maps $\alpha : i \to k$, $\beta : j \to k$. Then we have that 
\begin{equation*}
\eta_k(\alpha(s)) = \eta_i(s) = x =  \eta_j(t) = \eta_k(\beta(t)).
\end{equation*}
However by (FC1), $\eta_k$ is injective, hence $\alpha(s) = \beta(t)$. Then we compute:
\begin{equation*}
\mu_i(s) = \mu_k(\alpha(s)) = \mu_k(\beta(t)) = \mu_j(t). 
\end{equation*}
So indeed $f$ is well-defined. It is clear from the definition of $f$ that $\Delta(f) \circ \eta = \mu$. Finally if $g : X \to Y$ also satisfies $\Delta(g) \circ \eta = \mu$, then for any $x \in X$ using (FC2) choose $i \in \bI$ and $s \in \mD(i)$ such that $\eta_i(s) = x$, and compute 
\begin{equation*}
g(x) = g(\eta_i(s)) = \mu_i(s) = f(x). 
\end{equation*}
\end{proof}


\subsection{Discrete Categories} 

A \emph{discrete category} is a category whose only morphisms are the identity maps. If $S$ is a set, there is an associated discrete category, denoted by $\bDis(S)$. The set of objects of this category is simply $S$, and for $x , y \in S$, $\Hom_{\bDis(S)}(x , y)$ is empty unless $x = y$ in which case this set has a single element labeled $1_x$. A limit (resp. colimit) of a diagram whose shape is discrete is called a \emph{product} (resp. \emph{coproduct}). 


\subsection{Equalizers} 

A diagram in a category $\bC$ of the form 
\begin{equation*}
X \overset{f}{\to} Y \underset{h}{\overset{g}{\rightrightarrows}} Z
\end{equation*}
is called an \emph{equalizer} if $g \circ f = h \circ f$ and $X$ together with the maps $f$ and $g \circ f = h \circ f$ is a limit of the diagram $Y \underset{h}{\overset{g}{\rightrightarrows}} Z$. One can easily check that equalizers of sets have the following particularly nice characterization which we will use a lot. 

\begin{proposition}\label{prop: characterization of equalizers in set}
A diagram $X \overset{f}{\to} Y \underset{h}{\overset{g}{\rightrightarrows}} Z$ in $\bSet$ is an equalizer if and only if the following two conditions are satisfied. 
\begin{itemize}

\item[(EQ1)] The function $f$ is injective.

\item[(EQ2)] If $y \in Y$ is such that $h(y) = g(y)$, then $y \in \im f$. 

\end{itemize}
\end{proposition}


\subsection{Pullbacks}

A diagram in a category $\bC$ of the form 
\begin{center}
\begin{tikzcd}
 X \arrow[r, "f"] \arrow[d , "g" '] &  Y \arrow[d , "h"] \\
Z  \arrow[r, "k" '] &  W
\end{tikzcd}
\end{center}
is called a \emph{pullback} if the diagram commutes, and $X$ together with the maps $f$, $g$, and $h \circ f = k \circ g$ is a limit of the diagram $Z \overset{k}{\longrightarrow} W \overset{h}{\longleftarrow} Y$. 


\subsection{Adjoint Functors}

A \emph{hom-set adjunction} between a pair of categories $\bC$ and $\bD$ consists of functors $\mG : \bC \rightleftarrows \bD: \mF$ and a natural isomorphism $\Phi : \Hom_\bC(\mF - , -) \to \Hom_\bD(- , \mG -)$. Here $\Hom_\bC(\mF - , -) : \bD^\op \times \bC \to \bSet$ is the composition of functors $\Hom_\bC( - , -) \circ (\mF^\op \times \id_\bC)$ and similarly $\Hom_\bD(- , \mG -) = \Hom_\bD( - , -) \circ (\id_\bD^\op \times \mG)$. In this case we say that $\mF$ is \emph{left adjoint} to $\mG$, or that $\mG$ is \emph{right adjoint} to $\mF$, and write $\mF \dashv_\Phi \mG$ or simply $\mF \dashv \mG$. 

\begin{theorem}\label{thm: composition of adjoints are adjoints}
If $\mG : \bB \rightleftarrows \bC: \mF$ and $\mH: \bC \rightleftarrows \bD : \mE$ are functors with $\mF \dashv_\Phi \mG$ and $\mE \dashv_\Psi \mH$, then $\mF \circ \mE \dashv \mH \circ \mG$. 
\end{theorem}

\begin{proof}
We have that $\Phi : \Hom_\bB( \mF - , -) \to \Hom_\bC( - , \mG -)$ and $\Psi : \Hom_\bC( \mE - , -) \to \Hom_\bD( - , \mH -)$ are natural isomorphisms. Therefore by Lemma \ref{lem: precomposition of a natural isomorphism with a functor}, we have that $\Phi (\mE^\op \times \id_\bB) : \Hom_\bB( \mF \circ \mE - , -) \to \Hom_\bC( \mE - , \mG -)$ and $\Psi (\id_\bD^\op \times \mG) : \Hom_\bC( \mE - , \mG -) \to \Hom_\bD( - , \mH\circ \mG  -)$ are also natural isomorphisms. Composing these yields the desired natural isomorphism:
\begin{equation*}
\Psi (\id_\bD^\op \times \mG) \circ \Phi (\mE^\op \times \id_\bB) : \Hom_\bB( \mF \circ \mE - , -) \to \Hom_\bD( - , \mH\circ \mG  -)
\end{equation*}
\end{proof}

The following theorem shows that right adjoints of a given functor are unique up to a unique isomorphism (a similar theorem holds for left adjoints). 

\begin{theorem}\label{thm: adjoints are unique}
If $\mG, \mH : \bC \rightleftarrows \bD : \mF$ are functors with $\mF \dashv_\Phi \mG$ and $\mF \dashv_\Psi \mH$, then there exists a unique natural isomorphism $\theta : \mG \to \mH$ such that $\Psi \circ \Phi^{-1} = \Hom_\bD ( -  , - ) (1_{\id_\bD^\op} \times \theta)$.  
\end{theorem}

\begin{proof}
We have that $\Phi : \Hom_\bC( \mF - , -) \to \Hom_\bD( - , \mG -)$ and $\Psi : \Hom_\bC( \mF - , -) \to \Hom_\bD( - , \mH -)$ are natural isomorphisms, hence $\Psi \circ \Phi^{-1} : \Hom_\bD( - , \mG -) \to \Hom_\bD( - , \mH -)$ is an isomorphism in $[\bD^\op \times \bC , \bSet ]$. Let $E : [\bD^\op \times \bC , \bSet ] \to [\bC , [\bD^\op , \bSet]]$ be the obvious isomorphism. Then $E(\Psi \circ \Phi^{-1}) : E(\Hom_\bD( - , \mG -)) \to E(\Hom_\bD( - , \mH -))$ is also an isomorphism.  Consider the Yoneda embedding $h_- : \bD \to [\bD^\op , \bSet]$ and its pushforward $(h_-)_* : [\bC , \bD] \to [\bC , [\bD^\op , \bSet]]$. We note that $E(\Hom_\bD( - , \mG -)) = (h_-)_*\mG$ and similarly $E(\Hom_\bD( - , \mH -)) = (h_-)_*\mH$. Therefore $E(\Psi \circ \Phi^{-1}) : (h_-)_*\mG \to (h_-)_*\mH$ is an isomorphism. Since the Yoneda embedding is fully faithful, by Theorem \ref{thm: pushforwards of fully faithful functors are also fully faithful}, its pushforward is also fully faithful. Hence there exists a unique natural transformation $\theta : \mG \to \mH$ such that $(h_-)_*(\theta) = E(\Psi \circ \Phi^{-1})$. By Lemma \ref{lem: objects are isomorphic iff their images under a fully faithful functor are isomorphic}, $\theta$ is an isomorphism. The desired result follows since $E^{-1}( (h_-)_*(\theta)  )$ is equal to $\Hom_\bD(-  , -) (1_{\id_\bD^\op} \times \theta)$.
\end{proof}

The most important property of adjoint functors is stated in the following theorem. 

\begin{theorem}[Right Adjoints Preserve Limits (RAPL)]
If $\mG: \bC \rightleftarrows \bD : \mF$ are functors with $\mF \dashv \mG$ and if $A : \bI \to \bC$ is a diagram with limit $(X , \eta)$, then $(\mG X , \mG \eta)$ is a limit of the diagram $\mG \circ A : \bI \to \bD$. Dually, if $B : \bI \to \bD$ is a diagram with colimit $(Y , \epsilon)$ then $(\mF X , \mF \epsilon)$ is a colimit of the diagram $\mF \circ B : \bI \to \bC$.
\end{theorem}

\begin{proof}
\cite{maclane2013} V.5. 
\end{proof}


\subsection{Set Conventions} 

Suppose $P , X$ are sets. A \emph{family} in $X$ \emph{indexed} by $P$ is simply a function $S : P \to X$. We often denote the image of $p \in P$ under $S$ by $S_p$ and then denote the family by $(S_p : p \in P )$. A \emph{subfamily} of $S$ is a family obtained by restricting $S$ to a subset $Q \subseteq P$.


\subsection{Posets} 

A \emph{preordered set} consists of a set $P$ and a relation $\leq$ on $P$ satisfying the following conditions: (P1) (\emph{Reflexivity}) For all $p \in P$, we have $p \leq p$, (P2) (\emph{Transitivity}) If $p \leq q$ and $q \leq r$, then $p \leq r$. A \emph{filtered set} is a preordered set satisfying the extra condition: (FIL1) Every pair of elements has an upper bound. A \emph{poset} is a preordered set satisfying the extra condition: (P3) (\emph{Antisymmetry}) If $p \leq q$ and $q \leq p$, then $p = q$. A poset $(P , \leq)$ is \emph{totally (or linearly) ordered} if it satisfies the following additional condition: (T1) Given $p , q \in P$ either $p \leq q$ or $q \leq p$. If $P$ and $Q$ are posets, a function $f: P \to Q$ is \emph{order-preserving} if for all $x , y \in P$ with $x \leq y$ we have $f(x) \leq f(y)$. It is an \emph{order-embedding} if for all $x , y \in P$, $x \leq y \iff f(x) \leq f(y)$. 

Let $(P , \leq)$ be a poset with $p < q \in P$, and suppose that $p \leq r \leq q$ implies $r = p$ or $r = q$. Then we say that $q$ is a \emph{successor} of $p$, and $p$ is a \emph{predecessor} of $q$. If $q$ has a predecessor, we say that $q$ \emph{is a successor}. If $p$ has a successor, we say that $p$ \emph{is a predecessor}. An element $p \in P$ is called an \emph{upper limit element} (resp. \emph{lower limit element}) if does not have a predecessor (resp. successor). We denote the set of successor elements (resp. predecessor elements) in $P$ by $P_s$ (resp. $P_p$). Note that if $(I , \leq)$ is a totally ordered set, a successor (resp. predecessor) of $i \in I$ is unique if it exists, although this may not be true for general posets. 

We say that a poset $P$ is \emph{well-founded} if every nonempty subset of $P$ has a minimal element. If $I$ is linear and well-founded we call it \emph{well-ordered}.


\subsection{Thin Category associated to a Poset} 

A category $\bC$ is \emph{thin} if for any objects $X$ and $Y$ in $\bC$, there is at most one morphism $X \to Y$. If $(P , \leq)$ is a preordered set, we can form its \emph{associated thin category} $\bThin(P)$ whose objects are exactly the elements of $P$, $\Hom_{\bThin(P)}(x , y)$ is equal to $\{ i_{x , y} \}$ if $x \leq y$ and $\emptyset$ otherwise, and composition is given by $i_{y , z} \circ i_{x, y} = i_{x , z}$ where the latter exists by transitivity of the preorder. For any $x \in P$, by reflexivity $i_{x , x}$ exists and then it is clear from the definition of composition that this morphism is the identity $x \to x$. Note that if $(P ,\leq)$ is a filtered set, then $\bThin(P)$ is a filtered category. 


\subsection{Rings and Algebras}

All rings in this paper are commutative and have a multiplicative identity, denoted by $1$ or $1_R \in R$ if we wish to be specific. A ring homomorphism $R \to S$ preserves addition, multiplication, and sends $1_R \mapsto 1_S$. The category whose objects are rings and morphisms are ring homomorphisms will be denoted by $\bCRing$. 

Throughout this paper, \textbf{$k$ will denote a ring unless otherwise stated}. A \emph{$k$-algebra} is a pair $(\phi_A , A)$ where $A$ is a ring and $\phi_A: k \to A$ is a ring homomorphism called the \emph{structure homomorphism} of the $k$-algebra. A \emph{$k$-algebra homomorphism} $(\phi_A , A) \to (\phi_B , B)$ is a ring homomorphism $\varphi: A \to B$ compatible with the structure homomorphisms. 

We denote the resulting category by $\bAlg_k$. By abuse of notation, we usually suppress the morphism $\phi_A$ and just say ``$A$ is a $k$-algebra.'' The set of $k$-algebra homomorphisms from $A$ to $B$ will be written $\Hom_k(A , B)$.


\subsection{Free Modules}

The poset of submodules of an $R$-module $M$, ordered by inclusion, will be denoted by $\Sub(M)$. Suppose that $P$ is a set and $M$ is a $R$-module. We say that a family $(N_p \mid p \in P)$ of submodules of $M$ is \emph{independent} if given finitely many pairwise distinct elements $p_1, \ldots, p_\ell \in P$ and elements $n_i \in N_{p_i}$ for $1 \leq i \leq \ell$ such that $n_1 + \ldots + n_\ell = 0$, it follows that $n_i = 0$ for all $i$. Furthermore, we say that the family \emph{spans} $M$ if $\sum_{p \in P} N_p = M$. Note that if $(N_p : p \in P)$ is an independent family of subspaces of $M$ and $Q$ is a subset of $P$, the the subfamily $(N_q : q \in Q)$ is also independent. We say that $M$ is the \emph{internal direct sum} of the family $(N_p : p \in P)$ if (DS1) it is independent and (DS2) it spans $M$. In this case we write $M = \bigoplus_{p \in P} N_p$. We say that a submodule $N \subseteq M$ is a \emph{direct summand} if there exists $N' \subseteq M$ such that $M = N \oplus N'$. The poset of summands of $M$ will be denoted by $\Sum(M)$. 

Let $M$ be an $R$-module and $I$ be a set. A family $(b_i : i \in I)$ of elements of $M$ is called a \emph{basis} of $M$ if $M$ is the internal direct sum of the family $(\Span(b_i) : i \in I)$ of submodules. In this case every element of $M$ can be written as a unique (finite!) linear combination of the elements $b_i$. If a module $M$ has a basis, it is called a \emph{free} $R$-module.

We define $R^I$ to be the set of functions $I \to R$, which is an $R$-module, and indeed a product of $I$ copies of $R$ in $\bMod_R$. A function $a : I \to R$ is often denoted in ``tuple'' notation, i.e. by $(a_i)_{i \in I}$. In general $R^I$ \textbf{is not} a free $R$-module when $I$ is infinite. We also define $R^{\oplus I}$ to be the submodule of $R^I$ consisting of the functions $I \to R$ with finite support. The module $R^{\oplus I}$ \textbf{is} free. The \emph{standard basis} for this module is $(e_i : i \in I )$ where we define $e_i : I \to R$ by sending $j \mapsto 0$ for $j \neq i$ and $i \mapsto 1$. 

The category theory perspective on this is as follows. The forgetful functor $\mU_{\bMod_R} : \bMod_R \to \bSet$ has a left adjoint $\bMod_R \leftarrow \bSet : R^{\oplus -}$ which sends a set $I$ to the free $R$-module $R^{\oplus I}$. The free $R$-modules are exactly those which are isomorphic to something in the image of the functor $R^{\oplus -}$. 


\subsection{Projective Modules}

If $M$ is a projective $R$-module and $\fp$ is a prime ideal, then $M_\fp$ is a projective $R_\fp$-module. But since $R_\fp$ is a local ring, Kaplansky's theorem \cite{kaplansky1958} implies that $M_\fp$ is actually a free $R_\fp$-module. We define the \emph{projective rank} of $M$ at $\fp$ to be the rank of $M_\fp$ as a free $R_\fp$-module.


\subsection{Finitely Presented Modules and Algebras}

An $R$-module $M$ is called \emph{finitely presented} if there is an exact sequence $R^m \to R^n \to M \to 0$. A $k$-algebra $R$ is \emph{finitely presented} if $R$ is isomorphic (as a $k$-algebra) to $k[x_1, \ldots, x_n] / (f_1, \ldots, f_m)$ for some $n$ and some polynomials $f_1, \ldots, f_m \in k[x_1, \ldots, x_n]$. 

\begin{lemma}\label{lem: finitely presented is preserved by pushouts}
Suppose that $R$ is a $k$-algebra and $\varphi : R \to S$ and $\psi : R \to T$ are $k$-algebra homomorphisms. Suppose further that $\psi$ makes $T$ into a finitely presented $R$-algebra. Then $\iota_S : S \to S \otimes_R T$ which sends $s \mapsto s \otimes 1$ makes $S \otimes_R T$ into a finitely presented $S$-algebra. 
\end{lemma}

\begin{proof}
We have that $T \cong R[x_1, \ldots, x_n] / (f_1, \ldots, f_m)$, and hence that 
\begin{equation*}
S \otimes_R T \cong S[x_1, \ldots, x_n] / (\varphi(f_1), \ldots, \varphi(f_m))
\end{equation*} 
which is indeed a finitely presented $S$-algebra. 
\end{proof}


\subsection{Change of Rings}

If $R$ and $S$ are rings and $\varphi: R \to S$ is a ring homomorphism, then there exists a functor $\varphi^* : \bMod_S \rightarrow \bMod_R$ called \emph{restriction of scalars} which sends an $S$-module $N$ to the $R$-module with the same underlying abelian group as $N$ but with scalar multiplication given by $r , n \mapsto \varphi(r) n$ and sends an $S$-module homomorphism to itself. On the other hand, there exists a functor $\varphi_! : \bMod_R \to \bMod_S$ called the \emph{extension of scalars} which sends an $R$-module $M$ to the $S$-module $M \otimes_R S$ (with scalar multiplication given by $s, m \otimes t \mapsto m \otimes st$) and sends an $R$-module homomorphism $f$ to $f \otimes_R \id_S$. The relationship between these two functors is that $\varphi_!$ is left adjoint to $\varphi^*$. Explicitly, the isomorphism $\Phi_{R , S} : \Hom_{R}(M , \varphi^* N) \overset{\sim}{\to}  \Hom_{S}(\varphi_! M , N)$ is given by 
\begin{equation*}
[f : M \to \varphi^* N ]   \mapsto [ m \otimes s \mapsto s \cdot f(m) ].
\end{equation*} 

\begin{definition}[See \cite{demazure-gabriel1980} Lemma 3.14]\label{def: adapted function}
A function $f: M \to N$ is \emph{adapted} to $\varphi$ if $f : M \to \varphi^* N$ is an $R$-module homomorphism and the $S$-module homomorphism $\Phi_{R , S}(f) : \varphi_! M \to N$ given by the adjunction $\varphi_! \vdash \varphi^*$ is an isomorphism.
\end{definition}


\subsection{Change of Rings Twice}

If we also have a ring homomorphism $\psi : S \to T$, then it is clear that $(\psi \circ \varphi)^* = \varphi^* \circ \psi^*$. Therefore by Theorem \ref{thm: composition of adjoints are adjoints} we have that $\psi_! \circ \varphi_!$ is left adjoint to $(\psi \circ \varphi)^*$. But $(\psi \circ \varphi)_!$ is also left adjoint to $(\psi \circ \varphi)^*$, so by Theorem \ref{thm: adjoints are unique} there exists a unique natural isomorphism  $\psi_! \circ \varphi_! \cong (\psi \circ \varphi)_!$ compatible with the adjunctions. In particular this gives a canonical isomorphism $(M \otimes_R S) \otimes_S T \cong M \otimes_R T$. On simple tensors it is given by $(m \otimes s) \otimes t \mapsto m \otimes st$.


\subsection{Change of Rings and Direct Sums} 

Suppose that $M$ is an $R$-module and $(N_p \mid p \in P)$ is a family of submodules. This specifies a diagram $N : \bDis(P) \to \bMod_R$ sending $p \mapsto N_p$, and a natural transformation $i : N \to \Delta M$ where $i_p : N_p \to M$ is the inclusion map. Then the cocone $(M, i)$ is a coproduct of $N$ if and only if $M = \bigoplus_{p \in P} N_p$. If this condition holds and $\varphi : R \to S$ is a ring homomorphism, then because the functor $\varphi_! : \bMod_R \to \bMod_S$ is a left adjoint, it preserves colimits, hence $(\varphi_! M , \varphi_! i)$ is a coproduct of $\varphi_! \circ N$. This implies that for every $p \in P$, $\varphi_! (i_p) : \varphi_! N_p  \to \varphi_! M$ is injective and
\begin{equation}\label{eq: extension of scalars preserves direct sums}
\varphi_! M = \bigoplus_{p \in P} \varphi_!(i_p)[ \varphi_! N_p].
\end{equation}

We obtain the following result as a consequence.

\begin{proposition}\label{prop: extension of scalars applied to direct summands are direct summands}
If $M$ is an $R$-module and $N \subseteq M$ is a direct summand and $\varphi : R \to S$ is a ring homomorphism, then $\varphi_!(i_N) : \varphi_! N \to \varphi_! M$ is an injective map and the image $\varphi_!(i_N)[\varphi_! N]$ is a direct summand of $\varphi_! M$. (Here $i_N : N \to M$ is the inclusion map.) 
\end{proposition}


\subsection{Change of Rings and Localization} 

Suppose that $R$ is a ring, $S \subseteq R$ is a multiplicative subset, and $M$ is an $R$-module. Let $\loc_S: R \to S^{-1}R$ denote the map $r \mapsto \frac{r}{1}$. Then the \emph{localization of $M$ at $S$}, denoted $S^{-1} M$, is defined to be the $(S^{-1}R)$-module $(\loc_S)_! M$.


\subsection{Change of Rings and Free Modules}

Suppose that $\varphi : R \to S$ is a ring homomorphism. Notice that $\mU_{\bMod_R} \circ \varphi^* = \mU_{\bMod_S}$. Therefore by Theorem \ref{thm: composition of adjoints are adjoints}, both $S^{\oplus -}$ and $\varphi_! \circ R^{\oplus -}$ are left adjoint to $\mU_{\bMod_S}$, so by Theorem \ref{thm: adjoints are unique} we have the following result. 

\begin{proposition}\label{prop: isomorphism between direct sum tensor ring extension and free module}
There exists a unique natural isomorphism $\varphi_! \circ R^{\oplus -} \to S^{\oplus -}$ compatible with the adjunctions. For a given set $I$, the component of this natural isomorphism is equal to $\Phi_{R , S}(\varphi^I) : R^{\oplus I} \otimes_R S \to S^{\oplus I}$ which sends $(a_i)_{i \in I} \otimes s \mapsto (s\varphi(a_i))_{i \in I}$. It follows that $\varphi^I : R^{\oplus I} \to S^{\oplus I}$ is adapted to $\varphi$.
\end{proposition}

In particular, this gives the following.

\begin{corollary}\label{cor: extension of scalars preserves free rank}
If $\varphi : R \to S$ is a ring homomorphism and $F$ is a free $R$-module, then $\varphi_! F$ is a free $S$-module with rank equal to that of $F$. 
\end{corollary}


\subsection{Change of Rings and Finitely Generated Modules} 

\begin{proposition}\label{prop: extension of scalars preserves finitely generated}
If $\varphi : R \to S$ is a ring homomorphism and $M$ is a finitely generated $R$-module then $\varphi_! M$ is a finitely generated $S$-module. 
\end{proposition}

\begin{proof}
This follows from the fact that if $\{ x_i \mid i \in I \} \subseteq M$ is an $R$-generating set of $M$, then $\{ x_i \otimes 1_S \mid i \in I \}$ is an $S$-generating set of $\varphi_!  M$.
\end{proof}


\subsection{Change of Rings and Projective Modules} 

\begin{theorem}\label{thm: extension of scalars and projectivity}
Let $\varphi : R \to S$ be a ring homomorphism and $M$ an $R$-module.

\begin{enumerate}

\item If $M$ is a projective $R$-module then $\varphi_! M$ is a projective $S$-module. 

\item Projectivity satisfies faithfully flat descent: If $\varphi_! M$ is a projective $S$-module and $\varphi$ is \textbf{faithfully flat} then $M$ is a projective $R$-module. 

\end{enumerate}

\end{theorem}

\begin{proof}

\

\begin{enumerate}

\item If $M$ is a projective $R$-module, then there exists an $R$-module $N$ and a free $R$-module $F$ such that $F = M \oplus N$. Since $\varphi_!$ is a left adjoint it preserves direct sums and free modules, hence $\varphi_! F = \varphi_! M \oplus \varphi_! N$ and $\varphi_! F$ is a free $S$-module, therefore $\varphi_! M$ is projective as well. 

\item This was proved by Raynaud-Gruson \cite{raynaud-gruson1971} and Perry \cite{perry2010}. 

\end{enumerate}

\end{proof}

\begin{lemma}\label{lem: localization and homomorphisms}
Let $\varphi : R \to S$ be a ring homomorphism and $\fq \subseteq S$ a prime ideal. There exists a unique ring homomorphism $\psi : R_{\varphi^{-1}(\fq)} \to S_\fq$ making the following diagram commute.

\begin{center}
\begin{tikzcd}
 R \arrow[r, "\varphi"] \arrow[d , "\loc_{\varphi^{-1}(\fq)}" '] &  S \arrow[d , "\loc_\fq"] \\
 R_{\varphi^{-1}(\fq)} \arrow[r, "\psi" '] &   S_\fq
\end{tikzcd}
\end{center}

\end{lemma}

\begin{proof}
This follows from the universal property of localization, since $\varphi(R \smallsetminus \varphi^{-1}(\fq) ) \subseteq S \smallsetminus \fq$. 
\end{proof}

\begin{proposition}\label{prop: rank and localizations of projective modules}
Let $\varphi : R \to S$ be a ring homomorphism and $M$ an $R$-module. Suppose that both $M$ and $\varphi_! M$ are projective. (In particular by Theorem \ref{thm: extension of scalars and projectivity} this holds if $M$ is projective or if $\varphi_! M$ is projective and $\varphi$ is faithfully flat.) For all $\fq \in \Spec(S)$, the rank of the free $S_\fq$-module $(\varphi_! M)_\fq$ is equal to the rank of the free $R_{\varphi^{-1}(\fq)}$-module $M_{\varphi^{-1}(\fq)}$.
\end{proposition}

\begin{proof}
If $\fq \subseteq S$ is a prime ideal, combining Lemma \ref{lem: localization and homomorphisms} and Theorems \ref{thm: composition of adjoints are adjoints} and \ref{thm: adjoints are unique} we obtain a ring homomorphism $\psi : R_{\varphi^{-1}(\fq)} \to S_\fq$ with the following property: 
\begin{equation*}
(\loc_\fq )_! \circ \varphi_! = \psi_! \circ (\loc_{\varphi^{-1}(\fq)})_!
\end{equation*}

Therefore $(\varphi_! M)_\fq \cong \psi_! ( M_{\varphi^{-1}(\fq)} )$. Since extension of scalars preserves free rank (Proposition \ref{cor: extension of scalars preserves free rank}), it follows that the free $S_\fq$-module $(\varphi_! M)_\fq$ has the same rank as the free $R_{\varphi^{-1}(\fq)}$-module $M_{\varphi^{-1}(\fq)}$.
\end{proof}


\subsection{(Faithfully) Flat Modules}

An $R$-module $N$ is \emph{flat} if the functor $- \otimes_R N : \bMod_R \to \bMod_R$ is exact, i.e. if whenever $K \overset{f}{\to} L \overset{g}{\to} M$ is an exact sequence of $R$-modules, $K \otimes_R N \overset{f \otimes 1_N}{\to} L \otimes_R N \overset{g \otimes 1_N}{\to} M \otimes_R N$ is exact. Since $- \otimes_R N$ is always right exact for any $N$, $N$ is flat if and only if whenever $\varphi : L \to M$ is an injective $R$-module homomorphism, $\varphi \otimes 1_N : L \otimes_R N \to M \otimes_R N$ is also injective.

The module $N$ is \emph{faithfully flat} if a sequence of $R$-modules $K \overset{f}{\to} L \overset{g}{\to} M$ is exact if and only if the sequence $K \otimes_R N \overset{f \otimes 1_N}{\to} L \otimes_R N \overset{g \otimes 1_N}{\to} M \otimes_R N$ is exact. 


\subsection{Faithfully Flat Ring Homomorphisms}

A ring homomorphism $\varphi : R \to S$ is \emph{(faithfully) flat} if $\varphi^* S$ is a (faithfully) flat $R$-module. Explicitly this means that $\varphi$ is (faithfully) flat if given a sequence of $R$-modules $K \overset{f}{\to} L \overset{g}{\to} M$, this sequence is exact (if and) only if the sequence of $R$-modules $K \otimes_R S  \overset{f \otimes \id_S}{\to} L \otimes_R S \overset{g \otimes \id_S}{\to} M \otimes_R S$ is exact. 

In the following proposition (the proof of which can be found at \cite[{00H9}]{stacks-project}), we collect some equivalent characterizations of faithfully flat ring homomorphisms  for future use.

\begin{proposition}\label{prop: equivalent characterizations of faithful flatness}
Let $\varphi : R \to S$ be a flat ring homomorphism. The following are equivalent. 
\begin{itemize}

\item $\varphi$ is faithfully flat.

\item Whenever $\varphi_! M = M \otimes_R S = 0$, this implies $M = 0$.

\item The induced map $\varphi^* : \Spec(S) \to \Spec(R)$ is surjective. 

\item $\varphi$ is injective and $M \to M \otimes_R S$ is injective for every $R$-module $M$. 

\end{itemize}
\end{proposition}

\begin{lemma}\label{lem: flatness is preserved by pushouts}
If $\varphi : R \to S$ and $\psi : R \to T$ are ring homomorphisms and $\psi$ is (faithfully) flat then $S \to S \otimes_R T$ given by $s \mapsto s \otimes 1$ is also (faithfully) flat. 
\end{lemma}

\begin{proof}
First we show that $S \to S \otimes_R T$ is flat if $\psi$ is. Suppose that $0 \to L \to M \to N \to 0$ is an exact sequence of $S$-modules. Then we form the sequence $0 \to L \otimes_S S \otimes_R T \to M \otimes_S S \otimes_R T  \to N \otimes_S S \otimes_R T \to 0$ of $S$-modules. But since $L \otimes_S S \cong L$ and similarly for $M$ and $N$, this sequence is isomorphic to $0 \to L \otimes_R T \to M \otimes_R T  \to N \otimes_R T \to 0$ which is exact because $\psi$ is flat by hypothesis. 

Now suppose that $\psi$ is faithfully flat, and let $N$ be an $S$-module which satisfies $N \otimes_S S \otimes_R T = 0$. Then as $N \otimes_S S \otimes_R T \cong N \otimes_R T$, we have that $N \otimes_R T = 0$, and as $\psi$ is faithfully flat, by Proposition \ref{prop: equivalent characterizations of faithful flatness} it follows that $N = 0$. So again by Proposition \ref{prop: equivalent characterizations of faithful flatness}, we have that $\psi$ is faithfully flat. 
\end{proof}

Recall that we can think of $M \otimes_R S$ as either an $R$-module or and $S$-module (which we usually denote by $\varphi_!$), but this does not affect whether this module is $0$. Therefore if $S$ is already flat, then it is faithfully flat if and only if whenever $\varphi_! M  = 0$ it follows that $M = 0$.

\begin{lemma}[\cite{demazure-gabriel1980} Lemma 3.14]\label{lem: DG Lemma 3.14}
Let $\varphi: R \to S$ be a faithfully flat ring homomorphism. Suppose $N , K , L$ are modules over $S$, $S \otimes_R S$ and $S \otimes_R S \otimes_R S$ respectively, and we have maps
\begin{equation*}
N \ \substack{\overset{d}{\longrightarrow} \\ \overset{e}{\longrightarrow}}  \ K \ \substack{\overset{a}{\longrightarrow} \\ \overset{b}{\longrightarrow} \\ \overset{c}{\longrightarrow}} \ L
\end{equation*}
which are adapted, respectively to the ring homomorphisms 
\begin{equation*}
S \ \substack{\overset{\delta}{\longrightarrow} \\ \overset{\epsilon}{\longrightarrow}}  \ S \otimes_R S \ \substack{\overset{\alpha}{\longrightarrow} \\ \overset{\beta}{\longrightarrow} \\ \overset{\gamma}{\longrightarrow}} \ S \otimes_R S \otimes_R S
\end{equation*}
where we define $\delta(s) = s \otimes 1$ and $\epsilon(s) = 1 \otimes s$, and $\alpha(s \otimes t) = s \otimes t \otimes 1$, $\beta(s \otimes t) = s \otimes 1 \otimes t$, and $\gamma(s \otimes t) = 1 \otimes s \otimes t$. 
Suppose further that $a \circ d = b \circ d$, $a \circ e = c \circ d$, and $b \circ e = c \circ e$. Then if we define $J = \ker(d - e)$, then the inclusion of $J$ into $N$ induces an isomorphism $J \otimes_R S \overset{\sim}{\to} N$.

\end{lemma}


\subsection{The Amitsur Complex}

\begin{theorem}[Grothendieck, \cite{grothendieck1962}] \label{thm: amitsur complex is exact for faithfully flat homomorphisms}
If $\varphi : R \to S$ is faithfully flat, and $M$ is an $R$-module, then the \emph{Amitsur complex}
\begin{equation*}
0 \to M \overset{\id_M \otimes \varphi}{\longrightarrow} M \otimes_R S \overset{\id_M \otimes (\delta - \epsilon)}{\longrightarrow} M \otimes_R S \otimes_R S \overset{\id_M \otimes (\alpha - \beta + \gamma)}{\longrightarrow} M \otimes_R S \otimes_R S \otimes_R S \to \ldots
\end{equation*}
is an exact sequence of $R$-modules. In particular, 
\begin{equation*}
0 \to R \overset{\varphi}{\to} S \overset{\delta - \epsilon}{\longrightarrow} S \otimes_R S \overset{\alpha - \beta + \gamma}{\longrightarrow} S \otimes_R S \otimes_R S \to \ldots
\end{equation*}
is and exact sequence of $R$-modules. 
\end{theorem} 


\subsection{$k$-Functors}

A \emph{$k$-functor} is a functor $\mX: \bAlg_k \to \bSet$. A \emph{morphism of $k$-functors} $\mX \to \mY$ is just a natural transformation. The category of $k$-functors is simply the functor category $[\bAlg_k , \bSet]$. Here it must be said that as Demazure and Gabriel do in \cite{demazure-gabriel1970} in order for the foundations to be rigorous, we should fix two Grothendieck universes $\mathbb{U} \subseteq \mathbb{V}$ with $\nn \in \mathbb{U}$ and restrict all objects in any category we work with (such as $\bSet$) to be elements of $\mathbb{V}$, and restrict the elements of $\bAlg_k$ to only those $k$-algebras in $\mathbb{U}$ with $k$ also in $\mathbb{U}$ (Demazure and Gabriel call such rings \emph{models}). We will be loose with these foundations, but modifications can be made if desired by the reader.

A \emph{subfunctor} of $\mX$ is a $k$-functor $\mY$ such that for any $k$-algebra $A$, $\mY(A)$ is a subset of $\mX(A)$, and for any $k$-algebra morphism $\varphi: A \to B$, $\mY(\varphi) = \mX(\varphi)|_{\mY(A)}$. If $\mZ$ is a subfunctor of a $k$-functor $\mY$, and $f : \mX \to \mY$ is a morphism, then the \emph{inverse image} $f^{-1}\mZ$ is the subfunctor of $\mX$ defined by $R \mapsto f_R^{-1}(\mZ(R)) \subseteq \mX(R)$. We also define $\im f$ to be the subfunctor of $\mY$ given by $R \mapsto \im f_R \subseteq \mY(R)$. 

Via the Yoneda embedding, any $k$-algebra $R$ defines a $k$-functor $h^R: \bAlg_k \to \bSet$ which sends $A \mapsto \Hom_k(R , A)$. Sometimes, to emphasize that we are thinking of $R$ as a $k$-algebra, we write $h^R = h^R_{\bAlg_k}$. Then as before a $k$-functor $\mX$ is representable if there exists a $k$-algebra $R$ such that $\mX \cong h^R$. 

There is also a slightly more general notion of representability in the context of $k$-functors. We say that a $k$-functor $\mX$ is \emph{representable by schemes} if there exists a scheme $Y$ over $k$ such that $\mX$ is isomorphic to the $k$-functor $R \mapsto \Hom_{\bSch / k}( \Spec(R) , Y)$. If $Y$ happens to be an affine scheme, say isomorphic to $\Spec(A)$, then we have $\Hom_{\bSch / k}( \Spec(R) , \Spec(A)) \cong \Hom_k(A , R)$, so any representable $k$-functor is automatically representable by schemes. 

We now give a few classic examples of $k$-functors, and discuss their representability. 

\begin{example}[The Affine Line]
The \emph{affine line} is the functor $\aaa^1_k : \bAlg_k \to \bSet$ defined by sending $R \mapsto R$ and $[\varphi : R \to S] \mapsto \varphi$, i.e. it is the forgetful functor. It is represented by the polynomial ring $k[x]$ since we have a natural isomorphism $\Phi: h^{k[x]} \to \aaa^1_k$ given by $\Phi_R( \psi ) = \psi(x)$ for $\psi \in h^{k[x]}(R) = \Hom_k(k[x] , R)$. 
\end{example}

\begin{example}[Affine $n$-Space]
Extending the example above, \emph{affine $n$-space} is the functor $\aaa^n_k : \bAlg_k \to \bSet$ defined by sending $R \mapsto R^n$ and $[\varphi : R \to S] \mapsto [\varphi^n : (r_1, \ldots, r_n) \mapsto (\varphi(r_1), \ldots, \varphi(r_n)) ]$. This functor is represented by the polynomial ring $k[x_1, \ldots, x_n]$, since there is a natural isomorphism $\Phi: h^{k[x_1, \ldots, x_n]} \to \aaa^n_k$ given by $\Phi_R( \psi ) = ( \psi(x_1), \ldots, \psi(x_n) )$ for $\psi \in h^{k[x_1, \ldots, x_n]}(R) = \Hom_k(k[x_1, \ldots, x_n] , R)$. 
\end{example}

\begin{example}[Affine $I$-Space]
Let $I$ be an infinite set. For any $k$-algebra $R$, $R^{\oplus I} := \bigoplus_{i \in I} R$ and $R^I := \prod_{i \in I} R$ are not equal (the former is a strict submodule of the latter), so we have two possible generalizations of the previous examples: 
\

\begin{itemize}

\item There is a functor $\bAlg_k \to \bSet$ which sends $R \mapsto R^{I}$ and $[\varphi : R \to S ] \to  \varphi^I$. This functor is represented by the polynomial ring $k[x_i | i \in I ]$ for the same reason as above, and is studied e.g. in \cite{nagpal-snowden2021}.

\item We call the functor $\aaa^I_k : \bAlg_k \to \bSet$ which sends $R \mapsto R^{\oplus I}$ and $[\varphi : R \to S ] \to  \varphi^I$ \emph{affine $I$-space}. It turns out that this functor is not representable, though it is an ind-scheme.

\end{itemize}
\end{example}


\subsection{The Coordinate Ring of a $k$-Functor}

If $\mX$ is a $k$-functor, we define $\sO(\mX)$ to be the set $\Nat(\mX , \aaa^1)$\footnote{Note that here is one instance where it should be required that all $k$-algebras are contained in some universe $\mathbb{U}$, otherwise $\Nat(\mX , \aaa^1)$ will not in general be a set!}. Given $\eta , \rho \in \Nat(\mX , \aaa^1)$, we define the natural transformations $\eta + \rho$ and $\eta \rho$ by specifying for each $R \in \bAlg_k$, the functions $(\eta + \rho)_R, (\eta \rho)_R  : \mX(R) \to R$ to be $\eta_R + \rho_R$ and $\eta_R \rho_R$ respectively. One can check that this gives the structure of a commutative unital ring on $\sO(\mX)$. Furthermore, there is a canonical ring homomorphism $k \mapsto \sO(\mX)$ which sends $a \in k$ to the natural transformation $\alpha: \mX \to \aaa^1$ defined by letting $\alpha_R : \mX(R) \to R$ be the constant function $x \mapsto \phi_R(a)$ where $\phi_R : k \to R$ is the structure map. Hence $\sO(\mX)$ is a $k$-algebra, called the \emph{coordinate ring} of $\mX$. 

Furthermore, if $\rho : \mX \to \mY$ is a natural transformation, then pre-composition with $\rho$ gives a $k$-algebra homomorphism $\sO(\mY) \to \sO(\mX)$, and we obtain a functor $\sO : [\bAlg_k , \bSet] \to \bAlg_k^\op$. Recall that the Yoneda embedding is a functor $h^- : \bAlg_k^\op \to [\bAlg_k , \bSet]$. As is proven in \cite{demazure-gabriel1980} I.1.4.3, the relationship between these two functors is that the Yoneda embedding $h^-$ is right adjoint to the coordinate ring functor $\sO$. 

If $\mX$ is representable, i.e. $\mX \cong h^R$ for some $k$-algebra $R$, then the Yoneda lemma gives that
\begin{equation*}
\sO(\mX) \cong \sO(h^R) = \Nat(h^R , \aaa^1) \cong \aaa^1(R) = R.
\end{equation*}


\subsection{Pullbacks in the category of $k$-Functors}

Since limits and colimits in functor categories can be computed component-wise (Proposition \ref{prop: limits and colimits in functor categories}), in particular the pullback of a diagram $\mX \rightarrow \mY \leftarrow \mZ$ of $k$-functors is given by the functor which sends a $k$-algebra $A$ to the set $\mX(A) \times_{\mY(A)} \mZ(A)$. The category of $k$-functors has a terminal object, namely the functor $\aaa^0$ which sends a $k$-algebra $A$ to the set $\{ 0 \}$. It follows that the product of $\mX$ and $\mY$ is given by the pullback of the diagram $\mX \to \aaa^0 \leftarrow \mY$.  For any $k$-algebra $A$, we have that $h^k(A)$ has just a single element, namely the structure homomorphism $k \to A$, hence $\aaa^0 \cong h^k$, i.e. $\aaa^0$ is representable.

Suppose that $\mX \cong S$ , $\mY \cong h^R$, and $\mZ \cong h^T$ are representable $k$-functors and we have a diagram $\mX \rightarrow \mY \leftarrow \mZ$ of $k$-functors. Taking the coordinate ring yields a diagram of $k$-algebras isomorphic to $S \leftarrow R \rightarrow T$. The pushout of this diagram in the category of $k$-algebras is given by $S \otimes_R T$ together with the maps $S \to S \otimes_R T$ defined by $s \mapsto s \otimes 1$ and a similarly defined map $T \to S \otimes_R T$. Applying the Yoneda embedding to this pushout diagram of $k$-algebras yields a diagram of $k$-functors

\begin{center}
\begin{tikzcd}
 h^{S \otimes_R T}  \arrow[r] \arrow[d] &  h^T \arrow[d] \\
 h^S \arrow[r] & h^R
\end{tikzcd}
\end{center}

which is a pullback since the Yoneda embedding is a right adjoint and hence preserves limits. Therefore we have that the fiber product $\mX \times_{\mY} \mZ$ is isomorphic to $h^{S \otimes_R T}$ and is therefore representable. 


\subsection{Open and Closed Subfunctors}


We now use representable functors to define closed and open subfunctors, following \cite{jantzen2003}  I.1. Suppose that $R$ is a $k$-algebra and $\fr \subseteq R$ is an ideal. There are two important subfunctors $\mD(\fr)$ and $\mV(\fr)$ of $h^R$ defined respectively by $\mD(\fr)(A) = \{ \varphi : R \to A | A \varphi(\fr) = A \}$ and $\mV(\fr)(A) = \{ \varphi : R \to A | \varphi(\fr) = 0 \}$ for any $k$-algebra $A$. Note that for any ideal $\fr \subseteq R$, we have that $\mV(\fr) \cong h^{R / \fr}$, so this functor is representable. In general $\mD(\fr)$ may not be representable, however for any $f \in R$, we have that $\mD(f) \cong h^{R_f}$, so in this case it is. 

A subfunctor $\mY$ of a $k$-functor $\mX$ is called \emph{open} (resp. \emph{closed}) if for every $k$-algebra $R$ and every morphism $f: h^R_k \to \mX$, there exists an ideal $\fr \subseteq R$ such that $f^{-1}(\mY) = \mD(\fr)$ (resp. $f^{-1}(\mY) = \mV(\fr)$).


We now list some properties of closed subfunctors for future use, the proofs of which can be found in \cite{jantzen2003} I.1. 

\begin{proposition}\label{prop: properties of closed subfunctors}
\

\begin{enumerate}

\item If $\mX \cong h^R$ is representable then $\mY \subseteq \mX$ is open (resp. closed) if and only if $\mY = \mD(\fr)$ (resp. $\mY = \mV(\fr)$) for some ideal $\fr \subseteq R$. 

\item The subfunctor $\aaa^0 \subseteq \aaa^1$ which assigns to every $k$-algebra $R$ the subset $\{ 0 \} \subseteq R = \aaa^1(R)$ is closed. 

\item Let $\mX, \mY$ be $k$-functors, $\mZ \subseteq \mY$ a closed subfunctor, and $\eta : \mX \to \mY$ a natural transformation. The inverse image $\eta^{-1}(\mZ)$ is closed in $\mX$. 

\item The intersection of an arbitrary number of closed subfunctors is closed. 

\item If $\mX$ is a $k$-functor and $(\eta_i \mid i \in I)$ is a family of natural transformations $\mX \to \aaa^1$, then $V( \eta_i \mid i \in I ) := \bigcap_{i \in I} \eta_i^{-1}(\aaa^0)$ is closed in $\mX$. 

\end{enumerate}
\end{proposition}


Because of this proposition, it makes sense to define for a subfunctor $\mY$ of a $k$-functor $\mX$ the \emph{closure} $\overline{\mY}$ to be the intersection of all closed subfunctors of $\mX$ containing $\mY$. 


\subsection{$k$-Group Functors}

A \emph{$k$-group functor} is a functor $\mG : \bAlg_k \to \bGrp$ (the latter is the category of groups). If $\mG$ is a $k$-group functor, by post-composing with the forgetful functor $\bGrp \to \bSet$ one obtains a $k$-functor, called the \emph{underlying $k$-functor} which by abuse of notation we also denoted by $\mG$. A $k$-group functor is \emph{representable} if its underlying $k$-functor is representable.  

\begin{example}[The General Linear Group]
Let $n  \in \nn$. We define the \emph{general linear group functor} $\GL_n : \bAlg_k \to \bGrp$ by sending a $k$-algebra $R$ to the group $\GL_n(R)$ of $n \times n$ invertible matrices with entries in $k$, and a morphism $\varphi : R \to S$ to the group homomorphism $\GL_n(R) \to \GL_n(S)$ defined by $(r_{ij}) \mapsto (\varphi(r_{ij}) )$. Let $\det(x)$ denote the determinant polynomial in $k[x_{ij} \mid 1 \leq i,j \leq n]$. Then the general linear group functor is represented by the localized $k$-algebra $k[x_{ij} \mid 1 \leq i,j \leq n]_{\det(x)}$. 
\end{example}


\subsection{Zariski Sheaves}
A $k$-functor $\mX$ is called a \emph{Zariski sheaf} (or \emph{local}) if for any $k$-algebra $R$ and elements $f_1, \ldots, f_n \in R$ which generate the unit ideal, the following diagram is an equalizer in $\bSet$:

\begin{equation*}
\mX( R ) \to \prod_{i = 1}^n \mX(R_{f_i}) \rightrightarrows \prod_{i , j = 1}^n \mX(R_{f_i f_j}).
\end{equation*}

We shall use the following two important results about Zariski sheaves which can are I.1.8 (5) and I.1.12 (6) in \cite{jantzen2003} respectively. 

\begin{proposition}\label{prop: zariski sheaves satisfy fp1}
Any Zariski sheaf $\mX$ satisfies the following property
\begin{itemize}

\item[(FP1)] For any finite collection of $R$-algebras $R_1, \ldots, R_n$, the map 
\begin{equation*}
\mX(R_1 \times R_2 \times \ldots \times R_n) \to \mX(R_1) \times \mX(R_2) \times \ldots \times \mX(R_n)
\end{equation*}
induced by the projection maps $\pi_i : R_1 \times \ldots \times R_n \to R_i$ is a bijection. Note that this map is explicitly given by $x \mapsto ( \mX(\pi_1)(x) ,  \ldots \mX(\pi_n)(x) )$. 

\end{itemize}
\end{proposition}

\begin{proposition}\label{prop: closed subfunctors of zariski sheaves are zariski sheaves}
Any closed subfunctor of a Zariski sheaf is also a Zariski sheaf. 
\end{proposition}


\subsection{fppf-Sheaves}

An \emph{fppf covering family} of a $k$-algebra $R$ is a finite family $R_1 , R_2, \ldots, R_n$ of finitely presented $R$-algebras such that $R_1 \times \ldots \times R_n$ is a faithfully flat $R$-algebra. A $k$-functor $\mX$ is called an \emph{fppf sheaf} (or \emph{faisceau}) if for every $k$-algebra $R$ and every fppf covering family $R_1, \ldots, R_n$ of $R$, the following diagram is an equalizer in $\bSet$:
\begin{equation*}
\mX( R ) \to \prod_{i = 1}^n \mX(R_i) \rightrightarrows \prod_{i , j = 1}^n \mX(R_i \otimes_R R_j).
\end{equation*}
Note that the empty family is an fppf covering of the zero $k$-algebra, which implies that $\mX(0)$ is a singleton set. We denote the full subcategory of the category of $k$-functors consisting of fppf sheaves by $\bSh_{\fppf}$. 

The relationship between Zariski sheaves and fppf sheaves is described in the following proposition (see \cite{jantzen2003} I.5.3 (5)).

\begin{proposition}\label{prop: fppf sheaves are zariski sheaves}
Every fppf sheaf is a Zariski sheaf. 
\end{proposition}

There is another characterization of fppf sheaves which will be useful. If $A, B$ are $k$-algebras and $\varphi : A \to B$ is a $k$-algebra homomorphism, we say that $B$ is an \emph{fppf-$A$-algebra} if it is a faithfully flat $A$-module and a finitely presented $A$-algebra.  

\begin{theorem}\label{thm: fppf sheaf characterization}
A $k$-functor $\mX$ is an fppf-sheaf if and only if the following two conditions are satisfied.
\begin{itemize}

\item[(FP1)] For any finite collection of $R$-algebras $R_1, \ldots, R_n$, the map 
\begin{equation*}
\mX(R_1 \times R_2 \times \ldots \times R_n) \to \mX(R_1) \times \mX(R_2) \times \ldots \times \mX(R_n)
\end{equation*}
induced by the projection maps $\pi_i : R_1 \times \ldots \times R_n \to R_i$ is a bijection (see Proposition \ref{prop: zariski sheaves satisfy fp1}). 

\item[(FP2)] For all $k$-algebras $R$ and all fppf-$R$-algebras $S$, the diagram
\begin{equation*}
\mX( R ) \to \mX(S) \rightrightarrows \mX(S \otimes_R S).
\end{equation*}

 is an equalizer in $\bSet$, where the two arrows on the right are induced by the ring homomorphisms $\delta , \epsilon : S \to S \otimes_R S$ defined by $\delta(s) =  s \otimes 1$ and $\epsilon(s) =  1 \otimes s$. 

\end{itemize}
\end{theorem}

\begin{proof}
See \cite{jantzen2003} I.5.3(4).
\end{proof}

\begin{theorem} \label{thm: representable by schemes implies fppf sheaf}
If a $k$-functor $\mX$ is representable by schemes then it is an fppf sheaf. 
\end{theorem}

\begin{proof}
See \cite{jantzen2003} I.5.3(7).
\end{proof}


\subsection{Quotients and Sheafification}

Suppose that $\bC$ is a full subcategory of the category of $k$-functors that is closed under finite products, and suppose that $\mG$ is a $k$-group functor which, by composing with the forgetful functor $\bGrp \to \bSet$ is in $\bC$, and $\alpha : \mG \times \mX \to \mX$ is an action of $\mG$ on a $k$-functor $\mX$ also in $\bC$. We define also a natural transformation $\beta : \mG \times \mX \to \mX$ by $\beta_R(g , x) = x$. Then a morphism $f: \mX \to \mY$ is constant on $\mG$-orbits if and only if $f \circ  \alpha = f \circ \beta$. 

It can happen that the naive choice for a quotient of $\mX$ by the action of $\mG$, i.e. the $k$-functor $R \mapsto \mX(R) / \mG(R)$, is not an element of $\bC$. However we can still define a $k$-functor $\mY \in \bC$ to be the \emph{quotient of $\mX$ by $\mG$ in} $\bC$ if $\mY$ is a colimit of the diagram $\mG \times \mX \underset{\beta}{\overset{\alpha}{\rightrightarrows}} \mX$ in $\bC$. (Note that because $\bC$ is closed under finite products, $\mG \times \mX$ is also in $\bC$.) This tells us that if the quotient exists, it is unique (up to a unique isomorphism), but the problem is that it may not always exist. 

The functor $\mY : R \mapsto \mX(R) / \mG(R)$, together with the natural transformation $\pi : \mX \to \mY$ mapping $x \in \mX(R)$ to its $\mG(R)$-orbit in $\mX(R) / \mG(R)$, is the quotient of $\mX$ by $\mG$ in the entire category of $k$-functors. Suppose that the inclusion $i : \bC \to [\bAlg_k, \bSet]$ admits a left adjoint $\mS$. Then because left adjoints preserve colimits, $( \mS(\mY) , \mS(\pi) )$ will be the colimit of $\mS(\mG \times \mX )\underset{\mS(\beta)}{\overset{\mS(\alpha)}{\rightrightarrows}} \mS(\mX)$. However because $\bC$ is a full subcategory, the inclusion $i : \bC \to [\bAlg_k, \bSet]$ is fully faithful, hence by a result originally due to Gabriel-Zisman \cite{gabriel-zisman2012} the counit $\epsilon: \mS \circ i \to 1_\bC$ of the adjunction $\mS \vdash i$ is an isomorphism. So we have a diagram in $\bC$
\begin{center}
\begin{tikzcd}
 \mS(\mG \times \mX ) \arrow[r, shift left, "\mS(\alpha)"] \arrow[r, shift right, "\mS(\beta)" ']  \arrow[d , "\epsilon_{\mG \times \mX}" '] &  \mS(\mX) \arrow[r , "\mS(\pi)"] \arrow[d , "\epsilon_\mX"]  & \mS(\mY)\\
\mG \times \mX \arrow[r, shift left, "\alpha" ] \arrow[r , shift right, "\beta" ']&  \mX & 
\end{tikzcd}
\end{center}
where the vertical maps are isomorphisms, and hence $\mS(\mY) , \mS(\pi) \circ \epsilon_X^{-1}$ is the quotient of $\mX$ by $\mG$ in $\bC$. Such a left adjoint \textbf{does} exist (see \cite{demazure-gabriel1970} III \S 1, 1.8-1.12) when $\bC$ is the category of fppf $k$-sheaves (which is indeed closed under finite products because limits commute with limits, see \cite{maclane2013} IX.2), it is called the \emph{fppf sheafification functor} and is denoted by $\sh : [\bAlg_k , \bSet] \to \bSh_{\text{fppf}}$. This implies that if an fppf $k$-group $G$ acts on an fppf $k$-sheaf $\mX$, the quotient of $\mX$ of $\mX$ by $\mG$ in $\bSh_{fppf}$ always exists. We denote it by $\mX / \mG$. 


The following theorem gives an easy description of the fppf sheafification of certain nice subfunctors of fppf sheaves (see \cite{jantzen2003} I.5.4 (4)). 
 
\begin{theorem}\label{thm: easy fppf sheafification}
Suppose that $\mY$ is a fppf-$k$-sheaf and that $\mX \subseteq \mY$ is a subfunctor which satisfies (FP1). Then $\sh \mX$ is the subfunctor of $\mY$ given by
\begin{equation*}
\sh \mX(R) = \{ y \in \mY(R) \mid \exists \text{ an fppf ring homomorphism } \varphi: R \to S \text{ with } \mY(\varphi)(y) \in \mX(S) \}.
\end{equation*}
\end{theorem}

One consequence of Theorem \ref{thm: easy fppf sheafification} is the following.

\begin{corollary}\label{cor: fppf subfunctor that contains a subfunctor contains its sheafification}
Suppose that $\mY$ is an fppf $k$-sheaf, and $\mX, \mZ \subseteq \mY$ are subfunctors such that $\mX$ satisfies (FP1) and $\mZ$ is an fppf-sheaf. If $\mX \subseteq \mZ$ then $\sh (\mX) \subseteq \mZ$.  
\end{corollary}

\begin{proof}
Given $y \in \sh\mX(R)$, by Theorem \ref{thm: easy fppf sheafification} there exists an fppf ring map $\varphi : R \to S$ such that $\mY(\varphi)(y) \in \mX(S) \subseteq \mZ(S)$. But because $\mY$ is an fppf-sheaf, we have that $\mY(\delta) \mY(\varphi)(y) = \mY(\epsilon) \mY(\varphi)(y)$ which implies that $\mZ(\delta) \mY(\varphi)(y) = \mZ(\epsilon) \mY(\varphi)(y)$. So since $\mZ$ is an fppf-sheaf, there exists $z \in \mZ(R)$ such that $\mZ(\varphi)(z) = \mY(\varphi)(y)$. However since $\mY$ is an fppf-sheaf, $\mY(\varphi)$ is injective, and hence $z = y$, so $y \in \mZ(R)$. 
\end{proof}

Another important consequence of Theorem \ref{thm: easy fppf sheafification} is the fact that closed subfunctors of fppf-sheaves are themselves fppf-sheaves. 

\begin{corollary}\label{cor: closed subfunctors of sheaves are sheaves}
Let $\mY$ be an fppf-$k$-sheaf and $\mX \subseteq \mY$ a closed subfunctor. Then $\mX$ is also an fppf-sheaf. 
\end{corollary}

\begin{proof}
By Proposition \ref{prop: fppf sheaves are zariski sheaves}, $\mY$ is a Zariski sheaf, and so by Proposition \ref{prop: closed subfunctors of zariski sheaves are zariski sheaves}, $\mX$ is too. Then by Proposition \ref{prop: zariski sheaves satisfy fp1}, $\mX$ satisfies (FP1), hence Theorem \ref{thm: easy fppf sheafification} gives a nice description of $\sh \mX$ as a subfunctor of $\mY$. We will show that $\mX = \sh \mX$ thereby proving that $\mX$ is an fppf-sheaf. Let $R$ be a $k$-algebra and let $y \in \mY(R)$ be such that there exists an fppf ring map $\varphi : R \to S$ such that $\mY(\varphi)(y) \in \mX(S)$. Then by the Yoneda lemma, $y$ corresponds to a natural transformation $\eta : h^R \to \mY$ which for any $k$-algebra $A$ sends $\psi : R \to A$ to $\mY(\psi)(y) \in \mY(A)$. However by the definition of a closed subfunctor, there exists an ideal $\fa \subseteq R$ such that $\eta^{-1}(\mX) = V(\fa)$. But note that $\varphi \in \eta^{-1}(\mX)(S) = \{ \psi : R \to S \mid \mY(\psi)(y) \in \mX(S) \}$. Hence we must have $\varphi \in V(\fa)(S) = \{ \psi : R \to S \mid \psi(\fa) = 0 \}$. So $\varphi(\fa) = 0$. But $\varphi$ is faithfully flat, and hence injective by Proposition \ref{prop: equivalent characterizations of faithful flatness}. So it must be that $\fa = 0$, and hence that $V(\fa) = h^R$. In particular this means that the identity map $\id_R : R \to R$ is in $V(\fa)(R) = \eta^{-1}(\mX)(R)$, so $\mY(\id_R)(y) \in \mX(R)$. But functors take identity maps to identity maps, hence $\mY(\id_R) = \id_{\mY(R)}$, so we have that $y \in \mX(R)$. Since $R$ was arbitrary, this implies that $\sh \mX \subseteq \mX$ as desired. 
\end{proof}

Another application of Theorem \ref{thm: easy fppf sheafification} follows. If $f : \mX \to \mY$ is a morphism of fppf $k$-sheaves, the image subfunctor $\im f \subseteq \mY$ is not in general an fppf sheaf. However it does satisfy (FP1), so Theorem \ref{thm: easy fppf sheafification} yields a nice description of its sheafification, $\sh( \im f)$, which is called the \emph{fppf image} of $f$. Explicitly, we have
\begin{equation}\label{eq: fppf-image description}
\sh( \im f)(R) = \{ y \in \mY(R) \mid \exists \varphi: R \to S \text{ fppf with } \mY(\varphi)(y) = f_S(x) \text{ for some } x \in \mX(S) \}.
\end{equation}

Now suppose that a $k$-group functor $\mG$ acts on a $k$-functor $\mX$. Given an element $x \in \mX(k)$ and any $k$-algebra $R$ with structure homomorphism $\phi_R : k \to R$, we define $x^R = \mX(\phi_R)(x)$. In particular we have $x^k = x$. The element $x$ defines a subfunctor of $\mX$ by sending a $k$-algebra $R$ to $\{ x^R \}$, and we define the \emph{stabilizer} of $x$ under the $\mG$-action to be the subgroup functor 
\begin{equation*}
\Stab_\mG(x)(R) =\{ g \in \mG(R) \mid g x^R = x^R  \}.
\end{equation*}

If both $\mG$ and $\mX$ are fppf sheaves, then $\Stab_\mG(x)$ is too, and we have a right action of the fppf $k$-group functor $\Stab_\mG(x)$ on the fppf sheaf $\mG$ by left multiplication. Define a natural transformation $\pi^x : \mG \to \mX$ by letting $\pi^x_R(g) = g x^R$ for $R$ a $k$-algebra and $g \in \mG(R)$. The \emph{fppf-orbit} of $x$, denoted by $\Orb_\mG(x)$, is defined to be the fppf image of $\pi^x$ inside $\mX$. Explicitly, we have 
\begin{equation}\label{eq: fppf orbit}
\Orb_\mG(x)(R) = \{ y \in \mX(R) \mid \exists \text{ an fppf ring map } \varphi: R \to S \text{ and } g \in \mG(S) \text{ with } \mX(\varphi)(y) = g x^S \}.
\end{equation}

Then the analog of the orbit stabilizer theorem for functorial algebraic geometry is true (see \cite{jantzen2003} I.5.6 (2)).

\begin{theorem}[Orbit-Stabilizer Theorem]
We have an isomorphism $\mG / \Stab_\mG(x) \overset{\sim}{\to} \Orb_\mG(x)$. 
\end{theorem}

Suppose now that an fppf $k$-group functor $\mG$ acts on an fppf $k$-functor $\mX$ and that $\mY \subseteq \mX$ is a $\mG$-stable subsheaf, meaning that for any $k$-algebra $R$, $\mY(R)$ is $\mG(R)$-stable. Then there is a natural map $\mY / \mG \to \mX / \mG$ which is a monomorphism if the action of $\mG$ on $\mX$ is free (see \cite{jantzen2003} I.5.5). If $\pi : \mX \to \mX / \mG$ is the quotient map, then the image of the monomorphism $\mY / \mG \to \mX / \mG$ is equal to the fppf-image $\sh(\pi(\mY))$. 

In particular the action of $\Stab_\mG(x)$ on $\mG$ is free, hence if $\mZ \subseteq \mG$ is $\Stab_\mG(x)$-stable then $\mZ / \Stab_\mG(x) \to \mG / \Stab_\mG(x)$ is a monomorphism, and the former can be identified with the subfunctor of $\Orb_\mG(x) \subseteq \mX$ given by 
\begin{equation}\label{eq: image of a subfunctor under the quotient map}
R \mapsto \{ y \in \mX(R) \mid \exists \varphi: R \to S \text{ fppf with } \mX(\varphi)(y) = zx \text{ for some } z \in \mZ(R) \}. 
\end{equation}


\subsection{Ind-Schemes}

A \emph{(weak) $k$-ind-scheme} is a $k$-functor $\mX$ which is a filtered colimit of $k$-functors which are representable by $k$-schemes. An ind-scheme $\mX$ is called \emph{ind-affine} if it is a filtered colimit of $k$-functors which are representable (i.e. representable by affine schemes). 

\begin{proposition}\label{prop: ind-schemes are fppf sheaves}
Every ind-scheme is an fppf sheaf. 
\end{proposition}

\begin{proof}
This follows because filtered colimits commute with finite limits in $\bSet$ (see \cite[{002W}]{stacks-project}). 
\end{proof}


\section{The Grassmannian} \label{sec: the grassmannian}

\subsection{Definitions}


Let $k$ be a ring and $I$ be a set. We define a $k$-functor $\Gr_I: \bAlg_k \to \bSet$ by sending a $k$-algebra $R$ to the set of $R$-submodules of $R^{\oplus I}$ which are direct summands (equivalently submodules $N$ where $R^{\oplus I} / N$ is projective), and by sending a $k$-algebra homomorphism $\varphi: R \to S$ to the function which sends $N \in \Gr_I(R)$ to the image of $\varphi_! ( i_N ) [ \varphi_! N ] \subseteq \varphi_!(R^{\oplus I})$ under the canonical isomorphism $\varphi_!(R^{\oplus I}) = R^{\oplus I} \otimes_R S \overset{\sim}{\to} S^{\oplus I}$ from Proposition \ref{prop: isomorphism between direct sum tensor ring extension and free module}. Here $i_N : N \to R^{\oplus I}$ denotes the inclusion map, and we note that $\varphi_!(i_N)$ is injective by Proposition \ref{prop: extension of scalars applied to direct summands are direct summands} so that $\Gr_I(\varphi)(N)$ is isomorphic to $\varphi_! N = N \otimes_R S$.


\subsection{Representability and the Case of Finite $I$}


If $I$ is finite, $\Gr_I$ is isomorphic to the \emph{full Grassmannian} functor from \cite{karpenko2000} Section 9. There it is proved that this functor is represented by a scheme, hence it is an fppf-sheaf. However we will be interested in the case when $I$ is infinite, and the proof of representability does not extend so we must prove the sheaf property using a more direct approach. 


\subsection{The Grassmannian is an fppf-Sheaf}


We begin with a few lemmas about the functor $\Gr_I$ and then move on to prove it is an fppf sheaf. 


\begin{lemma}\label{lem: gr preserves direct sums}
If $\varphi : R \to S$ is a $k$-algebra homomorphism, then $\Gr_I(\varphi)$ preserves direct sums, i.e. if $N \in \Gr_I(R)$ and $( M_j \mid j \in J )$ is a family of elements of $\Gr_I(R)$ such that $N = \bigoplus_{j \in J} M_j$, then $\Gr_I(\varphi)(N) = \bigoplus_{j \in J} \Gr_I(\varphi)(M_j)$. 
\end{lemma}

\begin{proof}
Since $\varphi_! : \bMod_R \to \bMod_S$ preserves direct sums (Equation \ref{eq: extension of scalars preserves direct sums}), we have that 
\begin{equation*}
\varphi_! N = \bigoplus_{j \in J} \varphi_!(i_{M_j , N}) [\varphi_! M_j ],
\end{equation*}
where $i_{M_j , N} : M_j \to N$ is the inclusion map. Applying the injective map $\varphi_!(i_{N , R^{\oplus I}}) : \varphi_! N \to \varphi_!(R^{\oplus I})$ yields 
\begin{align*}
\varphi_!(i_{N , R^{\oplus I}})[  \varphi_! N ] = \bigoplus_{j \in J} \varphi_!(i_{N , R^{\oplus I}}) \circ \varphi_!(i_{M_j , N}) [\varphi_! M_j ] =  \bigoplus_{j \in J} \varphi_!(i_{N , R^{\oplus I}} \circ i_{M_j , N}) [\varphi_! M_j ] = \bigoplus_{j \in J} \varphi_!(i_{M_j , R^{\oplus I}}) [\varphi_! M_j ].
\end{align*}

Finally, applying the isomorphism $\varphi_!(R^{\oplus I}) \to S^{\oplus I}$, we obtain that $\Gr_I(\varphi)(N) = \bigoplus_{j \in J} \Gr_I(\varphi)(M_j)$ as desired.
\end{proof}


\begin{lemma} \label{lem: equivalent definitions of gr}
If $N \in \Gr_I(R)$ and $\varphi: R \to S$ is a $k$-algebra homomorphism, then $\Gr_I(\varphi)(N)$ is the $S$-submodule of $S^{\oplus I}$ generated by the image of $N$ under the $R$-module homomorphism $\varphi^{I} : R^{\oplus I} \to S^{\oplus I}$. 
\end{lemma}

\begin{proof}
By definition $\Gr_I(\varphi)(N)$ is equal to the image of 
\begin{equation*}
\varphi_!(i_N) [\varphi_! N] = \{ \sum (n_i)_{i \in I} \otimes s \mid (n_i)_{i \in I} \in N , s \in S \} \subseteq R^{\oplus I} \otimes S
\end{equation*} 
under the isomorphism $\Phi_{R , S}(\varphi^I): R^{\oplus I} \otimes S \to S^{\oplus I}$ which sends $(r_i)_{i \in I} \otimes s \mapsto s (\varphi(r_i) )_{i \in I}$. Therefore $\Gr_I(\varphi)(N) = \{ \sum s (\varphi(n_i))_{i \in I} \mid (n_i)_{i \in I} \in N , s \in S \}$ which is clearly equal to the $S$-submodule of $S^{\oplus I}$ generated by the image of $N$ under $\varphi^{I}$. 
\end{proof}


\begin{lemma} \label{lem: surjective homs make image a submodule}
If $\varphi : R \to S$ is a \textbf{surjective} $k$-algebra homomorphism and $N \in \Gr_I(R)$, then $\Gr_I(\varphi)(N) = \varphi^I ( N)$. 
\end{lemma}

\begin{proof}
The inclusion ($\supseteq$) is immediate. For the reverse inclusion ($\subseteq$), by Lemma \ref{lem: equivalent definitions of gr} it is enough to show that $\varphi^I ( N)$ is already an $S$-submodule of $S^{\oplus I}$. Given $(n_i)_{i \in I} \in N$ and $s \in S$, since $\varphi$ is surjective there exists $r \in R$ such that $\varphi(r) = s$. Then the following computation gives the desired result. 
\begin{equation*}
s \varphi^I( (n_i)_{i \in I} ) = \varphi(r) (\varphi(n_i))_{i \in I} =  (\varphi(r) \varphi(n_i))_{i \in I} =  (\varphi(r n_i))_{i \in I} =  \varphi^I ((r n_i)_{i \in I} )
\end{equation*}
\end{proof}


\begin{lemma} \label{lem: power of phi is adapted}
Let $\varphi : R \to S$ be a ring homomorphism and let $N \in \Gr_I(R)$. The map $\varphi^I : N \to \Gr_I(\varphi)(N)$ is adapted to $\varphi$ (see Definition \ref{def: adapted function}). 
\end{lemma}

\begin{proof}
It is clear that $\varphi^I : N \to \varphi^* \Gr_I(\varphi)(N)$ is an $R$-module homomorphism. Furthermore, we have previously shown that the composition of the maps 
\begin{equation*}
\varphi_! N \overset{\varphi_!(i_N)}{\longrightarrow} \varphi_! R^{\oplus I} \overset{\Phi_{R , S}(\varphi^I)}{\longrightarrow} S^{\oplus I}
\end{equation*}
is an isomorphism onto its image, which is by definition $\Gr_I(\varphi)(N)$. But notice that this isomorphism is equal to the map $\varphi_! N \to \Gr_I(\varphi)(N)$ induced by the adjunction $\varphi_! \vdash \varphi^*$.
\end{proof}


\begin{theorem} \label{thm: grassmannian is an fppf sheaf}
$\Gr_I$ is an fppf sheaf.
\end{theorem}

\begin{proof}
We use Theorem \ref{thm: fppf sheaf characterization} and we begin with (FP1). Given $k$-algebras $R_1, \ldots,  R_n$, we need to show that the map $\Gr_I( R_1 \times \ldots \times R_n ) \to \Gr_I( R_1) \times \ldots \times \Gr_I(R_n )$ is a bijection. 

To prove injectivity, suppose that $M , N \in \Gr_I( R_1 \times \ldots \times R_n )$ are such that $\Gr_I( \pi_j )(M) = \Gr_I( \pi_j )(N)$ for $1 \leq j \leq n$. Since the projection maps $\pi_j$ are surjective, it follows from Lemma \ref{lem: surjective homs make image a submodule} that $\Gr_I( \pi_j )(M) = \pi_j^I (M)$ and similarly for $N$. Therefore given $(m_i)_{i \in I} \in M$, because $\pi_j^I (M) = \pi_j^I (N)$ for each $1 \leq j \leq n$, there exists $(n^j_i)_{i \in I} \in N$ such that $(\pi_j(m_i))_{i \in I}  = (\pi_j(n^j_i))_{i \in I}$. Therefore we can write 
\begin{equation*}
(m_i)_{i\in I} = \sum_{j = 1}^n e_j (n^j_i)_{i \in I} \in N
\end{equation*}
where $e_j$ is the idempotent $(0 , \ldots, 1 , \ldots, 0) \in R_1 \times \ldots \times R_n$ with a $1$ in the $j$th entry and $0$'s everywhere else. Hence $M \subseteq N$. A similar computation shows the other containment, so $M = N$. 

For surjectivty, suppose we are given $M_j \in \Gr_I(R_j)$, with direct sum complements $M_j \oplus M_j' = R_j^{\oplus I}$. We note that as $(R_1 \times \ldots \times R_n)$-modules we have $R_1^{\oplus I} \times \ldots \times R_n^{\oplus I} = (M_1 \times \ldots \times M_n) \oplus (M_1' \times \ldots \times M_n')$. Furthermore, under the $(R_1 \times \ldots \times R_n)$-module isomorphism
\begin{align*}
R_1^{\oplus I} \times \ldots \times R_n^{\oplus I} &\overset{\sim}{\to} (R_1 \times \ldots \times R_n)^{\oplus I}
\\ ( (x_i^1)_{i \in I} , \ldots , (x_i^n)_{i \in I} ) &\mapsto ( (x^1_i , \ldots , x^n_i ) )_{i \in I}
\end{align*}
$M_1 \times \ldots \times M_n$ corresponds to $L = \bigcap_{j = 1}^n (\pi^I_j)^{-1}(M_j)$ and $M_1' \times \ldots \times M_n'$ corresponds to $L' = \bigcap_{j = 1}^n (\pi^I_j)^{-1}(M_j')$. Then we have that $L \oplus L' = (R_1 \times \ldots \times R_n)^{\oplus I}$, so that $L \in \Gr_I(R_1 \times \ldots \times R_n)$. Furthermore, given $(m_i)_{i \in I} \in M_j$, we note that $( m_i e_j )_{i \in I} \in (R_1 \times \ldots \times R_n)^{\oplus I}$ is in fact in $L$, since $\pi_j^I ( ( m_i e_j )_{i \in I} ) =  (m_i)_{i \in I} \in M_j$ and $\pi_\ell^I ( ( m_i e_j )_{i \in I} ) = 0 \in M_\ell$ for all $\ell \neq j$. But this computation also shows that $\Gr_I(\pi_j)(L) = \pi_j^I( L ) = M_j$, as desired.

Now we show (FP2). Suppose that $\varphi: R \to S$ is a $k$-algebra homomorphism which makes $S$ into an fppf-$R$-algebra. We use Proposition \ref{prop: characterization of equalizers in set} to show that the sequence $\Gr_I(R) \to \Gr_I(S) \rightrightarrows \Gr_I(S \otimes_R S)$ is an equalizer diagram in $\bSet$. 

First we prove (EQ1). Given $M , N \in \Gr_I(R)$ such that $\Gr_I(\varphi)(M) = \Gr_I(\varphi)(N)$, we note that this implies $\im \varphi_!(i_M) = \im \varphi_!(i_N)$ in $\varphi_! (R^{\oplus I})$, where $i_M : M \to R^{\oplus I}$ and $i_N : N  \to R^{\oplus I}$ denote the inclusion maps. Consider the following commutative diagram whose rows are exact by Theorem \ref{thm: amitsur complex is exact for faithfully flat homomorphisms}. 
\begin{center}
\begin{tikzcd}
0 \arrow[r] & M \arrow[r , "\id_M \otimes \varphi"] \arrow[d , "i_M"] & M \otimes_R S   \arrow[d , "\varphi_!(i_M)"] \arrow[ r r , "\id_M \otimes (\delta - \epsilon)"] & & M \otimes_R S \otimes_R S \arrow[r] \arrow[d , "(\delta \circ \varphi)_! (i_M)"]  & \ldots
\\ 0 \arrow[r] & R^{\oplus I}  \arrow[r , "\id_{R^{\oplus I}} \otimes \varphi" '] & R^{\oplus I} \otimes_R S  \arrow[ r r , "\id_{R^{\oplus I}} \otimes (\delta - \epsilon)" '] & & R^{\oplus I} \otimes_R S \otimes_R S \arrow[r] & \ldots
\end{tikzcd}
\end{center}
Let $\rho = \id_{R^{\oplus I}} \otimes (\delta - \epsilon)$. Since $M$ is a summand of $R^{\oplus I}$, the columns of this diagram are all injective, which, by a diagram chase, implies that $\ker \rho \cap \im \varphi_!(i_M) = \{ m \otimes 1 \mid m \in M \}$. Now given $n \in N$, we have that $n \otimes 1 \in \ker \rho \cap \im \varphi_!(i_N)$. But since $\im \varphi_!(i_M) = \im \varphi_!(i_N)$, it follows that $n \otimes 1 \in \ker \rho \cap \im \varphi_!(i_M)$. Hence $n \otimes 1 = m \otimes 1$ for some $m \in M$, but by injectivity of the map $R^{\oplus I} \to R^{\oplus I} \otimes_R S$, it follows that $n = m$ and hence that $n \in M$. Therefore $N \subseteq M$, and a similar argument shows the reverse inclusion. 

Next we prove (EQ2). In the notation of Lemma \ref{lem: DG Lemma 3.14}, suppose that $N \in \Gr_I(S)$ is such that $\Gr_I(\delta)(N) = \Gr_I(\epsilon)(N)$. We define $K$ to be this latter $(S \otimes_R S)$-module. Note that the equations
\begin{equation}\label{eq: ring homomorphism composition equality}
\alpha \circ \delta = \beta \circ \delta, \ \ \ \gamma \circ \delta = \alpha \circ \epsilon, \ \ \ \beta \circ \epsilon = \gamma \circ \epsilon
\end{equation}

imply that
\begin{align*}
\Gr_I(\alpha)(K)&= \Gr_I(\alpha)\Gr_I(\delta)(N) = \Gr_I(\alpha \circ \delta)(N) = \Gr_I(\beta \circ \delta )(N) = \Gr_I(\beta) \Gr_I(\delta)(N) = \Gr_I(\beta)(K)
\\ \Gr_I(\beta)(K)  &= \Gr_I(\beta)\Gr_I(\epsilon)(N) = \Gr_I(\beta \circ \epsilon)(N) = \Gr_I( \gamma \circ \epsilon )(N) = \Gr_I(\gamma) \Gr_I(\epsilon)(N) = \Gr_I(\gamma)(K).
\end{align*}

Hence we define $L = \Gr_I(\alpha)(K) = \Gr_I(\beta)(K) = \Gr_I(\gamma)(K)$. By Lemma \ref{lem: power of phi is adapted} the maps 
\begin{equation*}
N \ \substack{\overset{\delta^I}{\longrightarrow} \\ \overset{\epsilon^I}{\longrightarrow}}  \ K \ \substack{\overset{\alpha^I}{\longrightarrow} \\ \overset{\beta^I}{\longrightarrow} \\ \overset{\gamma^I}{\longrightarrow}} \ L
\end{equation*}
are adapted, respectively, to the ring homomorphisms $\delta, \epsilon, \alpha, \beta, \gamma$. Furthermore, it follows directly from Equation \ref{eq: ring homomorphism composition equality} that $\alpha^I \circ \delta^I = \beta^I \circ \delta^I$, $\gamma^I \circ \delta^I = \alpha^I \circ \epsilon^I$, and $\beta^I \circ \epsilon^I = \gamma^I \circ \epsilon^I$. By the lemma we have that the $R$-submodule $J = \ker(\delta^I - \epsilon^I)$ of $N$ is such that the inclusion map $J \to N$ induces an isomorphism $J \otimes_R S \overset{\sim}{\to} N$.
 
Let $M = (\varphi^I)^{-1}(J) \subseteq R^{\oplus I}$. First of all, because $\varphi: R\to S$ is faithfully, in particular it is injective, hence $\varphi^I$ is injective as well.  Furthermore, given an element of $(n_i)_{i \in I}$ of $J$, we have that $(n_i \otimes 1)_{i \in I} = (1 \otimes n_i)_{i \in I}$ in $(S \otimes_R S)^{\oplus I}$, which implies that for each $i \in I$, $n_i \otimes 1 - 1 \otimes n_i = 0$ in $S \otimes_R S$. Because the Amitsur complex is exact for faithfully flat homomorphisms (Theorem \ref{thm: amitsur complex is exact for faithfully flat homomorphisms}), it follows that for each $i \in I$, there exists $m_i \in R$ such that $\varphi(m_i) = n_i$. Then $\varphi^I( (m_i)_{i \in I} ) = (n_i)_{i\in I}$. It follows that $\varphi^I: M \to J$ is an isomorphism of $R$-modules. All together, we have that $\varphi^I : M \to N$ induces an isomorphism $f: \varphi_! M  \overset{\sim}{\to} N$.

Now on the one hand, $M$ is a submodule of $R^{\oplus I}$ so we get an exact sequence of $R$-modules $0 \to M \overset{i_M}{\to} R^{\oplus I} \to R^{\oplus I} / M \to 0$. Since $S$ is a flat $R$-module, applying $\varphi_!$ yields an exact sequence of $S$-modules. On the other hand, $N$ is a submodule of $S^{\oplus I}$, so we get an exact sequence of $S$-modules $0 \to N \overset{i_N}{\to} S^{\oplus I} \to S^{\oplus I} / N \to 0$. Consider the following diagram of $S$-modules. 
\begin{center}
\begin{tikzcd}
0 \arrow[r] & \varphi_! M \arrow[r, "\varphi_! (i_{M})" ] \arrow[d, "f" ] & \varphi_! R^{\oplus I} \arrow[r, "\varphi_! (\pi_M) "]  \arrow[d, "\Phi_{R, S}(\varphi^I)"] & \varphi_! ( R^{\oplus I} / M ) \arrow[r] \arrow[d, "h"]  & 0
\\ 0 \arrow[r] & N  \arrow[r, "i_N"]  & S^{\oplus I}  \arrow[r, "\pi_N"]  &  S^{\oplus I}  / N \arrow[r]  & 0
\end{tikzcd}
\end{center}
Recall that $\Phi_{R, S}(\varphi^I)$ is the isomorphism $\varphi_! R^{\oplus I}  \overset{\sim}{\to} S^{\oplus I}$ induced by $\varphi^I$, and $h$ will be defined shortly. First of all, we show that this diagram commutes. Given $m \otimes s \in \varphi_! M$, mapping to the right yields $m \otimes s \in \varphi_! R^{\oplus I}$, and mapping that down yields $s \cdot \varphi^I(m) \in S^{\oplus I}$, while mapping down first yields $s \cdot \varphi^I(m) \in N$, and including gives $s \cdot \varphi^I(m) \in S^{\oplus I}$. Since $f$ and $\Phi_{R, S}(\varphi^I)$ are both isomorphisms, basic homological algebra gives the existence of the map $h$, which is then also an isomorphism. 

Now because $N$ is a direct summand of the free module $S^{\oplus I}$, it follows that $S^{\oplus I} / N$ is a projective $S$-module, hence $\varphi_! ( R^{\oplus I} / M )$ is also a projective $S$-module. But by a result of Raynaud-Gruson \cite{raynaud-gruson1971} and Perry \cite{perry2010}, projectivity satisfies faithfully flat descent, hence $R^{\oplus I} / M$ is a projective $R$-module, so $M$ is a direct summand of $R^{\oplus I}$, i.e. $M \in \Gr_I(R)$. Finally we note that $\Gr_I(\varphi)(M)$ is indeed equal to $N$ because we proved that the map $f: \varphi_! M \to N$ is an isomorphism. 

\end{proof}


\section{Submodule Filtrations} \label{sec: submodule filtrations}


The goal of this section is to introduce the very general notions of a poset filtration and an (almost) gradation which are necessary for defining flags and which will also be used as technical tools during our proof of the Bruhat decomposition.


\subsection{Poset Filtrations}

\begin{definition}
A \emph{poset filtration} of an $R$-module $M$ consists of a partially ordered set $(P , \leq )$, and a function $F: P \to \Sub(M)$ satisfying the following.

\begin{enumerate}

\item[(F1)] There exists $p \in P$ such that $F_p = \{ 0 \}$. 

\item[(F2)] There exists $q \in P$ such that $F_q = M$. 

\item[(F3)] $F$ is an order-preserving function (i.e. $p \leq q \implies F_p \subseteq F_q$).  

\end{enumerate}
\end{definition}

A poset filtration $(F , P)$ is called \emph{linear} if $(P , \leq)$ is a totally ordered set. In this case we typically use ``$I$'' or ``$J$'' instead of ``$P$''. Furthermore, we define a poset filtration $(F , P)$ to be an \emph{embedding} if the map $F: P \to \Sub(M)$ is an order embedding (i.e. for any $p , q \in P$, $p \leq q \iff F_p \subseteq F_q$). Note that this condition implies that $F$ is injective, and indeed an order-isomorphism into its image. If $(F , I)$ is linear, then it is an embedding if and only if $F$ is injective. Finally we say $(F , P)$ is \emph{well-founded} if $(P , \leq)$ is a well-founded poset. Well-founded filtrations which are linear are called \emph{well-ordered}.

\begin{example}
Let $R^{\oplus \nn} := \bigoplus_{n \in \nn} R$. This is often known as \emph{finitely supported sequence space}. Furthermore, let $I = \omega + 1$ have the usual total order. We define a filtration $E : I \to \Sub(R^{\oplus \nn})$ as follows. 
\begin{align*}
 E_{\omega} &= R^{\oplus \nn}
\\ E_k &= \{ (a_1, \ldots, a_k, 0 , \ldots ) \}
\\ E_{0} &= 0
\end{align*}
Then $(E , I)$ is a linear, embedded, well-ordered filtration. 
\end{example}

\begin{example}
Define $J = \{ -\infty \} \cup \zz_{\leq 0}$ with the usual linear order. Define $F : J \to \Sub(R^{\oplus \nn})$ as follows.  
\begin{align*}
F_0 & = R^{\oplus \nn}
\\ F_{-k} &= \{ (\underbrace{0 , \ldots , 0}_{k} , a_{k + 1}, \ldots ) \} 
\\ F_{-\infty} &= 0
\end{align*}
Then $(F , J)$ is a linear, embedded filtration of $R^{\oplus \nn}$, but note that it is not well-ordered.
\end{example}

\begin{example}
Let $R^{\nn} := \prod_{n \in \nn} R$. This is often known as ``sequence space'', and strictly contains $R^{\oplus \nn}$. As before we let $J = \{ -\infty \} \cup \zz_{\leq 0}$, and define $G : J \to \Sub(R^{\nn})$ by
\begin{align*}
G_0 & = R^{\nn}
\\ G_{-k} &= \{ (\underbrace{0 , \ldots , 0}_{k} , a_{k + 1}, \ldots ) \} 
\\ G_{-\infty} &= 0
\end{align*}
Note that $(G, J)$ is basically the same as $(F, J)$, the only difference is the underlying module has been enlarged. 
\end{example}


\subsection{Gradations and Almost Gradations}

\begin{definition}
A \emph{gradation} of a poset filtration $(F , P)$ of an $R$-module $M$ is a function $C: P \to \Sub(M)$ satisfying: 
\begin{itemize}

\item[(G1)] For all $p \in P$ we have $F_p = \bigoplus_{q \leq p} C_q$. 

\end{itemize}
\end{definition}

One can weaken (G1) so that the submodules $C_p$ only complement successive submodules of the filtration, resulting in the following definition.  

\begin{definition}
An \emph{almost gradation} of a poset filtration $(F , P)$ of an $R$-module $M$ is a function $C: P \to \Sub(M)$ satisfying: 
\begin{itemize}

\item[(AG1)] For all $p \in P$, $F_p = F_{< p} \oplus C_p$,

\end{itemize}
where we define $F_{<p} = \sum_{q < p} F_q$. 
\end{definition}

As the names suggest and the following lemma and proposition show, every gradation of a poset filtration $(F , P)$ is also an almost gradation of $(F , P)$. 

\begin{lemma}\label{lem:less than direct sum gradation}
Let $C: P \to \Sub(M)$ be a gradation of a poset filtration $(F , P)$ of an $R$-module $M$. Then for all $p \in P$, $F_{<p} = \bigoplus_{q < p} C_q$. 
\end{lemma}

\begin{proof}
Let $p \in P$. By definition we have that $F_p = \bigoplus_{q \leq p} C_q$. Therefore the family of submodules $(C_q : q \leq p)$ is independent. Hence the subfamily $(C_q : q < p)$ is independent as well, so (DS1) is satisfied. To prove (DS2), note that we have $F_{< p} = \sum_{q < p} F_q = \sum_{q < p} \sum_{r \leq q} C_r = \sum_{r < p} C_r$. 
\end{proof}

\begin{proposition}\label{prop: gradation implies almost gradation}
Every gradation of a poset filtration $(F , P)$ is an almost gradation of $(F , P)$. 
\end{proposition}

\begin{proof}
Let $C: P \to \Sub(M)$ be a gradation of $(F , P)$. Let $p \in P$. definition we have that $F_p = \bigoplus_{q \leq p} C_q$. Since $\{ q \leq p \} = \{ q < p \} \sqcup \{ p \}$, it follows that $F_p = \left( \sum_{q < p} C_q \right) \oplus C_p = F_{< p} \oplus C_p$, where the second equality follows from Lemma \ref{lem:less than direct sum gradation}. Hence (AG1) is satisfied. 
\end{proof}


\subsection{Independent Almost Gradations}

Recall the definition of an independent function $C: P \to \Sub(M)$ from the background section. Of course even in the case when $R$ is a field, almost gradations need not be independent in general. However, linearity gives one simple condition on the poset $P$ to guarantee that any almost gradation of any filtration indexed by $P$ will be independent.

\begin{theorem}\label{thm:ag of a linear filt is ind}
Every almost gradation $C: I \to \Sub(M)$ of a linear filtration $(F , I)$ of $M$ is independent. 
\end{theorem}

\begin{proof}
Suppose that $J \subseteq I$ is a finite set and $c_j \in C_j$ are elements such that $\sum_{j \in J} c_j = 0$. Because $I$ is linearly ordered, we let $J = \{ j_1, \ldots, j_n \}$ with $j_1 < j_2 < \ldots < j_n$. Then we have that $c_{j_1}, \ldots, c_{j_{n - 1}} \in F_{< j_n}$, and by definition, $F_{j_n} = C_{j_n} \oplus F_{< j_n}$. Hence we obtain that $c_{j_n} = 0$ and also $c_{j_1} + \ldots + c_{j_{n - 1}} = 0$. The result follows by induction on $n$.
\end{proof}

In fact we'll now show that the theorem above is also true for the intersection of two linear filtrations, though not for three!

\subsection{Intersections of Poset Filtrations}

Let $(P , \leq_P)$ and $(Q , \leq_Q)$ be posets. We define the \emph{product order} on $P \times Q$, denoted here by $\leq_\times$, by letting $(p , q) \leq_\times (x , y)$ if and only of $p \leq_P x$ and $q \leq_Q y$. From now on we may drop the subscripts on the orders. The \emph{intersection} of two poset filtrations $(E , P)$ and $(F , Q)$ of a $k$-module $M$ consists of the poset $(P \times Q, \leq_\times)$ and the function $E \cap F: P \times Q \to \Sub(M)$ defined by $[E \cap F]_{(p , q)} = E_p \cap F_q$. If $(E , P)$ and $(F , Q)$ are poset filtrations of $M$ then so is their intersection. 

First we prove some technical lemmas. 

\begin{lemma}\label{lem: nested sums and unions}
If $(F , I)$ is a linear filtration of an $R$-module $M$, and $J \subseteq I$ is a nonempty subset, then $\sum_{j \in J} F_j = \bigcup_{j \in J} F_j$. 
\end{lemma}

\begin{proof}
To prove ($\subseteq$), note that since $J$ is nonempty, every element of $\sum_{j \in J} F_j$ is of the form $v_{j_1} + \ldots + v_{j_n}$ with $v_{j_m} \in F_{j_m}$. Because $I$ is totally ordered, we may assume without loss of generality that $j_1 < \ldots < j_n$. Then $F_{j_m} \subseteq F_{j_n}$ for all $1 \leq m \leq n$, hence $v_{j_1} + \ldots + v_{j_n} \in F_{j_n} \subseteq \bigcup_{j \in J} F_j$. Since $F_k \subseteq \sum_{j \in J} F_j$ for all $k \in J$, the containment ($\supseteq$) follows trivially. 
\end{proof}

\begin{lemma}
If $(E , I)$ and $(F , J)$ are linear filtrations of an $R$-module $M$, then for each $(i , j) \in I \times J$, we have that 
\begin{equation*}
[E \cap F]_{< (i , j)} = E_i \cap F_{< j} + E_{< i} \cap F_j. 
\end{equation*}
\end{lemma}

\begin{proof}
By definition, we have that $[E \cap F]_{< (i , j)} = \sum_{(k , \ell) < (i , j)} E_k \cap F_\ell$. We now show mutual inclusion of the desired equality. 

For the containment ($\subseteq$), suppose that $(k , \ell ) < (i , j)$, and note that in the product order this holds if and only if $k < i$ and $\ell \leq j$, or $k \leq i$ and $\ell < j$. Without loss of generality, suppose that $k < i$ and $\ell \leq j$. Then we have that $E_k \subseteq E_{<i}$, and $F_\ell \subseteq F_j$. Therefore $E_k \cap F_\ell \subseteq E_{< i} \cap F_j$.

To show the containment ($\supseteq$), we prove only that $E_i \cap F_{< j} \subseteq [E \cap F]_{< (i , j)}$, the other case being similar. First of all, if $j$ is the smallest element of $J$, then $F_{<j} = 0$ and so the containment holds trivially. If not, given $v \in E_i \cap F_{< j}$, in particular we have that $v \in F_{< j} = \sum_{\ell < j} F_\ell$. Since $(J , \leq)$ is totally ordered and $\{ \ell \in J \mid \ell < j \}$ is nonempty by hypothesis, by Lemma \ref{lem: nested sums and unions} we have that $\sum_{\ell < j} F_\ell = \bigcup_{\ell < j} F_\ell$. Therefore there exists some $\ell' < j$ such that $v \in F_{\ell '}$. Note that this implies $(i , \ell') < (i , j)$, hence $v \in E_i \cap F_{\ell '} \subseteq \sum_{(k , \ell) < (i , j)} E_i \cap F_j = [E \cap F]_{< (i , j)}$. 
\end{proof}

Now we prove the theorem. 

\begin{theorem}\label{thm:ag of the intersection of two linear filts is ind}
Let $(E , I)$ and $(F , J)$ be linear filtrations of an $R$-module $M$. Every almost gradation of the intersection $(E \cap F , I \times J)$ is independent.  
\end{theorem}

\begin{proof}
Let $C : I \times J \to \Sub(V)$ be an almost gradation of $(E \cap F , I \times J)$. We prove by induction on $m$ and $n$ that if $i_1 < \ldots < i_m \in I$ and $j_1 < \ldots < j_n \in J$ and for $1 \leq k \leq m$, $1 \leq \ell \leq n$ we have $c_{k , \ell} \in C_{(i_k , j_\ell )}$ such that $\sum_{\substack{1 \leq k \leq m \\ 1 \leq \ell \leq n}} c_{k , \ell} = 0$, then $c_{k , \ell} = 0$ for all $1 \leq k \leq m$ and $1 \leq \ell \leq n$.

If $m = n = 1$, then by hypothesis we have that $c_{1 , 1} = 0$. Now suppose that the result holds for $m = 1$ and for $n - 1$. Given $c_{1 , 1} + \ldots + c_{1 , n} = 0$, this implies that $c_{1 , 1} + \ldots + c_{1 , n - 1} = - c_{1 , n}$. 
Note that $(i_1, j_1), \ldots, (i_1 , j_{n - 1} ) < (i_1, j_n)$, and hence for $1 \leq \ell \leq n - 1$, 
\begin{equation*}
c_{1 , \ell} \in C_{(i_1 , j_\ell )} \subseteq [E \cap F]_{(i_1 , j_\ell )} \subseteq [E \cap F]_{ < ( i_1, j_n ) }.
\end{equation*}
Therefore $c_{1 , 1} + \ldots + c_{1 , n - 1} \in [E \cap F]_{ < ( i_1, j_n ) }$, and since $c_{1 , n} \in C_{(i_1, j_n)}$, and by definition we have that $[E \cap F]_{ ( i_1, j_n ) } = [E \cap F]_{ < ( i_1, j_n ) } \oplus C_{(i_1, j_n)}$, it follows that $c_{1 , n} = 0$ and that $c_{1 , 1} + \ldots + c_{1 , n - 1} = 0$. By our induction hypothesis, we then have that $c_{1 , \ell} = 0$ for all $1 \leq \ell \leq n - 1$ as well. Hence we have shown that the result holds for $m = 1$ and for arbitrary $n$. 

Now suppose the result holds for $m - 1$ and arbitrary $n$. Given $c_{k , \ell}$ such that $\sum_{\substack{1 \leq k \leq m \\ 1 \leq \ell \leq n}} c_{k , \ell} = 0$, we'll prove that $c_{m , p} = 0$ for $1 \leq p \leq n$ by induction on $r = n - p$. If $r = 0$, then $p = n$. First of all, we have that $\sum_{\substack{1 \leq k \leq m \\ 1 \leq \ell \leq n}} c_{k , \ell} = 0$, which we rewrite as
\begin{equation*}
\sum_{\substack{1 \leq k \leq m - 1 \\ 1 \leq \ell \leq n}} c_{k , \ell} + \sum_{1 \leq \ell \leq n - 1} c_{m , \ell} = - c_{m , n}.
\end{equation*}
However, we note that the left hand side of this equation is contained in $[E \cap F]_{ < ( i_m, j_n ) }$. Since $c_{m , n} \in C_{(i_m , j_n )}$ and by definition we have that $[E \cap F]_{( i_m, j_n ) } =  [E \cap F]_{ < ( i_m, j_n ) } \oplus C_{(i_m , j_n )}$, it follows that $c_{m , n} = 0$. 

Now suppose that $c_{m , n - q} = 0$ for $0 \leq q \leq r - 1$. We'll show that $c_{m , n - r} = c_{m , p} = 0$ as well. By hypothesis, we have that $\sum_{\substack{1 \leq k \leq m - 1 \\ 1 \leq \ell \leq n}} c_{k , \ell} + \sum_{1 \leq \ell \leq p} c_{m , \ell} = 0$ which we rewrite as 
\begin{equation*}
\sum_{\substack{1 \leq k \leq m - 1 \\ 1 \leq \ell \leq n}} c_{k , \ell} = - \sum_{1 \leq \ell \leq p} c_{m , \ell}. 
\end{equation*}
We note that the left hand side of this equation is contained in $E_{i_{m - 1}}$, and the right hand side is contained in $F_{j_p}$. As the two sides are equal, both sums are contained in $E_{i_{m - 1}} \cap F_{j_p}$, so in particular, the left hand side is contained in $[E \cap F]_{< (i_{m} , j_p)}$. Furthermore, leaving off the $c_{m , p}$ term, we also have that the sum $\sum_{1 \leq \ell \leq p - 1} c_{m , \ell}$ is contained in $[E \cap F]_{< (i_{m} , j_p)}$. But by hypothesis, we have the equation 
\begin{equation*}
\sum_{\substack{1 \leq k \leq m - 1 \\ 1 \leq \ell \leq n}} c_{k , \ell} + \sum_{1 \leq \ell \leq p - 1} c_{m , \ell} = - c_{m , p}.
\end{equation*}
We have just shown that the left hand side is contained in $[E \cap F]_{< (i_{m} , j_p)}$, while the RHS is contained in $C_{(i_m , j_p)}$. Since by definition we have that $[E \cap F]_{( i_m, j_p ) } =  [E \cap F]_{ < ( i_m, j_p ) } \oplus C_{(i_m , j_p )}$, it follows that $c_{m , p} = 0$. 
\end{proof}


\subsection{Spanning Almost Gradations}

Let $(F , P)$ be a poset filtration of an $R$-module $M$ and $C: P \to \Sub(M)$ a function. If for all $p \in P$ we have that $F_p = \sum_{q \leq p} C_q$, then we say $C$ \emph{spans} $(F , P)$. Clearly an almost gradation of $(F , P)$ is a gradation if and only if it is independent and spans $(F , P)$.

If $M$ is a finite-dimensional vector space, then obviously all almost gradations of $F$ span. However, perhaps surprisingly, this is not true in infinite-dimensions (even for vector spaces). In the latter case we have three possibilities, even in the linear case: (1) No almost gradations of $F$ are gradations, (2) Some almost gradations of $F$ are gradations and some are not, and (3) all almost gradations of $F$ are gradations. We now give examples of (1) and (2), and for simplicity we work over a field $\ff$. Note that an equivalent example of (1) can be found in \cite{bautista-liu-paquette2011}. 

\begin{example}[A Linear Filtration that does not have a Gradation]
Define $V = \ff^\nn = \{ (a_1, a_2, \ldots ) : a_i \in \ff \}$, and for $i \in \zz_{\leq 0}$ define $F_i = \{ (\underbrace{0 , \ldots, 0}_{-i} , a_{-i + 1}, \ldots ) \}$, and $F_{-\infty} = \{ 0 \}$, and $I = \zz_{\leq 0 } \cup \{ -\infty \}$ with the usual order. Then $(F , I)$ is a maximal embedded linear filtration of $V$. For any almost gradation $C: I \to \Sub(V)$, we have that $\dim(C_i) = 1$ for all $i \in \zz_{\leq 0}$, and $\dim(C_{-\infty}) = 0$. Therefore $\sum_{i \in I} C_i$ has countable dimension, whereas $\ff^\nn$ has uncountable dimension. Hence $\sum_{i \leq 0} C_i \subsetneq F_0$, so $C$ cannot be a gradation. 
\end{example}

\begin{example}[An Almost Gradation that is not a Gradation]
Define $V = \ff^{\oplus \nn} = \{ (a_1, a_2, \ldots ) \in \ff^\nn : a_i = 0 \text{ for all but finitely many } i \}$ and similarly to the example above for $i \in \zz_{\leq 0}$ define $F_i = \{ (\underbrace{0 , \ldots, 0}_{-i} , a_{-i + 1}, \ldots ) \}$, and $F_{-\infty} = \{ 0 \}$. Define $I = \zz_{\leq 0 } \cup \{ -\infty \}$ with the usual order. Then $(F , I)$ is a maximal embedded linear filtration of $V$. For each $k \in \zz_{>0}$, define $e_{k}$ to be the $k$th standard basis vector, and let $f_k = e_k - e_{k + 1}$. 

Firstly, for $i \in \zz_{\leq 0}$, define $C_{i} = \Span(f_{-i + 1})$ and $C_{-\infty} = \{ 0 \}$. For each $i \in \zz_{\leq 0}$, we have that 
\begin{align*}
F_i = F_{i - 1} \oplus C_i,
\end{align*}
hence $C: I \to \Sub(V)$ is an almost gradation of $(F , I)$. However for all $i \in I$, $C_i$ is contained in the kernel of the linear functional $\varphi: V \to \ff$ defined by $\varphi(a_1, a_2, \ldots ) = a_1 + a_2 + \ldots$, which has codimension $1$ inside $V$. Therefore $F_0 \neq \sum_{i \leq 0} S_i$, and hence the almost gradation does not span $V$ and is therefore not a gradation. 

However, secondly, if for $i \in \zz_{\leq 0}$, we define $C'_{i} = \Span(e_{-i + 1})$ and $C'_{-\infty} = \{ 0 \}$, then $C': I \to \Sub(V)$ is a gradation. 
\end{example} 

Hence filtrations satisfying (1) and (2) exist. We now give a characterization of those filtrations satisfying (3) in terms of the well-founded property of the poset $P$. 

\begin{theorem}\label{thm:ag of a well-founded poset filt spnas}
If $C: P \to \Sub(V)$ is an almost gradation of a well-founded poset filtration $(F, P)$ of an $R$-module $M$, then $C$ spans $(F , P)$.  
\end{theorem}

\begin{proof}
Let $T = \{ p \in P : F_p = \sum_{q \leq p} C_q \}$, and suppose for some $n \in P$ we have that $m \in T$ for all $m < n$. By definition of an almost gradation, we have $F_n = F_{<n} \oplus C_n$. Then we compute 
\begin{equation*}
F_n = F_{< n} + C_n = \left[ \sum_{p < n} F_p \right] + C_n = \left[ \sum_{p < n} \sum_{q \leq p} C_q \right] + C_n = \left[ \sum_{q < n} C_q \right] + C_n = \sum_{q \leq n} C_q. 
\end{equation*}
Therefore $n \in T$ as well. By transfinite induction we have that $T = P$. 
\end{proof}

\begin{corollary}\label{cor:every almost gradation of a well-ordered or product of well-ordered filtration is a gradation}
Every almost gradation of a linear well-ordered filtration $(F, I)$ is in fact a gradation. Similarly, if $(E , I)$ and $(F , J)$ are linear well-ordered filtrations of $V$, then every almost gradation of their intersection $(E \cap F , I \times J)$ is a gradation. 
\end{corollary}

\begin{proof}
If $C$ is an almost gradation of $(F , I)$, then because $(F , I)$ is linear, by Theorem \ref{thm:ag of a linear filt is ind} $C$ is independent, and because $(F , I)$ is well-ordered, by Theorem \ref{thm:ag of a well-founded poset filt spnas}, $C$ spans $(F , I)$. Therefore $C$ is a gradation of $(F , I)$. 

If $C$ is an almost gradation of $(E \cap F , I \times J)$, then by Theorem \ref{thm:ag of the intersection of two linear filts is ind}, $C$ is independent, and because $I$ and $J$ are well-ordered, by Lemma \ref{lem: product of well-founded posets is well-founded}, $I \times J$ is well-founded, hence $(E \cap F , I \times J)$ is a well-ordered poset filtration. By Theorem \ref{thm:ag of a well-founded poset filt spnas}, $C$ spans $(E \cap F , I \times J)$, and hence $C$ is a gradation of $(E \cap F , I \times J)$. 
\end{proof}

In the case when $I$ and $J$ are finite subsets of $\nn$ (and hence automatically well-ordered), a proof of Corollary \ref{cor:every almost gradation of a well-ordered or product of well-ordered filtration is a gradation} can be found in \cite{ringel2016}.

\begin{lemma}\label{lem: product of well-founded posets is well-founded}
Let $(P, \leq_P)$ and $(Q, \leq_Q)$ be well-founded posets. Then $(P \times Q, \leq_{\times})$ is a well-founded poset. 
\end{lemma}

\begin{proof}
Let $S$ be a nonempty subset of $P \times Q$. We denote by $\pi_P : P \times Q \to P$ and $\pi_Q: P \times Q \to Q$ the projections onto the first and second coordinates respectively. Then $\pi_P(S)$ is a nonempty subset of $P$, so it has a minimal element, call it $p_0$. Then $\pi_Q( \pi_P^{-1}(p_0) \cap S )$ is a nonempty subset of $Q$, and therefore has a minimal element as well, call it $q_0$. Suppose that $(p , q) \in S$ were such that $(p , q) \leq_{P \times Q} (p_0 , q_0)$. Then we have that $p \leq p_0$ and that $p \in \pi_P(S)$, hence by minimality of $p_0$, it follows that $p = p_0$. Then $q \in \pi_Q( \pi_P^{-1}(p_0) \cap S )$, and we also have that $q \leq q_0$. Hence by minimality of $q_0$, we have $q = q_0$. Therefore $(p , q) = (p_0 , q_0)$, so $(p_0 , q_0)$ is a minimal element of $S$. 
\end{proof}


\subsection{Embedded Filtrations}

A poset filtration $(F , P)$ of an $R$-module $M$ is called \emph{embedded} if $F : P \to \Sub(M)$ is an order embedding. Note that a linear filtration $(F , I)$ is embedded if and only if $F$ is injective. 

\begin{proposition}\label{prop: embedded linear filtrations have least and greatest elements}
If $(F, I)$ is an embedded linear filtration of an $R$-module $M$, then $I$ has a least element $i_0$ and a greatest element $i_\infty$ and we have $F_{i_0} = 0$, $F_{i_\infty} = M$. 
\end{proposition}

\begin{proof}
By (F1) there exists $i_0 \in I$ such that $F_{i_0} = 0$. We claim that $i_0$ is the least element of $I$. If not, there exists $i \in I$ such that $i < i_0$, and by (F3) we must have $F_i = 0$ as well. However embedded filtrations are injective, hence $i = i_0$. This implies that $i_0$ is a minimal element, but since $I$ is totally ordered, it is in fact the least element. A similar argument shows that there exists a greatest element $i_\infty$ in $I$ with $F_{i_\infty} = M$. 
\end{proof}


\section{Flags}\label{sec: flags}


The goal of this section is to develop a theory of full flags in an infinite dimensional vector space and more generally in an infinite rank free module. Let $R$ be a ring, let $M$ be a \textbf{free} $R$-module, and let $I$ be a well-ordered set.

\begin{definition}
An \emph{$I$-flag} in $M$ is a well-ordered linear filtration $F : I \to \Sub(M)$ satisfying the following conditions. 

\begin{enumerate}

\item[(FL1)] For each $i \in I$, $F_i$ is a direct summand of $M$, i.e. the image of $F$ lies in $\Sum(M)$. 

\item[(FL2)] For each $i \in I$, we have $F_{i + 1} / F_i$ is finitely generated projective of rank $1$. 

\item[(FL3)] For each upper limit element $i \in I$ we have $F_{<i} = F_i$.

\end{enumerate}
\end{definition}

Note that (FL1) implies that for each $i \in I$, both $F_i$ and $M / F_i$ are projective. The following lemma shows that this implies the quotient $F_{i +1} / F_i$ is also automatically projective, hence the condition (FL2) simply imposes that this module is finitely generated and has projective rank $1$. 

\begin{lemma}\label{lem: nested complemented submodules are locally complemented}
Let $N \subseteq L$ be submodules of a free $R$-module $M$, with the property that $M / N$ and $M / L$ are projective. Then $L / N$ is projective as well, and in fact $N$ is complemented inside $L$, i.e. there exists a (necessarily projective) $R$-submodule $C \subseteq M$ such that $N \oplus C = L$. 
\end{lemma}

\begin{proof}
Because $M / L$ is projective, $L$ is complemented inside $M$, i.e. there exists a submodule $K \subseteq M$ such that $L \oplus K = M$. Since $N \subseteq L$, we have inclusions $L / N \to M / N$ and $K \to M / N$, and it follows that $L / N \oplus K = M / N$. But $M / N$ is projective and direct summands of projective modules are projective, hence $L / N$ is projective. This implies that $N$ is complemented inside $L$, say by $C$. Finally we have that $C \cong L / N$, and we have shown that the latter is projective, so $C$ is projective as well.  
\end{proof}

We now collect a few properties of $I$-flags that we will require later. 

\begin{corollary}\label{cor: I-flags have gradations}
Let $R$ be a ring, $M$ a free $R$-module. Every $I$-flag of $M$ has a gradation. 
\end{corollary}

\begin{proof}
If $F$ is an $I$-flag of $M$, then we define $C : I \to \Sub(M)$ as follows. If $i$ is an upper limit element then we define $C_i = 0$, and since in this case $F_{<i} F_i$, indeed we have $F_i = F_{<i} \oplus C_i$. If $i$ is a successor, say $i = j + 1$, $F_i / F_j$ is projective hence there exists a submodule $C_i \subseteq F_i$ such that $F_i = F_j \oplus C_i$. Since $j$ is the predecessor of $i$, we have that $F_j= F_{<i}$, so $F_i = F_{<i} \oplus C_i$. Then we have shown that $C$ is an almost gradation of $F$. By Corollary \ref{cor:every almost gradation of a well-ordered or product of well-ordered filtration is a gradation}, it follows that $C$ is a gradation of $F$. 
\end{proof}

\begin{lemma}\label{lem: flags are embedded}
$I$-flags are embedded.
\end{lemma}

\begin{proof}
Suppose that $F : I \to \Sub(M)$ is an $I$-flag in $M$. Since $I$ is totally ordered, to show $F$ is an order embedding, it is enough to show that it is injective. Given $i \neq j \in I$, without loss of generality we can assume that $i < j$. Then because $I$ is well-ordered, and $i$ is not the maximal element, $i$ has an immediate successor $k$. Since $F$ is order-preserving, we have that $F_i \subseteq F_k \subseteq F_j$ and $F_k / F_i$ has projective rank $1$ so in particular it is non-zero. Therefore $F_i$ is properly contained in $F_k$, hence in $F_j$ as desired. 
\end{proof}

Together with Proposition \ref{prop: embedded linear filtrations have least and greatest elements}, this lemma shows the following.

\begin{corollary}\label{cor: top and bottom elements of a flag}
If $F$ is an $I$-flag of $M$, then $I$ has a least element $i_0$ and a greatest element $i_\infty$ and we have $F_{i_0} = 0$, $F_{i_\infty} = M$.
\end{corollary}


\section{The Full Flag Functor}\label{sec: full flag functor}

\subsection{Definition}


Let $k$ be a ring and $I$ a well-ordered set. We define a functor $\Fl_I : \bAlg_k \to \bSet$, by sending a $k$-algebra $R$ to the set of $I$-flags in $R^{\oplus I_s}$, and a $k$-module homomorphism $\varphi: R \to S$ to the function which takes an $I$-flag $F : I \to \Sum(R^{\oplus I_s})$ to the composition $\Gr_{I_s}(\varphi) \circ F : I \to \Sum(S^{\oplus I_s})$. We now confirm that $\Gr_{I_s}(\varphi) \circ F$ is indeed an $I$-flag of $S^{\oplus I_s}$.


\begin{lemma}
If $F : I \to \Sum(R^{\oplus I_s})$ is a filtraion in $R^{\oplus I_s}$ and $\varphi : R \to S$ is a $k$-algebra homomorphism, then the composition $\Gr_{I_s}(\varphi) \circ F : I \to \Sum(S^{\oplus I_s})$ is a filtration of $S^{\oplus I_s}$. 
\end{lemma}

\begin{proof}
We first show (F1). By hypothesis there exists $i \in I$ such that $F_i = 0$, and we have that $\Gr_{I_s}(\varphi) F_i \cong F_i \otimes_R S = \{ 0 \} \otimes_R S = 0$. Similarly, to show (F2), we note that by hypothesis there exists $j \in I$ such that $F_j = R^{\oplus I}$, and we have that $\Gr_{I_s}(\varphi) (R^{\oplus I_s}) = S^{\oplus I_s}$. For (F3), if $i \leq j$, then $F_i \subseteq F_j$, so we consider the following commutative diagram in $\bMod_R$ where all of the arrows are inclusion maps, along with its image under the functor $\varphi_!$ which is a commutative diagram in $\bMod_S$.

\begin{center}
\begin{tikzcd}
& F_j \arrow[dd, "i_{F_j}" ] &&& & \varphi_! F_j \arrow[dd, "\varphi_!(i_{F_j})" ] 
\\ F_i \arrow[ru ]   \arrow[rd, "i_{F_i}" ']    &  &&& \varphi_! F_i \arrow[ru ]   \arrow[rd, "\varphi_!( i_{F_i} )" ']     &  
\\  &  R^{\oplus I_s} & &&&  \varphi_! R^{\oplus I_s}
\end{tikzcd}
\end{center}

It follows that $\im \varphi_!(i_{F_i}) \subseteq \im \varphi_!(i_{F_j})$, and hence that $\Gr_{I_s}(\varphi)( F_i ) \subseteq \Gr_{I_s}(\varphi) (F_j )$.
\end{proof}


\begin{lemma}\label{lem: gr preserves gradations}
If $F : I \to \Sum(R^{\oplus I_s})$ is a filtration in $R^{\oplus I_s}$, $C: I \to \Sum(R^{\oplus I_s})$ is a gradation of $F$, and $\varphi : R \to S$ is a $k$-algebra homomorphism, then $\Gr_{I_s}(\varphi) \circ C$ is a gradation of the filtration $\Gr_{I_s}(\varphi) \circ F$ of $S^{\oplus I_s}$. 
\end{lemma}

\begin{proof}
Because $C$ is a gradation of $F$, given any $i \in I$, we have that $F_i = \oplus_{j \leq i} C_j$. By Lemma \ref{lem: gr preserves direct sums} it follows that $\Gr_{I_s}(\varphi)(F_i) = \oplus_{j \leq i} \Gr_{I_s}(\varphi)(C_j)$, as desired. 
\end{proof}


\begin{proposition}
If $F : I \to \Sum(R^{\oplus I_s})$ is an $I$-flag in $R^{\oplus I_s}$ and $\varphi : R \to S$ is a $k$-algebra homomorphism, then the composition $\Gr_{I_s}(\varphi) \circ F : I \to \Sum(S^{\oplus I_s})$ is an $I$-flag in $S^{\oplus I_s}$. 
\end{proposition}

\begin{proof}
(FL1) is clearly satisfied because by Lemma \ref{lem: gr preserves direct sums}, $\Gr_{I_s}(\varphi)$ takes direct summands to direct summands. We now prove (FL2). If $i \in I$ is a successor with predecessor $j$, then because $F_i / F_j$ is projective, $F_j$ is a direct summand of $F_i$, say $F_i = F_j \oplus N$. By Lemma \ref{lem: gr preserves direct sums} it follows that $\Gr_{I_s}(\varphi)(F_i) = \Gr_{I_s}(\varphi)(F_j) \oplus \Gr_{I_s}(\varphi)(N)$. Since $N$ is finitely generated projective of rank $1$, it follows from Propositions \ref{prop: extension of scalars preserves finitely generated} and \ref{prop: rank and localizations of projective modules} that $\varphi_! (N) \cong \Gr_{I_s}(\varphi)(N) \cong \Gr_{I_s}(\varphi)(F_i) / \Gr_{I_s}(\varphi)(F_j)$ are all also finitely generated projective of rank $1$. 

For (FL3), let $C : I \to \Sub(R^{\oplus I_s})$ be a gradation of $F$. Then by Lemma \ref{lem: gr preserves gradations}, $\Gr_{I_s}(\varphi) \circ C$ is a gradation of $\Gr_{I_s}(\varphi) \circ F$. For any $i \in I$, we have that $F_{<i} = \bigoplus_{j < i} C_j$ and also $(\Gr_{I_s}(\varphi) \circ F)_{<i} = \bigoplus_{j < i} \Gr_{I_s}(\varphi)(C_j)$. Now if $i$ is an upper limit element, then $F_i = F_{<i}$, hence we compute as follows. 
\begin{equation*}
(\Gr_{I_s}(\varphi) \circ F)_i \overset{(1)}{=} \bigoplus_{j < i} \Gr_{I_s}(\varphi)(C_j) = (\Gr_{I_s}(\varphi) \circ F)_{<i}
\end{equation*}
Here (1) follows from Lemma \ref{lem: gr preserves direct sums}. 

\end{proof}


\subsection{The Finite Flag Functor}
Now let us discuss the case when $I$ is a finite set. Since all well-ordered finite sets are order-isomorphic to $\{0 , \ldots ,n\}$ with the usual order, we assume that $I$ is equal to this set. Then $I_s = \{ 1,\ldots, n\}$ and $\Fl_I$ is isomorphic to the usual full flag functor $\Fl_n$ which sends a $k$-algebra $R$ to the set of sequences of submodules $0 \subseteq M_1 \subseteq M_2 \subseteq \ldots \subseteq M_{n - 1} \subseteq R^{\oplus n}$ where each $M_i$ is a summand of $R^{\oplus n}$ and is finitely generated projective of rank $i$ (see \cite{demazure-gabriel1980} I.2.6.5 and \cite{karpenko2000}). In this case $\Fl_n$ is represented by a scheme, and hence is automatically an fppf-sheaf. However in the case when $I$ is infinite, the proof of representability (again) does not extend, and therefore (again) we spend the remainder of this section proving the sheaf condition directly.


\subsection{The Full Flag Functor is an fppf-Sheaf}
We begin with some technical lemmas. 

\begin{lemma}\label{lem: change of rings and products}
Let $R_1,\ldots, R_n$ be rings and let $\pi_i : R_1 \times \ldots \times R_n \to R_i$ denote the projection map. For each $i$ let $M_i$ be an $R_i$ module and note that the projection $p_i : M_1 \times \ldots \times M_n \to \pi_i^* M_i$ is an $(R_1 \times \ldots \times R_n)$-module homomorphism. Then $p_i$ is adapted to $\pi_i$. 
\end{lemma}

\begin{proof}
Given $m \in M_i$, we have that $(\pi_i)_! ( (0 , \ldots, m , \ldots , 0) \otimes 1) = m$, so $(\pi_i)_!$ is surjective. On the other hand, consider the map $g : M_i \to (\pi_i)_! (M_1 \times \ldots \times M_n)$ which sends $m \mapsto (0 , \ldots, m , \ldots , 0) \otimes 1$. An easy computation shows that this map is $R_i$-linear.
Furthermore another easy computation shows that 
$g \circ (\pi_i)_! = 1_{(\pi_i)_!(M_1 \times \ldots \times M_n)}$. Hence $(\pi_i)_!$ has a left inverse and is therefore injective. 
\end{proof}


\begin{lemma}\label{lem: nested union as a colimit}
Let $R$ be a ring, $M$ an $R$-module, $I$ a totally ordered set, $\mE : \bThin(I) \to \bMod_R$ a functor, and $\eta  : \mE \to \Delta(M)$ a natural transformation. Suppose that for every $j$ in $I$, $\eta_j$ is injective. Then $(\sum_{i \in I}\eta_i(\mE(i)) , \nu)$ is a colimit of $\mE$, where $\nu: \mE \to \Delta( \sum_{i \in I}\eta_i(\mE(i)) )$ is defined by letting $\nu_i$ be the factorization of $\eta_i$ through $\sum_{i \in I}\eta_i(\mE(i))$. 
\end{lemma}

\begin{proof}
First of all, $(\sum_{i \in I}\eta_i(\mE(i)) , \nu)$ is a cocone over $\mE$ because $(M , \eta)$ is. Now given $j , k \in I$, since $I$ is totally ordered without loss of generality we may assume that $j \leq k$. Then there exists a morphism $i_{j , k} : j \to k$ in $I$, and by hypothesis $\mE(i_{j , k})$ is injective. Since $\eta_j$ and $\eta_k$ are injective as well, we have that $\eta_j( \mE(j) ) \subseteq \eta_k(\mE(k))$, hence the sum $\sum_{i \in I}\eta_i(\mE(i))$ is nested, so it is in fact equal to $\bigcup_{i \in I} \eta_i(\mE(i))$. Now suppose that $N$ is an $R$-module and $\mu : \mE \to N$ is a cocone over $\mu$. We define a morphism $f : \bigcup_{i \in I}\eta_i(\mE(i)) \to N$ as follows. If $x \in \bigcup_{i \in I}\eta_i(\mE(i))$, then there exists some $j \in I$ such that $x \in \eta_j(\mE(j))$, and hence there exists $y \in \mE(j)$ with $\eta_j(y) = x$. We define $f(x) = \mu_j(y)$. 

This is well-defined because if $x \in \eta_k(\mE(k))$ with $\eta_k(z) = x$ as well, then because $I$ is totally ordered, without loss of generality we have that $j \leq k$, and $(M , \eta)$ is a cocone over $\mE$, we have that 
\begin{equation*}
\eta_k(z) = \eta_j(y) = \eta_k(i_{j , k}(y)). 
\end{equation*}
But because $\eta_k$ is injective, it follows that $z = i_{j , k}(y)$, and hence because $( N , \mu)$ is also a cocone over $\mE$, we have that $\mu_k (z) = \mu_j (y)$. 

For any $j \in I$ and $y \in \mE(j)$, we have by definition that $f (\eta_j(y) ) = \mu_j(y)$, so that $\Delta(f) \circ \eta = \mu$. Finally $f:  \bigcup_{i \in I}\eta_i(\mE(i)) \to N$ is unique since if $g$ were also such a map, and $x \in \bigcup_{i \in I}\eta_i(\mE(i))$, then $x \in \eta_j(\mE(j))$ for some $j \in I$ and so there exists $y \in \mE(j)$ with $\eta_j(y) = x$. But then $g(x) = g(\eta_j(y)) = \mu_j(y) = f(\eta_j(y)) = f(x)$. 
\end{proof}


\begin{theorem}
The functor $\Fl_I$ is an fppf sheaf. 
\end{theorem}

\begin{proof}
We again use Theorem \ref{thm: fppf sheaf characterization} and we begin with (FP1). Given $k$-algebras $R_1, \ldots,  R_n$, we need to show that the map $\Fl_I( R_1 \times \ldots \times R_n ) \to \Fl_I( R_1) \times \ldots \times \Fl_I(R_n )$ is a bijection. To show injectivity, suppose that $F , F' \in \Fl_I( R_1 \times \ldots \times R_n )$ are such that $\Fl_I( \pi_j )(F) = \Fl_I( \pi_j )(F')$ for $1 \leq j \leq n$. By definition this means that $\Gr_{I_s}(\pi_j) \circ F = \Gr_{I_s}(\pi_j) \circ F'$, hence for all $i \in I$ we have that $\Gr_{I_s}(\pi_j)(F_i) = \Gr_{I_s}(\pi_j)(F'_i)$. However by our proof that $\Gr_{I_s}$ is an fppf sheaf (Theorem \ref{thm: grassmannian is an fppf sheaf}), it follows that $F_i = F'_i$, hence $F = F'$. To show surjectivity, suppose we are given $I$-flags $F^j \in \Fl_I(R_j)$ for $1 \leq j \leq n$, and define $F : I \to \Sub( R_1^{\oplus I_s} \times \ldots \times R_n^{\oplus I_s})$ by $F_i = F_i^1 \times \ldots \times F_i^n$. We claim that $F$ is an $I$-flag.

(F1) is satisfied because by Corollary \ref{cor: top and bottom elements of a flag}, there exists a least element $i_0$ of $I$ and $F^j_{i_0} = 0$ for all $j$. Therefore $F_{i_0} = 0$ as well. Similarly, (F2) is satisfied because by Corollary \ref{cor: top and bottom elements of a flag}, there exists a greatest element $i_\infty$ of $I$ and $F^j_{i_\infty} = R_j^{\oplus I_s}$ for all $j$. Therefore $F_{i_\infty} = R_1^{\oplus I_s} \times \ldots \times R_n^{\oplus I_s}$. (F3) is satisfied since if $i < i'$ then for each $1 \leq j \leq n$ we have $F^j_i \subseteq F^j_{i'}$, and hence 
\begin{equation*}
F_i = F_i^1 \times \ldots \times F_i^n \subseteq F_{i'}^1 \times \ldots \times F_{i'}^n = F_{i'}. 
\end{equation*}

Therefore $F$ is a filtration. (FL1) is satisfied because as in the proof of Theorem \ref{thm: grassmannian is an fppf sheaf}, since each $F_i^j$ is a direct summand of $R_j^{\oplus I_s}$, it follows that $F_i^1 \times \ldots \times F_i^n$ is a direct summand of $R_1^{\oplus I_s} \times \ldots \times R_n^{\oplus I_s}$. To show that (FL2) holds, suppose that $i \in I$ is a successor with predecessor $k$. We compute:  
\begin{equation*}
F_i / F_k = (F_i^1 \times \ldots \times F_i^n) / (F_k^1 \times \ldots \times F_k^n) \cong F_i^1 / F_k^1 \times \ldots \times F_i^n  / F_k^n
\end{equation*}

Given a prime ideal of $R_1 \times \ldots \times R_n$, it is of the form $\pi_j^{-1}(\fp_j) = R_1 \times \ldots \times \fp_j \times \ldots \times R_n$ for some $1 \leq j \leq n$ and some prime ideal $\fp_j \subseteq R_i$. Since $F_i / F_k$ is a projective $(R_1 \times \ldots \times R_n)$-module, so is $F_i^1 / F_k^1 \times \ldots \times F_i^n  / F_k^n$, and the rank of the free $(R_1 \times \ldots \times R_n)_{\pi_j^{-1}(\fp_j)}$-module $(F_i / F_k)_{\pi_j^{-1}(\fp_j)}$ is equal to that of $(F_i^1 / F_k^1 \times \ldots \times F_i^n  / F_k^n)_{\pi_j^{-1}(\fp_j)}$. By Proposition \ref{prop: rank and localizations of projective modules} applied to the homomorphism $\pi_j : R_1 \times \ldots \times R_n \to R_j$, we have that the rank of the free $(R_1 \times \ldots \times R_n)_{\pi_j^{-1}(\fp_j)}$-module $(F_i^1 / F_k^1 \times \ldots \times F_i^n  / F_k^n)_{\pi_j^{-1}(\fp_j)}$ is equal to the rank of the free $(R_j)_{\fp_j}$-module $((\pi_j)_! ( F_i^1 / F_k^1 \times \ldots \times F_i^n  / F_k^n ) )_{\fp_j}$. However, Lemma \ref{lem: change of rings and products} implies that this is isomorphic to $(F_i^j / F_k^j )_{\fp_j}$, which by hypothesis is equal to $1$. 

Finally we prove (FL3). Given any $i \in I$ and any $\ell < i$ we have that $F_\ell^1 \times \ldots \times F_\ell^n \subseteq  \left(\sum_{k < i} F_k^1 \right) \times \ldots \times  \left( \sum_{k < i} F_k^n \right)$ which implies that $\sum_{k < i} F_k^1 \times \ldots \times F_k^n \subseteq \left(\sum_{k < i} F_k^1 \right) \times \ldots \times  \left( \sum_{k < i} F_k^n \right)$. Conversely, we note that $0 \times \ldots \times \left(\sum_{k < i} F_k^j \right) \times \ldots \times 0 = \sum_{k < i} 0 \times \ldots \times F_k^j \times \ldots \times 0 \subseteq \sum_{k < i} F_k^1 \times \ldots \times F_k^n$, which implies that $\sum_{k < i} F_k^1 \times \ldots \times F_k^n \supseteq \left(\sum_{k < i} F_k^1 \right) \times \ldots \times  \left( \sum_{k < i} F_k^n \right)$. 

If $i \in I$ is an upper limit element, then for any $1 \leq j \leq n$ we have $F^j_{< i} = F^j_i$, so we can compute 
\begin{equation*}
F_{<i} = \sum_{k < i} F_k^1 \times \ldots \times F_k^n = \left(\sum_{k < i} F_k^1 \right) \times \ldots \times  \left( \sum_{k < i} F_k^n \right) = F_i^1 \times \ldots \times F_i^n = F_i. 
\end{equation*}

By the proof of Theorem \ref{thm: grassmannian is an fppf sheaf}, if we compose $F$ with the isomorphism $R_1^{\oplus I_s} \times \ldots \times R_n^{\oplus I_s} \overset{\sim}{\to} (R_1 \times \ldots \times R_n)^{\oplus I_s}$, we obtain an $I$-flag $E$ in $(R_1 \times \ldots \times R_n)^{\oplus I_s}$ such that for each $i$, we have $\Gr_{I_s}(\pi_j)(E_i) = F^j_i$, and hence $\Fl_I(\pi_j)(E) = F^j$. 

Now we show (FP2). Suppose that $\varphi: R \to S$ is a $k$-algebra homomorphism which makes $S$ into an fppf-$R$-algebra. We use Proposition \ref{prop: characterization of equalizers in set} to show that the sequence $\Fl_I(R) \to \Fl_I(S) \rightrightarrows \Fl_I(S \otimes_R S)$ is an equalizer diagram in $\bSet$. To prove (EQ1), suppose we are given $F , F' \in \Fl_I(R)$ such that $\Fl_I(\varphi)(F) = \Fl_I(\varphi)(F')$. By definition we have that $\Gr_{I_s}(\varphi) \circ F = \Gr_{I_s}(\varphi) \circ F'$, so for each $i \in I$ we have $\Gr_{I_s}(\varphi) (F_i ) = \Gr_{I_s}(\varphi) ( F'_i )$. By our proof of Theorem \ref{thm: grassmannian is an fppf sheaf}, it follows that $F_i = F'_i$. Since this is true for all $i \in I$, we have that $F = F'$. For (EQ2), suppose that $F \in \Fl_I(S)$ is such that $\Fl_I(\delta)(F) = \Fl_I(\epsilon)(F)$. For each $i \in I$, we let 
\begin{equation*}
E_i = (\varphi^{I_s})^{-1}( \ker( \delta^{I_s} - \epsilon^{I_s} ) \cap F_i )
\end{equation*}
By the proof of Theorem \ref{thm: grassmannian is an fppf sheaf}, we know that $E_i \in \Sum(R^{\oplus I_s})$ and that $\Gr_{I_s}(\varphi)(E_i)  = F_i$. Hence, if $E : I \to \Sum(R^{\oplus I_s})$ is an $I$-flag, then $\Fl_I(E) = F$. 

We now show that $E$ is an $I$-flag. To prove (F1), note that because $F$ is a filtration, there exists $i \in I$ such that $F_i = 0$. Then $E_i = (\varphi^{I_s})^{-1}(0)$. But because $\varphi$ is faithfully flat, by Proposition \ref{prop: equivalent characterizations of faithful flatness}, $\varphi$ is injective, and hence $\varphi^{I_s}$ is injective too. Therefore $E_i = 0$ as desired. For (F2), because $F$ is a filtration, there exists $i \in I$ such that $F_i = S^{\oplus I_s}$. Then $E_i = (\varphi^{I_s})^{-1}( \ker( \delta^{I_s} - \epsilon^{I_s} ))$, but this latter module is actually all of $R^{\oplus I_s}$. For (F3), if $i \leq j$, then we have that $F_i  \subseteq F_j$ hence $\ker( \delta^I - \epsilon^I ) \cap F_i \subseteq \ker( \delta^I - \epsilon^I ) \cap F_j$, which implies that $E_i = (\varphi^{I_s})^{-1}( \ker( \delta^{I_s} - \epsilon^{I_s} ) \cap F_i ) \subseteq (\varphi^{I_s})^{-1}( \ker( \delta^{I_s} - \epsilon^{I_s} ) \cap F_j ) = E_j$. 

That (FL1) holds was shown in the proof of Theorem \ref{thm: grassmannian is an fppf sheaf}. For (FL2), suppose that $i \in I$ is a successor with predecessor $j$. Then $E_j \subseteq E_i$ and both submodules are summands, hence by Lemma \ref{lem: nested complemented submodules are locally complemented}, there exists a summand $C \subseteq R^{\oplus I_s}$ such that $E_j \oplus C= E_i$. By Lemma \ref{lem: gr preserves direct sums} we have that $\Gr_{I_s}(\varphi)(E_i) = \Gr_{I_s}(\varphi)(E_j) \oplus \Gr_{I_s}(\varphi)(C)$, which implies $F_i = F_j \oplus \Gr_{I_s}(\varphi)(C)$. Hence $\varphi_! C \cong \Gr_{I_s}(\varphi)(C) \cong F_i / F_j$ which by hypothesis has constant projective rank $1$. Now because $\varphi$ is faithfully flat, the induced map $\varphi^{-1}(-) : \Spec(S) \to \Spec(R)$ is surjective, hence given any prime ideal $\fp \subseteq R$, there exists a prime ideal $\fq \subseteq S$ such that $\varphi^{-1}(\fq) = \fp$. Since $C$ is projective, by Proposition \ref{prop: rank and localizations of projective modules} we have that the rank of the free $R_\fp$-module $C_\fp$ is equal to the rank of the free $S_\fq$-module $(\varphi_! C)_\fq$ which is $1$. Since $\fp$ was an arbitrary prime ideal of $R$, it follows that $E_i / E_j \cong C$ has constant projective rank $1$. Finally for (FL3), suppose that $i \in I$ is an upper limit element and for any $j < i$, consider the following diagram in $\bMod_R$, with all maps denoting inclusions, as well as its image under $\varphi_!$. 

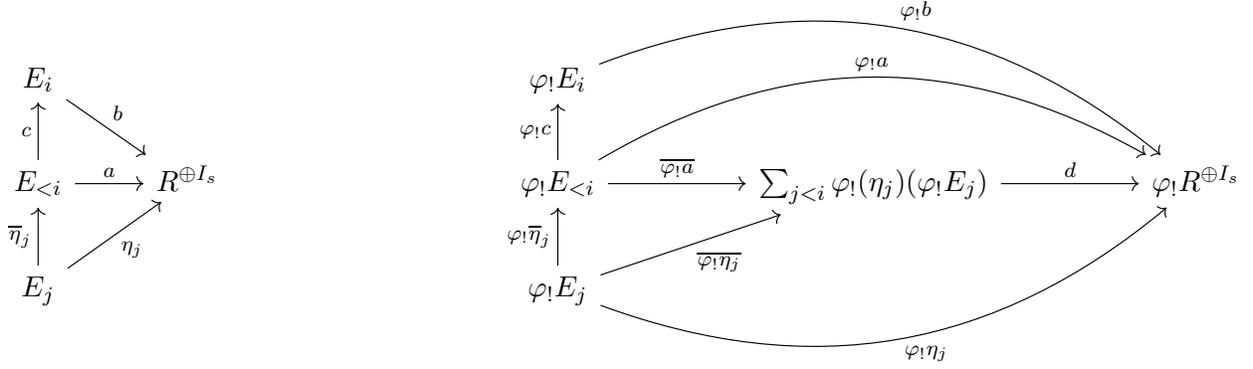
\begin{center}
\begin{tikzcd}
 E_i \arrow[rd, "b" ] &                                                           &&&&    \varphi_! E_i \arrow[rrrrd, bend left=30, "\varphi_! b" ] &    
\\ E_{<i}  \arrow[r, "a"] \arrow[u, "c" ]   &  R^{\oplus I_s}  &&&&  \varphi_! E_{<i}  \arrow[rrrr, bend left=30, "\varphi_! a"] \arrow[u, "\varphi_! c" ] \arrow[rr, "\overline{\varphi_! a}" ]  & & \sum_{j < i} \varphi_! (\eta_j) (\varphi_! E_j) \arrow[rr, "d"] &&  \varphi_! R^{\oplus I_s} 
\\  E_j \arrow[u, "\overline{\eta}_j" ] \arrow[ru, "\eta_j" '] &                      &&&& \varphi_! E_j \arrow[u, "\varphi_! \overline{\eta}_j" ] \arrow[rrrru, bend right=30, "\varphi_! \eta_j" '] \arrow[rru , "\overline{\varphi_! \eta_j}" , swap] &
\end{tikzcd}
\end{center}

Note that the set $I_{<i} = \{ j \in I \mid j < i \}$ is totally ordered. Define a functor $\mE : \bThin(I_{<i}) \to \bMod_R$ which sends $j \mapsto E_j$ and $i_{j ,j'}$ to the inclusion map $E_j \to E_{j'}$. Then the inclusion maps $\eta_j$ and $\overline{\eta}_j$ in the diagram above define natural transformations $\eta : \mE \to \Delta(R^{\oplus I_s})$, and $\overline{\eta} : \mE \to \Delta(E_{<i})$. By Lemma \ref{lem: nested union as a colimit}, $(E_{<i} , \overline{\eta})$ is a colimit of $\mE$. Since $\varphi_!$ is a left-adjoint functor, it preserves colimits, hence $(\varphi_! (E_{<i}) , \varphi_! \overline{\eta})$ is a colimit of $\varphi_! \circ \mE$. 

Given $j < i$, we have that $E_j$ is a summand of $R^{\oplus I_s}$, so by Proposition \ref{prop: extension of scalars applied to direct summands are direct summands}, $\varphi_! \eta_j$ is injective. Again by Lemma \ref{lem: nested union as a colimit}, we have that $(\sum_{j < i} \varphi_! (\eta_j) (\varphi_! E_j) , \overline{\varphi_! \eta})$ is also a colimit of $\varphi_! \circ \mE$. Therefore there exists a unique $S$-module homomorphism $f: \varphi_! (E_{<i}) \to \sum_{j < i} \varphi_! (\eta_j) (\varphi_! E_j)$ such that $\Delta(f) \circ \varphi_! \overline{\eta} = \overline{\varphi_! \eta}$, and in fact $f$ is an isomorphism. Notice that if $a : E_{<i} \to R^{\oplus I_s}$ denotes the inclusion map, then for each $j < i$ we have that $a \circ \overline{\eta}_j = \eta_j$, and therefore $\varphi_! a \circ \varphi_! \overline{\eta}_j = \varphi_! \eta_j$. Hence the image of $\varphi_! a$ is contained in $\sum_{j < i} \varphi_! (\eta_j) (\varphi_! E_j)$, so by restricting the codomain, we obtain $\overline{\varphi_! a} : \varphi_! (E_{<i}) \to \sum_{j < i} \varphi_! (\eta_j) (\varphi_! E_j)$ and $\Delta(\overline{\varphi_! a}) \circ \varphi_! \overline{\eta} = \overline{\varphi_! \eta}$. So $f = \overline{\varphi_! a}$, which implies that $\overline{\varphi_! a}$ is an isomorphism. This implies that $\varphi_! a$ is injective. 

Under the isomorphism $\varphi_! R^{\oplus I_s} \to S^{\oplus I_s}$, $\varphi_! E_i$ maps to $F_i$, and $\sum_{j < i} \varphi_! (\eta_j) (\varphi_! E_j)$ maps to $F_{<i}$, but by hypothesis, these are equal. Hence $\varphi_! a (\varphi_! E_{<i}) = \varphi_! b (\varphi_! E_i)$ it must be that $\varphi_! c$ is an isomorphism. However change of bas by a faithfully flat map reflects isomorphisms, so $c$ is an isomorphism as well, implying that $E_{<i}= E_i$ as desired.

\end{proof}


\section{The $\GL$ Action on the Flag Functor}


\subsection{Matrices}

Let $I$ and $J$ be sets and $R$ be a ring. An $I \times J$ \emph{matrix with entries in $R$} is a function $x : I \times J \to R$. We denote the value of such a function at $(i , j)$ by $x_{i , j}$, and sometimes we denote the matrix $x$ by $[x_{i, j}]$. A matrix $x$ is said to be \textbf{column finite} if for each $j \in J$, the function $x_{- , j} : I \to R$ has finite support.  We denote the set of column finite $I \times J$ matrices by $\Hom_{I, J}(R)$, and in particular the set of $I \times I$ matrices by $\End_{I}(R)$. The latter is a monoid with law of composition given by usual matrix multiplication: 
\begin{equation*}
(xy)_{i , j} = \sum_{n \in I} x_{i n} y_{n j}.
\end{equation*}
We note that because $y$ is column finite, for a fixed $i$ and $j$, only finitely many of the terms $x_{i n} y_{n j}$ are nonzero, hence this sum actually does make sense. 

We get a functor $\Hom_{I , J} : \bCRing \to \bSet$ which sends a ring $R \mapsto \Hom_{I , J}(R)$ and a ring homomorphism $\varphi : R \to S$ to the function sending $[x_{i, j}] \mapsto [\varphi(x_{i, j})]$. Note that indeed $[\varphi(x_{i, j})]$ has finitely supported columns, since $\varphi(0) = 0$. Notice $\Hom_{I , J}$ is a subfunctor of the functor $\Mat_{I , J}$ which maps $R \in \bCRing$ to the set of \emph{all} $I \times J$ matrices. The latter functor is representable, with $\Mat_{I, J} \cong h^{k[x_{i, j} \mid i \in I , j \in J ]}$, but the former is not. However, we now show that $\Hom_{I , J}$ is a filtered colimit of representable functors. 

\begin{proposition} \label{prop: end is an ind-affine ind-scheme}
$\Hom_{I , J}$ is an ind-affine ind-scheme.
\end{proposition}

\begin{proof}
Let $X: J \to \mathcal{P}_\text{fin}(I)$ be a function that assigns to each $j \in J$ a finite subset $X_j \subseteq I$.  Let $C_{I , J}$ denote the set of all such functions. We write $X \subseteq Y$ if $X_j \subseteq Y_j$ for all $j \in J$. Then $(C_{I , J} , \subseteq)$ is a poset and in fact a filtered set because the function $X \cup Y$ is an upper bound for $X$ and $Y$ in $C_{I , J} $. Define a functor $\mD : \bThin(C_{I , J} ) \to [\bCRing , \bSet]$ by sending $X$ to the subfunctor $\Mat_{X}$ of $\Mat_{I, J}$ which assigns to a ring $R$, the set 
\begin{equation*}
\{ x \in \Mat_{I , J}(R) \mid x_{i, j} = 0 \text{ if } i \notin X_j \}, 
\end{equation*}
and by sending the arrow $i_{X , Y} : X \to Y$ (which exists if and only if $X \subseteq Y$) to the inclusion map $\Mat_{X} \hookrightarrow \Mat_{Y}$. Since $X_j$ is finite for all $j \in J$, it follows that $\Mat_{X} \subseteq \Hom_{I , J}$. Furthermore we have that $\Mat_{X} \cong h^{k[x_{i,j} \mid i \in X_j ]}$, i.e. this functor is representable. 

For each function $X \in C_{I , J}$, denote the inclusion map $\Mat_{X} \hookrightarrow \Hom_{I , J}$ by $\eta_{X}$. Then $(\Hom_{I , J} , \eta)$ is a cocone over the diagram $\mD$ and by Proposition \ref{prop: limits and colimits in functor categories} it is a colimit if and only if for every ring $R$, the cocone $(\Hom_{I , J}(R) , (\eta_-)_R)$ over $\Mat_-(R)$ is a colimit in $\bSet$. We use Proposition \ref{prop: filtered colimits of sets} to show that this is in fact the case. First of all, (FC1) is satisfied trivially because for each $X \in C_{I , J}$, $(\eta_X)_R : \Mat_X(R) \to \Hom_{I , J}(R)$ is the inclusion map. Finally (FC2) is satisfied because given any matrix $x \in \Hom_{I , J}(R)$, since $x$ is column finite, the function $X : J \to \mathcal{P}_\text{fin}(I)$ defined by letting $X_j = \{ i \in I \mid x_{i,j} \neq 0 \}$ is in $C_{I , J}$, and $x \in \Mat_X(R)$. 
\end{proof}


\subsection{The General Linear Functor}

We define the \emph{general linear functor} associated to $I$, denoted $\GL_I : \bAlg_k \to \bSet$, by letting $\GL_I(R)$ be the subset of $\End_I(R)$ consisting of the invertible matrices, i.e. column-finite matrices $x$ such that there exists $y \in \End_I(R)$ with $xy = yx = \delta$, where $\delta : I \times I \to R$ is the Kronecker delta function (i.e. the identity matrix). Note that if $\varphi: R \to S$ is a ring homomorphism and $[x_{i,j}]$ is invertible, then $[\varphi(x_{i,j})]$ is also invertible since if $[y_{i,j}] = [x_{i,j}]^{-1}$, then one can easily check that $[\varphi(y_{i,j})] = [\varphi(x_{i,j})]^{-1}$. Hence defining $\GL_I(\varphi) = \End_I(\varphi)$ gives that $\GL_I$ is a subfunctor of $\End_I$. Finally since $\GL_I(R)$ is a group for all rings $R$ and $\GL_I(\varphi)$ is a group homomorphism for all ring homomorphisms $\varphi$, $\GL_I$ is actually a $k$-group functor, i.e. takes values in $\bGrp$.


We define a subfunctor of the $k$-functor $\End_I \times \End_I$ by setting
\begin{equation*}
\mG_I(R) = \{ (x , y)  \in \End_I(R) \times \End_I(R)  \mid xy = yx = \delta \}. 
\end{equation*}


\begin{proposition}\label{prop: closed subfunctor version of GL is an ind-affine ind-scheme}
$\mG_I$ is an ind-affine ind-scheme.
\end{proposition}

\begin{proof}
Using the language of the proof of Proposition \ref{prop: end is an ind-affine ind-scheme}, given functions $X, Y: I \to \mathcal{P}_\text{fin}(I)$ in $C_{I, I}$, we define $\mG_{X , Y} = (\Mat_X \times \Mat_Y) \cap \mG_I$. In other words, $\mG_{X , Y}$ is the subfunctor of $\mG_I$ defined by 
\begin{equation*}
\mG_{X , Y}(R) = \{ (x ,  y) \in \mG_{I}(R) \mid x_{i, j} = 0 \text{ if } i \notin X_j \text{ and } y_{i, j} = 0 \text{ if } i \notin Y_j \}.
\end{equation*}
Then $\mG_{X , Y}$ is a closed subfunctor of the representable functor $\Mat_X \times \Mat_Y \cong h^{k[x_{i , j} , y_{n , m} \mid i \in X_j , n \in Y_m] }$ given by the vanishing of the polynomials
\begin{align*}
&\sum_{\{ \ell \in Y(b) | a \in X(\ell)\} } x_{a \ell} y_{\ell b} & & \text{and} & & \sum_{\{ \ell \in X(b) | a \in Y(\ell)\} } y_{a \ell} x_{\ell b} & & \text{ for all } a , b \in I \text{ with } a \neq b 
\\ &1 - \sum_{\{ \ell \in Y(c) | c \in X(\ell)\}} x_{c \ell} y_{\ell c} & & \text{and} & & 1 - \sum_{\{ \ell \in Y(c) | c \in X(\ell)\}} y_{c \ell} x_{\ell c} & & \text{ for all } c \in I,
\end{align*} 
and therefore $\mG_{X , Y}$ is also representable, indeed $\mG_{X , Y} \cong h^{k[x_{i , j} , y_{n , m} \mid i \in X_j , n \in Y_m]  / \mathfrak{a} }$ where $\mathfrak{a}$ is the ideal of $k[x_{i , j} , y_{n , m} \mid i \in X_j , n \in Y_m]$ generated by the polynomials above. A similar argument to that in the proof of Proposition \ref{prop: end is an ind-affine ind-scheme} shows that $\mG_I$ is a filtered colimit of the representable functors $\mG_{X , Y}$, as desired. 
\end{proof}


\begin{proposition}\label{prop: parts of closed subfunctor version of GL are closed in closed subfunctor version of GL}
For any set $I$ and any pair of functions $X , Y: I \to \mP_\fin(I)$, we have that $\mG_{X, Y}$ is a closed subfunctor of $\mG_I$. 
\end{proposition}

\begin{proof}
Given any $m, n \in I$, there are natural transformations $x_{m, n} : \mG_I \to \aaa^1$ and $y_{m, n} : \mG_I \to \aaa^1$ which send a pair of matrices $(g , h) \in \mG_I(R)$ to $g_{m , n} \in \aaa^1(R)$ and $h_{m , n} \in \aaa^1(R)$ respectively, for any $k$-algebra $R$. Then clearly $\mG_{X , Y} = V(x_{i , j} , y_{m, n} \mid i \notin X_j , m \notin Y_n )$, which by Proposition \ref{prop: properties of closed subfunctors} is a closed subfunctor of $\mG_I$. 
\end{proof}


\begin{proposition}\label{prop: closed subfunctor version of GL is isomorphic to GL}
$\mG_I \cong \GL_I$. 
\end{proposition}

\begin{proof}
The natural transformations $\eta : \mG_I \leftrightarrows \GL_I : \rho$ defined by $\eta_R ( x , y) = x$ and $\rho_R(x) = (x , x^{-1})$ are inverses of each other. 
\end{proof}

Note that composing the inclusion map $\mG_{X , Y} \hookrightarrow \mG_I$ and the isomorphism $\mG_I \overset{\sim}{\to} \GL_I$, yields a monomorphism $\mG_{X , Y} \hookrightarrow \GL_I$ which sends $(x , y) \mapsto x$, the image of which is the subfunctor $\GL_{X , Y}$ of $\GL_I$ given by $R \mapsto \{ x \in \GL_I(R) \mid x \in \Mat_X(R) , x^{-1} \in \Mat_Y(R) \}$. Then $\GL_{X , Y} \cong \mG_{X ,Y}$, so as $\mG_{X, Y}$ is representable and is a closed subfunctor of $\mG_I$, so too $\GL_{X,Y}$ is representable and is a closed subfunctor of $\GL_I$. Combining this with Propositions \ref{prop: closed subfunctor version of GL is an ind-affine ind-scheme}, \ref{prop: closed subfunctor version of GL is isomorphic to GL}, and \ref{prop: ind-schemes are fppf sheaves} we obtain the following corollary.

\begin{corollary}
The functor $\GL_I$ is a filtered colimit of the closed subfunctors $\GL_{X , Y}$, and is therefore an ind-affine ind-scheme, and hence an fppf sheaf. 
\end{corollary}



\subsection{The General Linear Action} 

If $I$ is any set, there is a left action of the general linear $k$-group functor $\GL_I$ on affine space $\aaa_k^I$ given by matrix multiplication on the left. This induces an action of $\GL_I$ on the Grassmannian $k$-functor $\Gr_I$ given by defining for each $k$-algebra $R$, the map $\GL_I(R) \times \Gr_I(R) \to \Gr_I(R)$ which sends $[x_{i , j}] , N \mapsto [x_{i , j}] N$. Similarly, if $I$ is a well-ordered set, then there is an action of the general linear $k$-group functor $\GL_{I_s}$ on the flag $k$-functor $\Fl_I$ given by defining for each $k$-algebra $R$, the map $\GL_{I_s}(R) \times \Fl_I(R) \to \Fl_I(R)$ which sends $[x_{p , q}] , F$ to the $I$-flag $i \mapsto [x_{p , q}] F_i$. 

If $I$ is a well-ordered set, let $E$ denote the standard $I$-flag in $k^{\oplus I_s}$, i.e. we let $E_i = \Span_R(e_j \mid j \in I_s, j \leq i)$ for any $i \in I$. We now describe the stabilizer and orbit of the standard flag under the $\GL_{I_s}$ action.


\begin{proposition}\label{prop: stab of standard flag is upper triangulars}
The stabilizer $\Stab_{\GL_{I_s}}(E)$ is the subgroup functor $B_{I_s} \subseteq \GL_{I_s}$ of all upper triangular matrices, given by $R \mapsto \{ [x_{i , j} ] \in \GL_{I_s}(R) \mid x_{i , j} = 0 \text{ if } i > j\}$. 
\end{proposition}

\begin{proof}
We wish to show that $\Stab_{\GL_{I_s}}(E) = B_{I_s}$, and since both are subfunctors of $\GL_{I_s}$, this amounts to showing that for all $k$-algebras $R$, we have that $\Stab_{\GL_{I_s}}(E)(R) = B_{I_s}(R)$. 

($\subseteq$). Given $[x_{i , j}] \in \Stab_{\GL_{I_s}}(E)(R)$, we have that $[x_{i , j}] E^R = E^R$, i.e. for all $n \in I$ we have $[x_{i , j}] E^R_n = E^R_n$. But in particular this means that $[x_{i , j} ] e_n \in \Span_R(e_m \mid m \in I_s, m \leq n)$. But $[x_{i , j} ] e_n$ is the $n$th column of $[x_{i , j}]$, i.e. the column vector $[x_{i , n}]_{i \in I}$. Since this is in the span of all $e_m$ with $m \leq n$, we have that $x_{i , n} = 0$ for $i > n$. Since this holds for all $n$, it follows that $[x_{i , j}] \in B_{I_s}(R)$. 

($\supseteq$). Given $[x_{i , j}] \in B_{I_s}(R)$ and any $m \leq n \in I$, we have that $[x_{i , j}] e_m = [x_{i , m}]_{i \in I}$, but because $[x_{i , j}]$ is upper triangular, we have that $x_{i , m} = 0$ for all $i > m$, and hence $[x_{i , m}]_{i \in I} \in E^R_m \subseteq E^R_n$. Therefore $[x_{i , j}] E^R_n = E^R_n$, and since $n$ was arbitrary, we have that $[x_{i , j}] E^R = E^R$. 

\end{proof}


\begin{proposition}\label{prop: naive orbit of the standard flag}
The image subfunctor of the natural transformation $\GL_I \to \Fl_I$ defined by sending $g \in \GL_I(R) \mapsto gE^R \in \Fl_I(R)$ for any $k$-algebra $R$ is given by 
\begin{equation*}
R \mapsto \{ F \in \Fl_I(R) \mid F_i / F_{<i} \text{ is free for all } i \in I \}. 
\end{equation*}
\end{proposition}

\begin{proof}
First we prove that if $F = gE^R$, then $F_i / F_{<i}$ is free for all $i \in I$. Note that $E^R$ has a gradation (which could be called the \emph{standard gradation}) defined by $C_i = \Span(e_i)$ if $i \in I_s$ and $C_i = 0$ otherwise. Then because $g : R^{\oplus I_s} \to R^{\oplus I_s}$ is an $R$-module isomorphism we have that $gC$ is a gradation of $g E^R = F$. Hence for any $i \in I$, $F_i/ F_{<i} \cong gC_i$ and the latter is free of rank $1$ when $i \in I_s$ and free of rank $0$ (i.e. is the zero module) otherwise. 

Now we show the converse. Suppose that $F \in \Fl_I(R)$ is such that $F_i / F_{<i}$ is free for all $i \in I$. Note that by the definition of a flag, this quotient is nonzero if and only if $i \in I_s$, and in that case has rank $1$. Let $C$ be a gradation for $F$ so that $C_i \cong F_i / F_{<i}$, and for all $i \in I_s$ let $(b_i)$ be a basis for $C_i$. Then by the definition of a gradation, $(b_i \mid i \in I_s )$ is a basis for $R^{\oplus I_s}$, so if we define $g \in \End_{I_s}(R)$ by letting the $i$th column be the coordinate vector of $b_i$ with respect to the standard basis then $g$ is invertible. Finally, again because of the definition of a gradation, we have for any $i \in I$, $F_i = \Span(b_j \mid j \leq i, j \in I_s)$ and so clearly we have that $g E^R = F$ as desired. 
\end{proof}


It therefore follows from Equation \ref{eq: fppf orbit} that the fppf-orbit of the standard flag is the subfunctor of $\Fl_I$ which sends a $k$-algebra $R$ to the set of $I$-flags in $R^{\oplus I_s}$ which have the property that there exists an fppf-ring map $\varphi: R \to S$ such that if $F' = \Fl_{I_s}(\varphi)(F)$ then $F'_i / F'_{<i}$ is free for all $i \in I$. Let us denote the fppf-orbit $\Orb_{\GL_{I_s}}(E)$ by $\mF_I$. Then by the orbit-stabilizer theorem we have that 
\begin{equation*}
\Gr_{I_s} / B_{I_s} \cong \mF_I
\end{equation*}
but we \textbf{do not know} if $\mF_I = \Fl_I$. In the finite-dimensional case this is proved by noting that given an $R$-module that is projective of rank $1$ there does exist an fppf-map $\varphi: R \mapsto S$ such that $\varphi_!(M)$ is free. If $I$ is finite, one can repeat this construction for all $F_i / F_{<i}$ and note that a composition of fppf-homomorphisms is also an fppf-homomorphism. However a colimit of infinitely many fppf extensions need not be finitely presented, so this method fails for $I$ infinite. However if it is true that $\mF_I = \Fl_I$, it answers this \cite[{05VF}]{stacks-project} question from the Stacks Project in the affirmative.


\section{The Bruhat Decomposition}\label{sec: the bruhat decomposition}\label{sec: the bruhat decomposition}

Let $I$ be a well-ordered set. The main tool in the proof of the Bruhat decomposition for $\Fl_I$ is gradations of the intersection of two $I$-flags. We begin with a general discussion of these objects before moving on to the aforementioned proof.


\begin{lemma}\label{lem: a gradation of the intersection gives a gradation of each filtration}
Let $(E , I)$ and $(F , J)$ be embedded linear filtrations of an $R$-module $M$, and let $C: I \times J \to \Sub(M)$ be a function. 

\begin{enumerate}

\item If $C$ is a gradation of $(E \cap F , I \times J)$, then $C^1$ is a gradation of $(E , I)$ and $C^2$ is a gradation of $(F , J)$, where we define $C^1 : I \to \Sub(M)$ by $C^1_i = \sum_{j \in J} C_{(i , j)}$, and similarly for $C^2$. 

\item If, for all $i \in I$ we have $E_i = \bigoplus_{(k , \ell) \leq (i , j_{\infty})} C_{(k , \ell)}$, and for all $j \in J$ we have $F_j = \bigoplus_{(k , \ell) \leq (i_{\infty} , j)} C_{(k , \ell)}$, then $C$ is a gradation of $E \cap F$. 

\end{enumerate}
\end{lemma}

\begin{proof}
\
\begin{enumerate}

\item We only show the result for $C^1$, the other case being similar. Given $i \in I$, we want to show that $E_i = \bigoplus_{k \leq i} C^1_k = \bigoplus_{k \leq i} \sum_{j \in J} C_{(k , j)}$. (DS1) is satisfied because the family of sets $( \{ (k , j ) : j \in J \} : k \in I )$ of the product $I \times J$ is disjoint, therefore the family of submodules $( \sum_{j \in J} C_{(k , j)} : k \in I)$ is independent, which implies for any $i \in I$, the family $( \sum_{j \in J} C_{(k , j)} : k \leq i)$ is independent as well. To prove (DS2), note that because $F$ is embedded, there exists a greatest element $j_\infty \in J$ such that $F_{j_\infty} = M$. Then because $C$ is a gradation of $E \cap F$, we have that 
\begin{equation*}
E_i = E_i \cap V = E_i \cap F_{j_\infty} = [E \cap F]_{(i , j_\infty)} = \bigoplus_{(k , j) \leq (i , j_\infty)} C_{(k , j)}  = \sum_{k \leq i} \sum_{j \in J} C_{(k , j)}. 
\end{equation*}

\item First of all, note that the two conditions imply that $C_{(i , j)} \subseteq E_i \cap F_j$ for all $i \in I , j \in J$. We now prove (G1). Let $(i , j) \in I \times J$. To prove (DS1), note that by hypothesis we have
\begin{equation*}
M = E_{i_\infty} = \bigoplus_{(k , \ell) \leq (i_\infty , j_\infty)} C_{(k , \ell)} = \bigoplus_{(k , \ell) \in I \times J} C_{(k , \ell)}.
\end{equation*}

Therefore in particular, the family $(C_{(k , \ell)} : (k , \ell) \in I \times J )$ is independent, hence any subfamily is independent, and in particular $(C_{(k , \ell)} : (k , \ell) \leq (i , j) )$ is. 

To prove (DS2), suppose we are given $v \in E_i \cap F_j$, since $E_i = \bigoplus_{(k , \ell) \leq (i , j_{\infty})} C_{(k , \ell)}$. Then we can write 
\begin{equation*}
v = w_{i_1, j_1} + \ldots + w_{i_r, j_r} + u_{k_1, \ell_1} + \ldots + u_{k_s, \ell_s}
\end{equation*}
where for all $1 \leq a \leq r$ we have $(i_a, j_a) \leq (i , j)$ and $w_{i_a, j_a} \in C_{(i_a, j_a)}$, and for all $1 \leq b \leq s$ we have $k_b \leq i$ but $\ell_b > j$ and $u_{k_b, \ell_b} \in C_{(k_b , \ell_b)}$. Subtraction yields
\begin{equation*}
u_{k_1, \ell_1} + \ldots + u_{k_s, \ell_s} = v - (w_{i_1, j_1} + \ldots + w_{i_r, j_r}).
\end{equation*}
But since $(i_a , j_a) \leq (i , j)$ implies $i_a \leq i$ and hence $E_{i_a} \subseteq E_i$, and similarly we have $F_{j_a} \subseteq F_j$, it follows that $C_{(i_a, j_a)} \subseteq E_{i_a} \cap F_{j_a} \subseteq E_i \cap F_j$. Therefore the RHS is in $F_j$, so the LHS is as well. However we have that 
\begin{equation*}
F_{j_\infty} = \bigoplus_{(k , \ell) \leq (i_\infty , j_\infty)} C_{(k , \ell)} = \left( \bigoplus_{(k , \ell) \leq (i_\infty , j)} C_{(k , \ell)} \right) \oplus  \left( \bigoplus_{k \leq i_\infty, \ell > j} C_{(k , \ell)} \right) = F_j \oplus \left( \bigoplus_{k \leq i_\infty, \ell > j} C_{(k , \ell)} \right),
\end{equation*}
and $u_{k_1, \ell_1} + \ldots + u_{k_s, \ell_s}  \in \bigoplus_{k \leq i_\infty, \ell > j} C_{(k , \ell)}$, so it must be that $u_{k_1, \ell_1} + \ldots + u_{k_s, \ell_s} = 0$, and 
\begin{equation*}
v = w_{i_1, j_1} + \ldots + w_{i_r, j_r} \in \sum_{(k , \ell) \leq (i , j)} C_{(k , \ell)}. 
\end{equation*}

\end{enumerate}

\end{proof}

If $M$ is an $R$-module and $L$ and $N$ are direct summands of $M$, it is not necessarily true that $L \cap N$ is a direct summand of $M$\footnote{Modules for which this is true are said to have the \emph{summand intersection property}. For example Kaplansky \cite{kaplansky2018} showed that the summand intersection property is satisfied by free modules over a PID.}. It follows that if $E$ and $F$ are two $I$-flags in an $R$-module $M$, even though each of $E$ and $F$ has a gradation, their intersection $E \cap F$ need not. Of course if $R$ is a field, then all subspaces are complemented, hence $E \cap F$ has an almost gradation which, by Corollary \ref{cor:every almost gradation of a well-ordered or product of well-ordered filtration is a gradation}, is a gradation. However in order to state the Bruhat decomposition over a general ring, we need to restrict the flags that we consider.

If $E$ is an $I$-flag and $F$ is a $J$-flag in $M$, we say that the pair $E , F$ \emph{is in constant (free) relative position} if there exists a gradation $C$ of $E \cap F$ such that for every $i \in I_s$ there exists a unique $j := w_{E , F}(i) \in J_s$ such that $C_{(i , j)}$ is projective (free) of constant rank $1$, and for all $j \in J_s$ there exists a unique $i := w_{E , F}(j) \in I_s$ such that $C_{(i , j)}$ is projective (free) of constant rank $1$. In this case we obtain functions $w_{E , F} : I_s \leftrightarrows J_s: w_{F , E}$ which do not depend on the choice of gradation, since $C_{(i , j)} \cong [E \cap F]_{(i , j)}  /  [ E \cap F ]_{< (i , j)}$, and which are inverses of each other. Denote by $\CRP_I(M)$ (resp. $\CRP_I^\free(M)$) the set of pairs of $I$-flags in $M$ that are in constant (resp. free) relative position. The following proposition describes all of the rest of the $C_{(i , j)}$. 

\begin{proposition}\label{prop: all non chosen pairs are zero for constant relative position}
Suppose that $E , F$ are in constant relative position and $C$ is a gradation of $E \cap F$. Then $C_{(i , j)} = 0$ for $j \neq w_{E, F}(i)$. 
\end{proposition}

\begin{proof}
By Lemma \ref{lem: a gradation of the intersection gives a gradation of each filtration} part (1), we have that $C^1$ is a gradation of $E$, hence it is an almost gradation (Proposition \ref{prop: gradation implies almost gradation}), so for any $i \in I$ we have $E_i = E_{<i} \oplus \bigoplus_{j \in J} C_{(i , j)}$, and therefore $E_i / E_{<i} \cong \bigoplus_{j \in J} C_{(i , j)}$. By definition of a flag, if $i \notin I_s$, then this quotient is $0$, hence $C_{(i , j)} = 0$ for all $j$, and if $i \in I_s$, then this quotient is finitely generated projective of rank $1$. Given a prime ideal $\fp \in \Spec(R)$, we localize this quotient at $\fp$ to obtain 
\begin{equation*}
( E_i / E_{<i} )_\fp \cong \bigoplus_{j \in J} (C_{(i , j)})_\fp. 
\end{equation*}
Since $C$ is a gradation, in particular each $C_{(i , j)}$ is projective, so each localization $(C_{(i , j)})_\fp$ is free. The LHS is free of rank $1$, and $(C_{(i , w_{E , F}(i))})_\fp$ is free of rank $1$. Hence $(C_{(i , j)})_\fp = 0$ for $j \neq w_{E, F}(i)$. Since $\fp$ was arbitrary and a module being zero is a punctual property, it follows that $C_{(i , j)} = 0$ for $j \neq w_{E, F}(i)$. 
\end{proof}


\begin{definition}
We say that an $I$-flag $E$ is \emph{successively free} if $E_i / E_{<i}$ is free for all $i \in I$. 
\end{definition}

We can compare $\CRP_I^\free(M)$ with $\CRP_I(M)$ for successively free flags. 

\begin{proposition}\label{prop: relative position of successively free flags}
Suppose that a pair of $I$-flags $F , E$ are in constant relative position. Then $E$ is successively free if and only if $F$ is, and in this happens if and only if they are in constant \textbf{free} relative position. 
\end{proposition}

\begin{proof}
We first prove that if $E$ is successively free, then so is $F$ (the converse is similar). By hypothesis there exists a gradation $C$ of $F \cap E$, and by Lemma \ref{lem: a gradation of the intersection gives a gradation of each filtration}, $C^2$ is a gradation of $E$. Therefore for any $j \in I$ we have $E_j / E_{<j} \cong \bigoplus_{i \in I} C_{(i , j)}$. If $j \in I_s$, by definition of an $I$-flag, the former is projective of rank $1$, but by hypothesis it is free, so $E_j / E_{<j}$ is a rank $1$ free module. But by definition of constant relative position, $C_{(i , j)}$ is projective of rank $1$ for $i = w_{E,  F}(j)$ and by Proposition \ref{prop: all non chosen pairs are zero for constant relative position}, the rest are $0$. Therefore $C_{(w_{E, F}(j) , j )}$ is also a rank $1$ free module. Hence $C_{(i , j)}$ is free for all $i , j \in I$, and so $F$ and $E$ are in constant free relative position. Furthermore $C^1$ is a gradation of $F$, and for all $i \in I$, we have $F_i / F_{<i} \cong \bigoplus_{j \in I} C_{(i , j)}$ and the latter is a direct sum of free modules so the former is free. 

Conversely, suppose that $E$ and $F$ are in constant free relative position. Choose a gradation $C$ of $F \cap E$. By hypothesis, $C_{(i , j)}$ is free for all $i , j \in I$, so by Lemma \ref{lem: a gradation of the intersection gives a gradation of each filtration}, $F_i / F_{<i} \cong \bigoplus_{j \in I} C_{(i , j)}$ is free as well, so $F$ is successively free, and so is $E$ by a similar computation. 

\end{proof}


The situation when working over a field becomes much less complicated. The following lemma shows that in this case, every pair of flags $E, F$ is in constant free relative position. 

\begin{lemma}\label{lem: over a field every pair is in constant free relative position}
Suppose that $I, J$ are well ordered sets, $E$ is an $I$-flag, and $F$ is a $J$-flag in a \textbf{vector space} $V$. Let $C: I \times J \to \Sub(V)$ be a gradation of $(E \cap F , I \times J)$. Then for each $i \in I_s$, there exists a unique $j \in J_s$ such that $C_{(i , j)}$ has dimension $1$.
\end{lemma}

\begin{proof}
Note that for any $i \in I$, we have that $\dim(E_i / E_{<i} ) = 0 , 1$, and $i \in I_s$ if and only if $\dim(E_i / E_{<i} ) = 1$. The same holds for $J$. By Lemma \ref{lem: a gradation of the intersection gives a gradation of each filtration}, $C^1 : I \to \Sub(V)$ is a gradation of $(E , I)$ and hence for any $i \in I$ we have 
\begin{equation*}
\dim( E_i / E_{<i} ) = \dim( C^1_i  ) = \dim (\bigoplus_{j \in J} C_{(i , j)} ).
\end{equation*}
If $i \in I_s$, then this number is $1$, hence there is a unique $j \in J$ such that $\dim(C_{(i , j)}) = 1$. But this means that $\dim( \bigoplus_{i \in I} C_{(i , j)} ) \geq 1$, and since we have 
\begin{equation*}
1 \geq \dim( F_j / F_{<j} ) = \dim( C^2_j  ) =  \dim( \bigoplus_{i \in I} C_{(i , j)} ).
\end{equation*}
it follows that $\dim(F_j / F_{<j}) = 1$, and therefore that $j \in J_s$. 
\end{proof}


\subsection{The General Linear Action on Relative Position Pairs}

If $P$ is a poset, $C : P \to \Sub(V)$ is a function, and $g \in \GL(V)$, then we define $gC : P \to \Sub(V)$ by $(gC)_p = gC_p$. It is clear that if $C : P \to \Sub(V)$ is a gradation of a poset filtration $(F , P)$ of $V$ then $g C$ is a gradation of $gF$. From this it follows that if $(E , I)$ and $(F , J)$ are linear filtrations of $V$, $C : I \times J \to \Sub(V)$ is a gradation of the filtration intersection $E \cap F$, and $g \in \GL(V)$, then $g C : I \times J \to \Sub(V)$ is a gradation of $g[E \cap F] = (gE) \cap (gF)$. Therefore there are left actions of the general linear group $\GL(M)$ on the sets $\CRP_I(M)$ and $\CRP_I^\free(M)$ by $(E , F) \mapsto (gE , gF)$. 

We now describe the set of orbits of $\GL(M)$ on $\CRP_I^\free(M)$. Recall that if $I$ is a set, then $\Perm(I)$ denotes the group of permutations of $I$, i.e. bijections $I \to I$ with group law given by composition.
 
\begin{theorem}\label{thm: relative position orbits}
Let $I$ be a well-ordered set and $M$ be a free $R$-module of rank $| I_s |$. The formula $[E, F] \mapsto w_{E , F}$ gives a well-defined bijection $\GL(M) \backslash \CRP_I^\free(M) \to \Perm(I_s)$.
\end{theorem}

\begin{proof}
First we prove that the formula is well defined. If $[E , F] = [E', F']$ in $\GL(M) \backslash \CRP_I^\free(M)$, by definition there exists an element $g \in \GL(V)$ such that $gE = E'$ and $gF = F'$. Let $C : I \times I \to \Sub(V)$ be a gradation of $E \cap F$. Then as remarked above,  $gC$ is a gradation of $(gE) \cap (gF) = E' \cap F'$. Since $g$ is invertible, for any $i , j \in I$, $g C_{(i , j)}$ is projective of constant rank $1$ if and only if $C_{(i , j)}$ is, hence $w_{E , F} = w_{E' , F'}$. 

Next we prove injectivity of the map $\GL(M) \backslash \CRP_I^\free(M) \to \Perm(I_s)$. Suppose that $[E , F]$ and $[E', F']$ are such that $w_{E , F} = w_{E' , F'}$. We'll call this permutation $w$ for short. Let $C : I \times I \to \Sub(M)$ and $D : I \times I \to \Sub(M)$ be gradations of the filtration intersections $E \cap F$ and $E' \cap F'$ respectively. Define 
\begin{equation*}
A = \{ (i , j) \in I_s \times J_s : j = w(i) \},
\end{equation*}
And recall that $C_{(i , j)}, D_{(i , j)} \neq 0$ if and only if $(i , j) \in A$. For each $(i , j) \in A$ let $( e_{i , j} )$ be a basis for the free module $C_{(i , j)}$ (it is here that we use the fact that $E, F$ are in constant \textbf{free} relative position). Similarly we define $f_{i , j} \in D_{(i , j)}$. Then we have that 
\begin{equation*}
( e_{i , j} : (i , j ) \in A ) \ \ \ \text{ and } \ \ \ ( f_{i , j} : (i , j ) \in A )
\end{equation*}
are bases for $V$. Define an $R$-module homomorphism $g : M \to M$ by $g(e_{i , j}) = f_{i , j}$. Since $g$ sends a basis to a basis, it is invertible. We claim that $gE = E'$ and $gF = F'$, from which it will follow that $[E , F] = [E'. F']$. Given any $i \in I$, we have 
\begin{equation*}
gE_i  = g \left( \bigoplus_{(k , \ell) \leq (i_\infty , j)} C_{(k , \ell)} \right) = \bigoplus_{(k , \ell) \leq (i_\infty , j)} g C_{(k , \ell)} = \bigoplus_{(k , \ell) \leq (i_\infty , j)} D_{(k , \ell)} = E'_i
\end{equation*}
where the first and last equalities hold because $C$ and $D$ are gradations of $E \cap F$ and $E' \cap F'$ respectively. A similar computation shows $g F = F'$. 

Finally we prove surjectivity of the map $\GL(M) \backslash \CRP_I^\free(M) \to \Perm(I_s)$. Let $w \in \Perm(I_s)$. Let $(e_i : i \in I_s)$ be a basis for $M$. Define functions $E, F : I \to \Sub(M)$ by 
\begin{equation*}
E_i = \Span( e_{w(j)} \mid j \in I_s , j \leq i ) \ \ \ \text{ and } \ \ \  F_i = \Span( e_j \mid j \in I_s , j \leq i ) .
\end{equation*}
Then $E$ and $F$ are $I$-flags in $M$. Furthermore, we define a function $C : I \times I \to \Sub(M)$ by 

\begin{equation*}
C_{(i , j)} = \left\{ \begin{matrix} \Span(e_j) & \text{ if } & i \in I_s \text{ and } j = w(i) \\ 0 & \text{ otherwise } & \end{matrix} \right. .
\end{equation*}

Note first that because $w$ is a bijection, $( C_{(i , j)} \mid (i , j) \in I \times I )$ is an independent family. We now show that $C$ satisfies the conditions of part (2) of Lemma \ref{lem: a gradation of the intersection gives a gradation of each filtration} for the filtrations $E$ and $F$. 

Given $i \in I$, we wish to show that $E_i = \bigoplus_{(k , \ell) \leq (i , j_{\infty})} C_{(k , \ell)}$. First of all (DS1) is satisfied because the family $( C_{(k , \ell)} \mid (i , j_{\infty}) \in I \times I )$ is a subfamily of an independent family and hence independent. Furthermore (DS2) follows from the computation
\begin{equation*}
\sum_{(k , \ell) \leq (i , j_{\infty})} C_{(k , \ell)} = \sum_{k \leq i , k \in I_s} \Span( e_{w(k)} ) = E_i.  
\end{equation*}
 
 Similarly, given $j \in I$, we wish to show that $F_j = \bigoplus_{(k , \ell) \leq (i_{\infty} , j)} C_{(k , \ell)}$. The proof of (DS1) is the same as before, and for (DS2), we compute
 \begin{equation*}
\sum_{(k , \ell) \leq (i_{\infty} , j)} C_{(k , \ell)} = \sum_{\ell \leq j, \ell \in I_s} \Span( e_{\ell} ) = F_j.  
\end{equation*}

Therefore by Lemma \ref{lem: a gradation of the intersection gives a gradation of each filtration}, $C: I \times I \to \Sub(V)$ is a gradation of the filtration intersection $E \cap F$. By construction of $C$, given any $i \in I_s$, the unique $j \in I_s$ such that $C_{(i , j)}$ is nonzero is given by $j = w(i)$. Hence $w_{E , F} = w$, so indeed the map is surjective. 

\end{proof}

If $V$ is a vector space then by Lemma \ref{lem: over a field every pair is in constant free relative position} we have that $\CRP^\free_I(V) = \Fl_I(V) \times \Fl_I(V)$, and so in this case Theorem \ref{thm: relative position orbits} can be stated as follows.  

\begin{corollary}
 If $V$ is a vector space, then the formula $[E , F] \mapsto w_{E, F}$ gives a well-defined bijection $\GL(V) \backslash [ \Fl_I(V) \times \Fl_I(V) ] \to \Perm(I_s)$. 
\end{corollary}


\subsection{The Flag Bruhat Decomposition}\label{sub: the flag bruhat decomposition}

Fix an $I$-flag $E \in \Fl_I(M)$ and denote by $\CRP_I(M , E)$ the set of $I$-flags in $\Fl_I(M)$ that are in constant relative position with $E$. Furthermore denote by $B_E$ the stabilizer of $E$ under the action of $\GL(M)$ on $\Fl_I(M)$. Suppose that $F \in \CRP(M , E)$ and let $C$ be a gradation of $F \cap E$. Then for any $b \in B_E$, $bC$ is a gradation of $bF \cap bE = bF \cap E$. Since $(bC)_{(i , j)} \cong C_{(i, j)}$ for all $i , j \in I$, it follows that $bF$ is in constant relative position with $E$ as well and indeed that $w_{bF, E} = w_{F , E}$. Hence we have a left action of $B_E$ on $\CRP_I(M , E)$, and the map $\CRP_I(M , E) \to \Perm(I_s)$ defined by $F \mapsto w_{F , E}$ lifts to a well-defined map on the quotient $B_E \backslash \CRP_I(M , E) \to \Perm(I_s)$. 

We are now in a position to state the Bruhat decomposition. 

\begin{theorem}[The Bruhat Decomposition]\label{thm: the bruhat decomposition}
Suppose that $I$ is a well-ordered set, $M$ is a free $R$-module of rank $|I_s|$ and $E$ is a successively free $I$-flag in $M$. Then the map $B_E \backslash \CRP_I(M , E) \to \Perm(I_s)$ defined by $[F] \mapsto w_{F , E}$ is a bijection. 
\end{theorem}

\begin{proof}
Note first that by Proposition \ref{prop: relative position of successively free flags}, every flag $F \in \CRP_I(M , E)$ is successively free and $F$ and $E$ are in constant free relative position. 

Now suppose that $w_{F , E} = w_{F', E}$. Then under the map $\GL(M) \backslash \CRP^\free_I(M) \to \Perm(I_s)$ we have that $[F , E]$ and $[F' , E]$ map to the same permutation. Hence by Theorem \ref{thm: relative position orbits}, there exists some $g \in \GL(M)$ such that $g (F , E) = (F' , E)$, which implies that $(gF , gE) = (F' , E)$. In particular we have $gE = E$, so $g \in B_E$, but then since $gF = F'$, this implies that $[F] = [F']$ in $B_E \backslash \CRP_I(M , E)$, so indeed the map is injective. 

Furthermore, given $w \in \Perm(I_s)$, again by Theorem \ref{thm: relative position orbits}, there exists a pair $(F , F') \in \CRP^\free_I(M)$ such that $w_{F , F'} = w$. However by Proposition \ref{prop: naive orbit of the standard flag}, there exists $g \in \GL(M)$ such that $g F' = E$. But we have $[gF , E] = [gF , gF'] = [F , F']$ in $\GL(M) \backslash \CRP^\free_I(M) \to \Perm(I_s)$, so $w_{gF , E} = w$. Then $gF \in \CRP_I(M)$, and we have that $[gF] \mapsto w$, so the map is surjective as well. 
\end{proof}

The Bruhat decomposition is less complicated over a field. Indeed by Lemma \ref{lem: over a field every pair is in constant free relative position}, every $I$-flag $F$ in a vector space $V$ is in constant relative position to a fixed flag $E \in \Fl_I(V)$, so we obtain the following. 

\begin{corollary}[The Bruhat Decomposition for Vector Spaces]\label{cor: the bruhat decomposition for vector spaces}
If $I$ is a well-ordered set, $V$ is a vector space of dimension $|I_s|$, and $E$ is any $I$-flag in $V$, then the map $B_E \backslash \Fl_I(V) \to \Perm(I_s)$ defined by $[F] \mapsto w_{F , E}$ is a bijection. 
\end{corollary}


\subsection{The General Linear Bruhat Decomposition}

Let $I$ be a well-ordered set, $M$ a free $R$-module of rank $|I_s|$ and $E$ a fixed successively free $I$-flag in $M$. Define a map $m_E : \GL(M) \to \Fl_I(M)$ by sending $g \mapsto gE$. The inverse image $\GL_E(M) := m_E^{-1}( \CRP_I(M , E) )$ is closed under both left and right multiplication by $B_E$, because given $g \in \GL_E(M)$ and $b \in B_E$, first of all we have $g b E = g E \in \CRP_I(M , E)$, and second of all we have $b g E \in \CRP_I(M , E)$ by Subsection \ref{sub: the flag bruhat decomposition}. Therefore we have a left and right action of $B_E$ on $\GL_E(M)$ by multiplication. The following theorem describes the double cosets, i.e. the orbits under the corresponding $B_E \times B_E$ action. 

\begin{theorem}[The General Linear Bruhat Decomposition]\label{thm: the general linear bruhat decomposition}
Let $I$ be a well-ordered set, $M$ be a vector space of rank $|I_s|$, and $E$ be a successively free $I$-flag in $M$. The formula $[g] \mapsto w_{gE , E}$ gives a well-defined bijection 
\begin{equation*}
B_E \backslash \GL_E(M) / B_E \to \Perm(I_s). 
\end{equation*}
\end{theorem}

\begin{proof}
If $[g] = [g']$, then by definition $g' = b g c$ for $b , c \in B_E$, and we have that $w_{g' E , E} = w_{b g c E , E} = w_{b g E , E}$. But by the Bruhat decomposition (Theorem \ref{thm: the bruhat decomposition}), $w_{b g E , E} = w_{ g E , E}$, so indeed our formula gives a well-defined map. 

Suppose that $w_{gE , E} = e_{g' E , E}$. Then again by the Bruhat decomposition, we have that $[gE] = [g'E]$ in $B_E \backslash \CRP_I(M, E) \to \Perm(I_s)$, so by definition there exists $b \in B_E$ such that $bgE = g'E$. But this implies that $(g')^{-1} bg \in B_E$, and hence that $bgc = g'$ for some $c \in B_E$. Hence our map is injective. 

Finally, given $w \in \Perm(I_s)$, by the Bruhat decomposition, there exists an $I$-flag $F \in \CRP_I(M , E)$ such that $w = w_{F , E}$. But because the $\GL(M)$ action on the set of successively free $I$-flags is transitive (Proposition \ref{prop: naive orbit of the standard flag}), there exists $g \in \GL(M)$ such that $gE = F$. However $F \in \CRP_I(M , E)$, so by definition $g \in \GL_E(M)$. Hence we have $[g] \mapsto w_{gE , E} = w_{F , E} = w$, and our map is surjective as well. 
\end{proof}

Again the situation is much simpler for vector spaces. For a vector space $V$, every $I$-flag is successively free and $\CRP_I(V , E) = \Fl_I(V)$, so we have $\GL_E(V) = \GL(V)$. Hence we obtain the following corollary. 

\begin{corollary}[The General Linear Bruhat Decomposition for Vector Spaces]\label{thm: the general linear bruhat decomposition}
Let $I$ be a well-ordered set, $V$ be a vector space of dimension $|I_s|$, and $E$ be any $I$-flag in $V$. The formula $[g] \mapsto w_{gE , E}$ gives a well-defined bijection 
\begin{equation*}
B_E \backslash \GL(V) / B_E \to \Perm(I_s). 
\end{equation*}
\end{corollary}

Let $R$ be a ring and let $E$ be the standard flag in $R^{\oplus I_s}$. Note that $E$ is successively free. Then define $\GL_{I_s , E}(R)$ to be the group of $I_s \times I_s$ matrices $g$ with entries in $R$ such that $g E \in \CRP_I(R^{\oplus I_s} , E)$. The following corollary follows from applying Theorem \ref{thm: the general linear bruhat decomposition} to the $R$-module $R^{\oplus I_s}$. 

\begin{corollary}[The General Linear Bruhat Decomposition over a Ring]\label{cor: the general linear bruhat decomposition over a ring}
Let $I$ be a well-ordered set and $R$ be a ring. Then 
\begin{equation*}
\GL_{I_s, E}(R) = \bigsqcup_{\sigma \in \Perm(I_s)} B_{I_s}(R) \sigma B_{I_s}(R)
\end{equation*}
\end{corollary}

\begin{proof}
Under the isomorphism $\GL(R^{\oplus I_s}) \cong \GL_{I_s}(R)$ given by the standard basis, the stabilizer subgroup $B_E \subseteq \GL(R^{\oplus I_s})$ corresponds to the subgroup $B_{I_s}(R)$ of upper triangular matrices (see Proposition \ref{prop: stab of standard flag is upper triangulars}). Notice that if $\sigma \in \Perm(I_s)$ (and by abuse of notation we denote the corresponding permutation matrix in $\GL_{I_s}(R)$ also by $\sigma$), then we have $w_{\sigma E , E} = \sigma$. Therefore the double coset $B_{I_s}(R) \sigma B_{I_s}(R)$ corresponds to the permutation $\sigma$ under the bijection of Theorem \ref{thm: the general linear bruhat decomposition}, so the desired result follows. 
\end{proof}

If $k$ is a field, then $\GL_{I_s, E}(k) = \GL_{I_s}(k)$ so from Corollary \ref{cor: the general linear bruhat decomposition over a ring} we obtain the following: 

\begin{corollary}[The General Linear Bruhat Decomposition over a Field]\label{cor: the general linear bruhat decomposition over a field}
Let $I$ be a well-ordered set and $\ff$ be a field. Then 
\begin{equation*}
\GL_{I_s}(k) = \bigsqcup_{\sigma \in \Perm(I_s)} B_{I_s}(k) \sigma B_{I_s}(k)
\end{equation*}
\end{corollary}

\begin{remark}\label{rem: the general linear bruhat decomposition over a ring}
It is \textbf{not} true for general rings $R$ that $\GL_{I_s}(R) = \GL_{I_s}(R)$ because the LHS will typically be larger (see e.g. \cite{onn-prasad-vaserstein2006}). 
\end{remark}


\subsection{The Functorial Bruhat Decomposition}
 
Let $I$ be a well-ordered set. We collect the previous results to give statements about the $k$-functors $\Fl_I$ and $\GL_{I_s}$. For the rest of the paper we define $B = B_{I_s}$ when there is no chance of confusion. There is both a left and right action of the $k$-group functor $B$ on the $k$-functor $\GL_{I_s}$ by left and right multiplication respectively, which gives a left action of $B \times B$ on $\GL_{I_s}$, explicitly defined by $b , c \in B(R) , g \in \GL_{I_s}(R) \mapsto b g c^{-1}$. If $\sigma \in \Perm(I_s)$, then the permutation matrix of $\sigma$ (which we also denote by $\sigma$) is in $\GL_{I_s}(k)$, and we define a subfunctor $B \sigma B \subseteq \GL_{I_s}$ by 
\begin{equation*}
R \mapsto B(R) \sigma B(R) = \{ b \sigma c \mid b , c \in B(R) \}.
\end{equation*} 

By definition the fppf-orbit of $\sigma$ under the $B \times B$ orbit is the fppf-image of the natural transformation $B \times B \to \GL_{I_s}$ given by $b , c \in B(R) \mapsto b \sigma c^{-1}$, which we denote by $\sh(B \sigma B)$ and by Equation \ref{eq: fppf-image description} is explicitly given by
\begin{equation}\label{eq: sheafification of BsigmaB}
R \mapsto \{ g \in \GL_{I_s}(R) \mid \exists \text{ an fppf-$R$-algebra } \varphi : R \to S \text{ with } \GL_{I_s}(\varphi)(g) = b \sigma c \text{ for } b , c \in B(S) \}. 
\end{equation}
Clearly we have $B \sigma B \subseteq \sh(B \sigma B)$. 

Recall that two subfunctors $\mY$ and $\mZ$ of a $k$-functor $\mX$ are called \emph{disjoint} if for every $k$-algebra $R$, the subsets $\mY(R)$ and $\mZ(R)$ of $\mX(R)$ are disjoint. Also recall that a family of subfunctors $(\mY_\alpha \mid \alpha \in A)$ of a $k$-functor $\mX$ are said to \emph{cover} $\mX$ if for all $k$-algebras $R$ \textbf{which are fields}, we have that $\mX(R) = \bigcup_{\alpha \in A} \mY_\alpha(R)$. 

\begin{theorem}[The Functorial General Linear Bruhat Decomposition]\label{thm: the functorial general linear bruhat decomposition}
Let $I$ be a well-ordered set. The subfunctors $\sh(B \sigma B)$ are disjoint and cover the $k$-functor $\GL_{I_s}$. 
\end{theorem}

\begin{proof}
First of all, it follows from Corollary \ref{cor: the general linear bruhat decomposition over a ring} that $B \sigma B$ and $B \tau B$ are disjoint. Suppose that $g \in \sh(B\sigma B)(R) \cap \sh(B \tau B)(R)$. By definition there exist fppf-ring maps $\varphi : R \to S$ and $\psi: R \to T$ such that $\GL_{I_s}(\varphi)(g) \in B(S) \sigma B(S)$ and $\GL_{I_s}(\psi)(g) \in B(T) \sigma B(T)$. By Lemma \ref{lem: flatness is preserved by pushouts}, $\iota_S:  S \to S \otimes_R T$ and $\iota_T : T \to S \to S \otimes_R T$ are both faithfully flat and by Lemma \ref{lem: finitely presented is preserved by pushouts} they are finitely presented, hence they are fppf ring maps. As $B \sigma B$ and $B \tau B$ are subfunctors, we obtain that 
\begin{equation*}
B(S \otimes_R T) \sigma B(S \otimes_R T) \ni \GL_{I_s}(\iota_S \circ \varphi)(g)  =  \GL_{I_s}(\iota_T \circ \varphi)(g) \in B(S \otimes_R T) \tau B(S \otimes_R T). 
\end{equation*}
But $B(S \otimes_R T) \sigma B(S \otimes_R T)$ and $B(S \otimes_R T) \tau B(S \otimes_R T)$ are disjoint, so we obtain the desired contradiction. 

The fact that the subfunctors $( B \sigma B \mid \sigma \in \Perm(I_s) )$ cover $\GL_{I_s}$ follows from Corollary \ref{cor: the general linear bruhat decomposition over a field}, and since $B \sigma B \subseteq \sh(B \sigma B)$, this implies the desired covering. 
\end{proof}

There is a left action of the $k$-group functor $B$ on the $k$-functor $\Fl_I$ given by $b \in B(R), F \in \Fl_I(R) \mapsto bF$. Given $\sigma \in \Perm(I_s)$, we denote by $B \sigma E$ the orbit of $\sigma E$ under the $B$ action and by $X^\circ_\sigma$ its $B$ fppf-orbit, i.e. $\sh(B \sigma E) = X_\sigma^\circ$. Explicitly the latter is given by 
\begin{equation*}
X_\sigma^\circ(R) = \{ F \in \Fl_{I}(R) \mid \exists \text{ an fppf-$R$-algebra } \varphi : R \to S \text{ with } \Fl_{I}(\varphi)(F) = b \sigma E \text{ for } b \in B(S) \}. 
\end{equation*}
The subfunctors $X^\circ_\sigma$ for $\sigma \in \Perm(I_s)$ are called \emph{Schubert cells}. We note that $B \sigma E \subseteq X^\circ_\sigma$ for all $\sigma$.

\begin{theorem}[The Functorial Bruhat Decomposition]
Let $I$ be a well-ordered set. The subfunctors $X^\circ_\sigma$ are disjoint and cover the $k$-functor $\Fl_I$. 
\end{theorem}

\begin{proof}
The fact that $B \sigma E$ and $B \tau E$ are disjoint for $\sigma \neq \tau$ follows from the fact that they are distinct orbits of a group action. It can be deduced from this that the subfunctors $X^\circ_\sigma$ and $X^\circ_\tau$ are disjoint for $\sigma \neq \tau$ using a proof similar to that of Theorem \ref{thm: the functorial general linear bruhat decomposition}. 

Furthermore, if $R$ is a $k$-algebra which is a field, then we have that $\CRP_I(R^{\oplus I_s} , E) = \Fl_I(R)$, and by Corollary \ref{cor: the bruhat decomposition for vector spaces}, $\Fl_I(R) = \bigsqcup_{\sigma \in \Perm(I_s)} B(R) \sigma E$. This implies that the family $(B \sigma E \mid \sigma \in \Perm(I_s))$ covers $\Fl_I(R)$, and as $B \sigma E \subseteq X_\sigma^\circ$, the desired covering follows. 
\end{proof}

We now discuss the relationship between $\sh(B \sigma B)$ and $X^\circ_\sigma$. 

\begin{lemma}\label{lem: sheafification of BsigmaB is B stable}
The subfunctor $\sh(B \sigma B) \subseteq \GL_\nn$ is stable under right multiplication by $B$. 
\end{lemma}

\begin{proof}
Given a $k$-algebra $R$,  $x \in \sh(B \sigma B)(R)$, and $b \in B(R)$, by Equation \ref{eq: sheafification of BsigmaB} there exists an fppf ring homomorphism $\varphi : R \to S$ and $b' , c' \in B(S)$ such that $\GL_\nn(\varphi)(x) = b' \sigma c'$. Then because $\GL_\nn(\varphi)$ is a group homomorphism, we have that $\GL_\nn(\varphi)(x b) = b' \sigma c' \GL_\nn(\varphi)(b)$ and since $B \subseteq \GL_\nn$ is a subfunctor, $\GL_\nn(\varphi)(b) \in B(S)$. Therefore again by Equation \ref{eq: sheafification of BsigmaB}, $x b \in \sh(B \sigma B)(R)$.
\end{proof} 

It follows from Lemma \ref{lem: sheafification of BsigmaB is B stable} that the quotient $\sh(B \sigma B) / B$ makes sense and that there is a monomorphism $\sh(B \sigma B) / B \to \GL_\nn / B$. The following proposition follows from Equation \ref{eq: image of a subfunctor under the quotient map}. 

\begin{proposition}\label{prop: identification of the schubert cells with subfunctors of the quotient}
The subfunctors $\sh(B \sigma B) / B \subseteq \GL_{I_s} / B$ and $X^\circ_\sigma \subseteq \mF_I$ correspond under the isomorphism $\GL_{I_s} / B \cong \mF_I$, 
\end{proposition}


\section{The Infinite Bruhat Order}\label{sec: the infinite bruhat order}


For the remainder of the paper, we turn our attention to describing the closures of the subfunctors $B \sigma B \subseteq \GL_{I_s}$ and the Schubert cells $\sh(B \sigma B) / B \subseteq \GL_{I_s} / B$. We shall specialize to the case of $I = \omega + 1 = \{ 0 , 1 , \ldots, \omega \}$ so that $I_s = \nn = \{ 1 , 2 , \ldots \}$. Many of the proofs here rely heavily on this specialization, and the extension of these results to the case of general well-ordered sets $I$ is a potential direction for future research.  



\subsection{A Topology on the Space of Permutations} 

If $I$ is a set, we place a topology on $\Perm(I)$ as follows. There is an injective map $\Perm(I) \to I^I$ sending $\sigma \mapsto (\sigma(i) )_{i \in I}$, so we can identify $\Perm(I)$ as a subset of $I^I$. We now give $I$ the discrete topology, $I^I$ the product topology, and $\Perm(I) \subseteq I^I$ the subspace topology. In fact, since $I$ has the discrete topology, all functions $I \to I$ are continuous, hence we can identify the space $C(I , I)$ of continuous functions $I \to I$ with $I^I$, and the product topology on $I^I$ gives the \emph{compact-open} topology on $C(I , I)$. Therefore we call our subspace topology on $\Perm(I)$ the compact-open topology as well. 


\begin{proposition}
A sequence $\sigma_1 , \sigma_2 , \ldots $ in $\Perm(\nn)$ converges to $\sigma$ in the compact-open topology if and only if the sequence is \emph{eventually equal} to $\sigma$, i.e. if and only if for all $m \in \nn$ there exists $N \in \nn$ such that if $n \geq N$ then $\sigma_n(\ell) = \sigma(\ell)$ for all $\ell \leq m$. 
\end{proposition}

\begin{proof}
Recall that there is a base for the product topology on $\nn^\nn$ which consists of sets of the form $U_1 \times U_2 \times \ldots $ such that each $U_i$ is open in $\nn$ and $U_i = \nn$ for all but finitely many $i$. Then the sequence $(\sigma_n)$ converges to $\sigma$ if and only if for all basic open subsets $U = U_1 \times U_2 \times \ldots $, there exists $N \in \nn$ such that if $n \geq N$, then $\sigma_n \in B$. 

Now suppose that $(\sigma_n)$ converges to $\sigma$. Then given $m \in \nn$, since $\nn$ carries the discrete topology, the set $U = \{ \sigma(1)  \times \sigma(2) \times \ldots \times \{ \sigma(m) \} \times \nn \times \ldots$ is a basic open set containing $\sigma$, so there exists $N \in \nn$ such that for all $n \geq N$, $\sigma_n \in U$, which implies in particular that if $n \geq N$ the $\sigma_n(\ell) = \sigma(\ell)$ for all $\ell \leq m$. 

Conversely, suppose $(\sigma_n)$ is eventually equal to $\sigma$ and let $U = U_1 \times U_2 \times \ldots \times U_m \times \nn \times \ldots$ be a basic open set containing $\sigma$. Then by definition there exists some $N \in \nn$ such that if $n \geq N$ then $\sigma_n(\ell) = \sigma(\ell)$ for all $\ell \leq m$. Since $\sigma \in B$, this implies that $\sigma_n \in U$ for all $n \geq N$, so $(\sigma_n) \to \sigma$ as desired. 
\end{proof}


\begin{corollary}\label{cor: inverse of a convergent sequence is convergent}
Suppose that $\sigma_1 , \sigma_2, \ldots $ is a sequence in $\Perm(\nn)$, that converges to $\sigma$. Then $\sigma_1^{-1} , \sigma_2^{-1}, \ldots$ converges to $\sigma^{-1}$.
\end{corollary}

\begin{proof}
Given $m \in \nn$, because $\sigma$ is a bijection, $\sigma^{-1}( \{ 1, \ldots, m \} ) \subseteq \nn$ is a set of cardinality $m$. Let $m'$ be the maximal element of this set. Then because $\sigma_1, \sigma_2, \ldots$ converges to $\sigma$, there exists $N \in \nn$ such that if $n \geq N$ then $\sigma_n(\ell) = \sigma(\ell)$ for all $\ell \leq m'$. But then for any such $n \geq N$, given $j \leq m$, we know that $\sigma^{-1}(j) \leq m'$, hence $\sigma_n(\sigma^{-1}(j)) = \sigma(\sigma^{-1}(j))$, hence $\sigma_n(\sigma^{-1}(j)) = j$, so $\sigma^{-1}_n(j) = \sigma^{-1}(j)$. So indeed $(\sigma_n)$ converges to $\sigma$. 
\end{proof}


It is well known that $S_n := \Perm([n])$ is generated by transpositions, i.e. permutations swapping two elements of $[n]$, but this is no longer the case for $\Perm(\nn)$. Given any $\sigma \in \Perm(I)$ for any set $I$, we define $\Supp(\sigma) = \{ i \in I \mid \sigma(i) \neq i \}$. Then because all of the transpositions in $\Perm(\nn)$ have finite support, the subgroup they generate is contained inside the subgroup $\Perm_\fin(\nn)$ of finitely supported permutations. This is a proper subgroup since there exist permutations with infinite support, for example the permutation $\rho$ which sends all odd $n$ to $n + 2$, all even $m$ with $m \neq 2$ to $m - 2$, and $2 \mapsto 1$. However as we shall see, it is true that every $\sigma \in \Perm(\nn)$ is a limit (in the compact-open topology) of a sequence of the form $( t_1 , t_2 t_1, t_3 t_2 t_1 , \ldots )$ where each $t_i$ is a transposition.


\subsection{The Bruhat Order} 


Suppose that $X$ and $Y$ are subsets of $\nn$ both of cardinality $n$, and that the elements of $X$ are denoted by $x_i$ for $1 \leq i \leq n$ such that $x_1 < x_2 < \ldots < x_n$ and similarly for $Y$. Then we say that $X \leq Y$ if and only if $x_i \leq y_i$ for all $i$. This defines a partial order on the set $\mP_n(\nn)$ of subsets of $\nn$ of cardinality $n$. Using this, we define a partial order, called the \emph{Bruhat order}, on $\Perm(\nn)$ by setting $\tau \leq \sigma$ if and only if for every $n \in \nn$, we have that $\{ \tau(1) , \ldots, \tau(n) \} \leq \{ \sigma(1), \ldots, \sigma(n) \}$.

\begin{note}[Tableau Criterion for the Bruhat Order]
The Bruhat order has the following interpretation in terms of Young tableaux. To a permutation $\sigma \in \Perm(\nn)$, we associate an infinite tableau of shape $(1 , 2 , 3 , \ldots )$ where the $n$th row has entries $\sigma(1) , \sigma(2) , \ldots, \sigma(n)$ written in increasing order. For example, the permutation $\rho$ which sends all odd $n$ to $n + 2$, all even $m$ with $m \neq 2$ to $m - 2$, and $2 \mapsto 1$ has one-line notation and associated tableau given respectively by 
\begin{align*}
3, 1, 5, 2, 7, 4, 9, 6, 11 , \ldots \ \ \ \ \ \ \ \ \ &\young(3,13,135,1235,12357)
\\ &\ \ \ \ \ \ \ \ \vdots
\end{align*}
Then for another permutation $\tau$ we have that $\tau \leq \sigma$ if and only if the entry in each box of the tableau associated to $\tau$ is less than or equal to the corresponding box in the tableau associated to $\sigma$.
\end{note} 


There is an alternate description of Bruhat order that will be useful to us. If $g \in \End_\nn(R)$, any ordered pair of subsets $X , Y \subseteq \nn$ defines a $|X| \times |Y|$ submatrix of $g$ by $g_{X,Y} = [ g_{i,j} ]_{i \in X, j \in Y}$. Let $r_{X , Y}(g)$ denote the rank of the submatrix $g_{X , Y}$. If $\sigma$ is a permutation matrix, then $r_{X , Y}(\sigma)$ is simply the number of $1$'s in the submatrix $\sigma_{X , Y}$. We adopt the following notational shorthand: for $m , n \in \nn$ with $m \leq n$, we let $[m , n] = \{ m , m + 1 , \ldots , n \}$, if $m < n$, we let $(m , n] = \{ m+ 1, \ldots, n\}$ and $[m ,n) = \{ m,\ldots, n-1 \}$, and we further shorten $[1 , n]$ to $[n]$. Then as a special case of above, for any $p , q \in \nn$ and $g \in \End_\nn(R)$, define $r_{p , q}(g) = r_{[p],[q]}(g)$ and $r_{\geq p , q}(g) = r_{[p, \infty), [q]}(g)$. Note that if $\sigma \in \Perm(\nn)$ then $r_{p , q}(\sigma) =  | \{ n \leq q \mid \sigma(n) \leq p \} |$ and $r_{\geq p , q}(\sigma) = | \{ n \leq q \mid \sigma(n) \geq p \} |$. The former is the number of $1$'s in the first $p$ rows and $q$ columns of the matrix of $\sigma$, and the latter is the number of $1$'s in the lower left submatrix of $\sigma$ with upper right corner $(p, q)$. 

\begin{proposition}\label{prop: region criterion for bruhat order}
Given $\sigma, \tau \in \Perm(\nn)$, we have that $\sigma \leq \tau$ if and only if $r_{i , j}(\sigma) \geq r_{i , j}(\tau)$ for all $i, j \in \nn$.
\end{proposition}

\begin{proof}
To prove ($\Rightarrow$), suppose that $\sigma \leq \tau$ and let $i , j \in \nn$. Then by definition we have that $\{ \sigma(1) , \ldots , \sigma(j) \} \leq \{ \tau(1) , \ldots , \tau(j) \}$. Let $n_1 , \ldots n_j$ be a rearrangement of the numbers $1 , \ldots , j$ such that $\sigma(n_1) < \ldots < \sigma(n_j)$ and similarly define $m_1 , \ldots m_j$ for $\tau$. Then by definition $\sigma(n_\ell) \leq \tau(n_\ell)$ for all $1 \leq \ell \leq j$.

Now suppose for contradiction that $a:= r_{i , j}(\sigma) < r_{i , j}(\tau):= b$. Then we have that $\{ n_1, \ldots, n_a \} = \{ n \leq j \mid \sigma(n) \leq i \}$ and $\{ m_1 , \ldots m_b \} = \{ m \leq j \mid \tau(m) \leq i \}$. But then we have 
\begin{equation*}
\sigma(n_{a + 1}) \overset{(1)}{>} i \overset{(2)}{\geq} \tau(m_{a + 1})
\end{equation*}
where (1) follows from the definition of $a$, and (2) follows because $a < b$ so $a + 1 \leq b$. This contradicts that $\sigma \leq \tau$, as desired. 

Now we prove ($\Leftarrow$). Suppose that $r_{i , j}(\sigma) \geq r_{i , j}(\tau)$ for all $i, j \in \nn$, and let $q \in \nn$. As before, let $n_1 , \ldots n_q$ be a rearrangement of the numbers $1 , \ldots , q$ such that $\sigma(n_1) < \ldots < \sigma(n_q)$, and similarly define $m_1 , \ldots m_q$ for $\tau$. We prove that $\sigma(n_\ell)  \leq \tau(n_\ell)$ for all $1 \leq \ell \leq q$. By hypothesis we have $r_{\tau(m_\ell) , q}(\sigma) \geq r_{\tau(m_\ell) , q}(\tau) = \ell$, therefore $n_1, \ldots, n_\ell \in \{ n \leq q \mid \sigma(n) \leq \tau(m_\ell) \}$, so in particular we have $\sigma(n_\ell) \leq \tau(m_\ell)$. Therefore $\{ \sigma(1) , \ldots , \sigma(q) \} \leq \{ \tau(1) , \ldots , \tau(q) \}$. Since $q$ was arbitrary, we have that $\sigma \leq \tau$. 
\end{proof}


Note that for any two permutations $\sigma$ and $\tau$, we have that $\Supp(\sigma^{-1} \tau) = \{ n \in \nn \mid \sigma(n) \neq \tau(n) \}$. If $\sigma \neq \tau$, then $\Supp(\sigma^{-1} \tau)$ is nonempty, so it has a minimal element which we call $d(\sigma , \tau)$.


\begin{lemma}\label{lem: less than in bruhat implies first differing place less than}
Suppose that $\sigma , \tau \in \Perm(\nn)$ are such that $\sigma < \tau$, and let $q = d(\sigma , \tau)$. Then $\sigma(q) < \tau(q)$. 
\end{lemma}

\begin{proof}
By Proposition \ref{prop: region criterion for bruhat order} we must have $r_{\tau(q) , q}(\sigma) \geq r_{\tau(q) , q}(\tau)$. However since the first $q - 1$ columns of $\sigma$ and $\tau$ are the same and $\tau$ has a $1$ in the $(\tau(q) , q)$ entry, $\sigma$ must have that $\sigma(q) \leq \tau(q)$. But by hypothesis $\sigma$ and $\tau$ differ in the $q$th column so it must be that $\sigma(q) < \tau(q)$. 
\end{proof}


\begin{lemma} \label{lem: composing with a transposition in the bruhat order}
Suppose that $\tau \in \Perm(\nn)$ and $p , q \in \nn$ with $p < q$. Then $\tau (p , q) < \tau$ if and only if $\tau(p) > \tau(q)$. 
\end{lemma}

\begin{proof} Let $\rho = \tau (p , q)$. It is clear from the definition of the transposition $(p , q)$ that for any $n \in \nn \smallsetminus \{ p , q \}$, we have $\rho(n) = \tau(n)$. 

First we prove ($\Rightarrow$). Suppose that $\tau (p , q)  < \tau$. Then $p = d(\tau , \tau(p, q))$ and so by Lemma \ref{lem: less than in bruhat implies first differing place less than} we have that $\tau(q) = \tau (p , q) (p) < \tau(p)$ as desired. 

Now we prove ($\Leftarrow$). Suppose that $\tau(p) > \tau(q)$. Since the matrix of $\tau (p , q)$ is obtained from that of $\tau$ by swapping columns $p$ and $q$, notice that $r_{i, j}( \tau (p , q) ) = r_{i, j}(\tau)$ as long as at least one of the following conditions is satisfied: (1) $j < p$, (2) $j \geq q$, (3) $i < \tau(q)$, or (4) $i \geq \tau(p)$. If $\tau(q) \leq i < \tau(p)$ and $p \leq j < q$, then the submatrix $(\tau (p , q))_{i , j}$ is identical to that of $\tau_{i , j}$ except with an extra $1$ in the $(\tau(q) , p)$ position, so we have $r_{i, j}( \tau (p , q) ) = r_{i, j}(\tau) + 1$. Therefore for any $i , j \in \nn$, we have $r_{i, j}( \tau (p , q) ) \geq r_{i, j}(\tau)$, which implies by Proposition \ref{prop: region criterion for bruhat order} that $\tau (p , q) < \tau$. 
\end{proof}


\begin{lemma}\label{lem: going down in bruhat using transpositions}
If $\sigma , \tau \in \Perm(\nn)$ with $\sigma < \tau$, let $p = d(\sigma , \tau)$. Then there exists $q \in \nn$ with $p < q$ such that if $t = (p ,  q)$ then

\begin{enumerate}

\item $\sigma \leq \tau t < \tau$,

\item $\sigma(p) \leq \tau t(p) < \tau(p)$, and 

\item $\tau t(m) = \tau(m)$ for all $m < p$. 

\end{enumerate}
\end{lemma}

\begin{proof}
First of all, we let $r = \tau^{-1}(\sigma(p))$, and we claim that $p < r$. Indeed if $p = r$, then $\sigma^{-1} \tau(r) = p = r$, and hence $p = r \notin \Supp(\sigma^{-1} \tau)$, a contradiction. On the other hand, if $r < p$ then $\sigma^{-1} \tau(r) = p \neq r$, hence $r \in \Supp(\sigma^{-1} \tau)$ contradicting the minimality of $p$. 

Now we note that Lemma \ref{lem: less than in bruhat implies first differing place less than} implies that $\sigma(p) < \tau(p)$. Therefore we may consider the set $\{ n > p \mid \sigma(p) \leq \tau(n) < \tau(p) \}$. As we have just shown, $r > p$ and $\tau(r) = \sigma(p)$, so this set contains $r$, and is therefore nonempty. Hence it has a minimal element, which we call $q$. By definition we have $p < q$ and also $\tau(p) > \tau(q)$, hence by Lemma \ref{lem: composing with a transposition in the bruhat order} it follows that $\tau (p , q) < \tau$. Additionally (3) is clear since $p < q$. Furthermore, by definition of $q$ we have $\sigma(p) \leq \tau(q) <\tau(p)$, so (2) follows because $\tau t(p) = \tau(q)$.

Now we claim that $\sigma \leq \tau (p , q)$. On the one hand, since the matrix of $\tau (p , q)$ is obtained from that of $\tau$ by swapping columns $p$ and $q$, notice that $r_{i, j}( \tau (p , q) ) = r_{i, j}(\tau)$ as long as at least one of the following conditions is satisfied: (1) $j < p$, (2) $j \geq q$, (3) $i < \sigma(p)$, or (4) $i \geq \tau(p)$.

Therefore as $\sigma \leq \tau$, for any such $i$ and $j$ by Proposition \ref{prop: region criterion for bruhat order} we have that $r_{i, j}( \sigma ) \geq r_{i, j}( \tau (p , q) )$. Now suppose that $\sigma(p) \leq i < \tau(p)$ and $p \leq j < q$. We break up the submatrices of $\sigma$ and $\tau(p , q)$ consisting of the first $i$ rows and $j$ columns into three disjoint regions: (a) $[1 , \sigma(p) ) \times [j]$, (b) $(\sigma(p) , i] \times [1 , p)$, and (c) $[ \sigma(p) , i] \times [p , j]$ so that we have
\begin{equation*}
r_{i , j}(\sigma) = r_{[1 , \sigma(p) ),  [j]}(\sigma) + r_{ (\sigma(p) , i] , [1 , p)}(\sigma) + r_{[ \sigma(p) , i] , [p , j] }(\sigma).
\end{equation*}
and similarly for $\tau$ and $\tau(p. q)$. 

For region (a), we have that $r_{[1 , \sigma(p) ),  [j]}(\tau ( p , q) ) = r_{[1 , \sigma(p) ),  [j]}(\tau ) \geq r_{[1 , \sigma(p) ),  [j]}(\sigma)$, where the equality follows because this submatrix satisfies condition (3) from above, and the inequality from the fact that $\sigma \leq \tau$. For region (b), we note that the $m$th columns of $\sigma$, $\tau$, and $\tau(p , q)$ are all the same for $1 \leq m < p$, hence $r_{ [\sigma(p) , i] , [1 , p)}(\sigma) = r_{ (\sigma(p) , i] , [1 , p)}(\tau(p , q))$. Finally for region (c), $\sigma$ has a $1$ in the $(\sigma(p) , p)$-entry, hence $r_{[ \sigma(p) , i] , [p , j] }(\sigma) \geq 1$. However by minimality of $q$, we have that $r_{[ \sigma(p) , i] , [p , j] }(\tau) = 0$. Since the matrix of $\tau(p , q)$ is obtained from that of $\tau$ by swapping columns $p$ and $q$, we have that $r_{[ \sigma(p) , i] , [p , j] }(\tau(p , q)) \leq 1$. Combining our results for all three regions together yields $r_{i , j}(\sigma) \geq r_{i ,j}(\tau(p , q))$, hence this inequality holds for every $i$ and $j$, so by Proposition \ref{prop: region criterion for bruhat order}, we have that $\sigma \leq \tau(p , q)$. 

\end{proof}


\begin{corollary}\label{cor: can make the support of a lower permutation smaller by applying transpositions}
Suppose that $\sigma, \tau \in \Perm(\nn)$ are such that $\sigma < \tau$. There exists transpositions $t_1,\ldots, t_n$ such that $\sigma \leq \tau t_1 \ldots t_n < \ldots < \tau t_1 \leq \tau$ and $d(\sigma ,  \tau  t_1 \ldots t_n) > d(\sigma , \tau)$.
\end{corollary}

\begin{proof}
Let $p= d(\sigma, \tau)$. We repeatedly apply Lemma \ref{lem: going down in bruhat using transpositions} to obtain transpositions $t_1, \ldots, t_n$ such that $\sigma \leq \tau t_1 \ldots t_n < \ldots < \tau t_1 < \tau$ and $\sigma(p) = \tau t_1 \ldots t_n(p) < \ldots < \tau t_1(p) < \tau(p)$. Then because $\sigma$, $\tau$, and $\tau t_1 \ldots t_n$ all agree on all $m < p$, and $\sigma$ and $\tau t_1 \ldots t_n$ agree at $p$ as well, we have $d(\sigma , \tau t_1 \ldots t_n) > p = d(\sigma , \tau)$ as desired. 
\end{proof}


Compare the following proposition to \cite{fulton1997} \S 10.5 Corollary 1 for the finite case. 


\begin{proposition}\label{prop: going infinitely far down in the bruhat order}
Suppose that $\sigma, \tau \in \Perm(\nn)$ are such that $\sigma < \tau$. Then one of the following is true: 

\begin{enumerate}

\item There exists transpositions $t_1,\ldots, t_n$ such that $\tau t_1 \ldots t_{i + 1} \leq \tau t_1 \ldots t_{i }$ for all $0 \leq i < n$ and $\tau t_1 \ldots t_n= \sigma$. 

\item There exists a sequence of transpositions $t_1, t_2, \ldots $ such that $\tau t_1 \ldots t_{i + 1} < \tau t_1 \ldots t_{i}$ for all $i \geq 0$, and the sequence $(\tau t_1 \ldots t_n)_{n = 0}^\infty$ converges to $\sigma$. 

\end{enumerate}
\end{proposition}

\begin{proof}
We need to show that if (1) does not hold, then (2) must hold. Since $\sigma < \tau$, by Corollary \ref{cor: can make the support of a lower permutation smaller by applying transpositions}, there exist transpositions $t_1, \ldots, t_{n_1}$ such that $\sigma \leq \tau t_1 \ldots t_n < \ldots < \tau t_1 \leq \tau$ and $d(\sigma ,  \tau  t_1 \ldots t_{n_1}) > d(\sigma , \tau)$. Since (1) does not hold, we must have $\sigma < \tau t_1 \ldots t_n$. Therefore we may recursively apply the corollary to obtain $n_1 < n_2 < \ldots$ and $t_i$ for all $i \in \nn$ such that $\sigma < \tau t_1 \ldots t_{i + 1} \leq \tau t_1 \ldots t_{i} < \tau$ for all $i \geq 0$, and 
\begin{equation*}
d(\sigma , \tau) < d(\sigma ,  \tau  \prod_{i = 1}^{n_1} t_i ) < d(\sigma ,  \tau  \prod_{i = 1}^{n_2} t_i ) < \ldots
\end{equation*}
This sequence of strict inequalities implies that for any $m \in \nn$, there exists $N \in \nn$ such that $m < d(\sigma ,  \tau  \prod_{i = 1}^{N} t_i )$, and hence for any $n \geq N$, $\tau  \prod_{i = 1}^{n} t_i (m) = \sigma(m)$, therefore $(\tau t_1 \ldots t_n)_{n = 0}^\infty$ converges to $\sigma$. 
\end{proof}



\section{Closure of the Schubert Cells}\label{sec: closure of the schubert cells}


In this section we define $B := B_\nn$. Given $\sigma \in \Perm(\nn)$, let $B \sigma B$ denote the image subfunctor of the natural transformation $m_\sigma : B \times B \to \GL_\nn$ defined by sending $b , c \in B(R)$ to $b \sigma c$ for any $k$-algebra $R$. 

If $I$ and $J$ are sets, $R$ is a $k$-algebra, and $x \in \Hom_{I, J}(R)$ and $\ell \in \nn$, we define $\fd_\ell(x)$ to be the ideal of $R$ generated by the $\ell \times \ell$ minors of $x$.  


\begin{lemma} \label{lem: schubert cells contained in affine subfunctors}
Given any $\sigma \in \Perm(\nn)$, there exist functions $X , Y : \nn \to \mP_{\text{fin}}(\nn)$ such that $X_i = [1, \ldots, n_i]$ and if $i \leq j$ then $n_i \leq n_j$ and similarly for $Y$, and such that the natural transformation $m_\sigma: B \times B \to \GL_\nn$ factors through the inclusion map $\GL_{X , Y} \hookrightarrow \GL_\nn$.
\end{lemma}

\begin{proof}
Let $X(n) = [1 ,  \max\{ \sigma(1) , \sigma(2), \ldots, \sigma(n) \} ]$ and $Y = [1 ,  \max\{ \sigma^{-1}(1) , \sigma^{-1}(2), \ldots, \sigma^{-1}(n) \} ]$. Then for any $k$-algebra $R$ and any $b , c \in B(R)$, note that the $i$th column of $b \sigma$ is the $\sigma(i)$th column of $b$ which is upper triangular, therefore if $j > i$ then the $i , j$ entry of $b \sigma$ is zero if $j > \sigma(i)$. Furthermore, $i$th column of $b \sigma c$ is obtained by taking linear combinations of columns $1 , 2 , \ldots , i$ of $b \sigma$. Therefore the $i , j$ entry $b \sigma c$ is zero as long as $j > [1 ,  \max\{ \sigma(1) , \sigma(2), \ldots, \sigma(i) \} ]$. Hence $b \sigma c \in \Mat_X(R) \subseteq \End_\nn(R)$. A similar argument shows that $(b \sigma c)^{-1} = c^{-1} \sigma^{-1} b^{-1} \in \Mat_Y(R)$. Therefore $b \sigma c \in \GL_{X , Y}(R)$, so $m_\sigma$ factors through $\GL_{X, Y}$. 
\end{proof}


\begin{definition}
For $\sigma \in \Perm(\nn)$, we say that a pair of functions $X , Y : \nn \to \mP_{\text{fin}}(\nn)$ satisfying the conditions of Lemma \ref{lem: schubert cells contained in affine subfunctors} is \emph{adapted} to $\sigma$. 
\end{definition}


\begin{corollary}
\
\begin{enumerate}

\item For any $\sigma \in \Perm(\nn)$, there exists functions $X , Y : \nn \to \mP_{\text{fin}}(\nn)$ such that $\overline{B \sigma B} \subseteq \GL_{X , Y}$. 

\item $\overline{B \sigma B}$ is affinely representable and hence an fppf-sheaf. 

\item $\overline{B \sigma B} = \overline{\sh( B \sigma B )}$.

\end{enumerate}
\end{corollary}

\begin{proof}
\
\begin{enumerate}

\item By Lemma \ref{lem: schubert cells contained in affine subfunctors}, we have that $B \sigma B \subseteq \GL_{X , Y}$ which is a closed subfunctor of $\GL_\nn$. Therefore $\overline{B \sigma B} \subseteq \GL_{X , Y}$ as well. 

\item Recall that $\GL_{X , Y}$ is affinely representable and as discussed in the background section, all closed subfunctors of affinely representable functors are also affinely representable. By Theorem \ref{thm: representable by schemes implies fppf sheaf} all representable functors are fppf-sheaves.

\item For one containment, we have $B \sigma B \subseteq \sh(B \sigma B)$, and hence $\overline{B \sigma B} \subseteq \overline{\sh( B \sigma B)}$. For the other containment, by the previous part, $\overline{B \sigma B}$ is an fppf-sheaf, and it contains $B \sigma B$, hence by Corollary \ref{cor: fppf subfunctor that contains a subfunctor contains its sheafification} it contains $\sh(B \sigma B)$, but it is also closed so it contains $\overline{\sh( B \sigma B )}$ as desired. 

\end{enumerate}
\end{proof}


\begin{proposition}\label{prop: invertible upper triangulars}
A upper triangular matrix $b \in \End_\nn(R)$ is invertible if and only if $b_{i ,i} \in R^\times$ for all $i \in \nn$. 
\end{proposition}

\begin{proof}
To prove ($\Rightarrow$), suppose that an upper triangular matrix $b \in \End_\nn(R)$ is invertible. Then $b \in \GL_\nn(R)$ and $b \in B(R)$. Recall (Proposition \ref{prop: stab of standard flag is upper triangulars}) that $B(R)$ is the stabilizer subgroup of the standard flag under the action of $\GL_\nn(R)$ on $\Fl_{\omega + 1}(R)$, and is therefore a subgroup of $\GL_\nn(R)$. Hence $a = b^{-1}$ is also upper triangular. In particular this means that $a_{i , i} b_{i , i} = 1$ for all $i \in \nn$ so indeed $b_{i , i}$ is a unit. 

Now to prove ($\Leftarrow$), suppose that $b_{i , i}$ is a unit for all $i$. We first construct a right inverse of $b$ which we shall call $c$. We first make $c$ upper triangular by setting $c_{i , j} = 0$ if $i > j$. Then we define the entries $c_{i , n}$ with $i \leq n$ of the $n$th column of $c$ recursively. Let $c_{n , n} = b_{n , n}^{-1}$. Now assuming we have defined $c_{j + 1, n}, c_{j + 2, n} , \ldots, c_{n , n}$, we define 
\begin{equation*}
c_{j , n} = - b_{j , j}^{-1} ( b_{j, j + 1} c_{j + 1 , n} + b_{j , j + 2} c_{j + 2 , n} \ldots +  b_{j , n}  c_{n , n} ). 
\end{equation*}
Since $b$ is upper triangular, it follows that $bc = \delta$. Hence we have constructed a right inverse for $b$. 
Now we construct a left inverse $a$ for $b$. This is similar to the construction of $c$ but not exactly the same as these matrices are column finite but not row finite. Again we make $a$ upper triangular by setting $a_{i , j} = 0$ if $i > j$ and this time define the entries $a_{n , j}$ with $j \geq n$ of the $n$th row of $a$ recursively. Let $a_{n , n} = b_{n , n}^{-1}$. If we have defined $a_{n , n}, a_{n  n + 1}, \ldots a_{n , j - 1}$, we set 
\begin{equation*}
a_{n , j } = b_{j , j}^{-1} (a_{n , n} b_{n , j} + a_{n, n + 1} b_{n + 1 , j} + \ldots + a_{n , j - 1 } b_{j - 1 , j} ). 
\end{equation*}
Again since $b$ is upper triangular, we have $ab = \delta$. Since $b$ has both a left and a right inverse, it is invertible (and in fact $a = c$). 
\end{proof}


\begin{corollary}\label{cor: B represented by a flat algebra}
$B$ is representable by a flat $k$-algebra.  
\end{corollary}

\begin{proof}
By Proposition \ref{prop: invertible upper triangulars}, $B$ is represented by the localization of the ring $k[x_{i , j} \mid i \leq j ]$ at the multiplicative subset $T$ generated by $x_{i , i}$ for all $i \in \nn$. Since the polynomial ring $k[x_{i , j} \mid i \leq j ]$ is a free and hence flat $k$-module, the inclusion $k \to k[x_{i , j} \mid i \leq j ]$ is a flat $k$-algebra homomorphism, and since localization is always flat, the localization map $k[x_{i , j} \mid i \leq j ] \to T^{-1} k[x_{i , j} \mid i \leq j ]$ is also flat. Because the composition of flat $k$-algebra homomorphisms is also flat, it follows that $T^{-1} k[x_{i , j} \mid i \leq j ]$ a flat $k$-algebra. 
\end{proof}


A morphism of representable $k$-functors $h^S \to h^R$ is said to be \emph{flat} if the corresponding ring homomorphism $R \to S$ under the Yoneda embedding is flat.


\begin{proposition}
For any $\sigma \in \Perm(\nn)$, $\overline{B \sigma B}$ is stable under left and right multiplication by $B$, i.e. for any $k$-algebra $R$ and any $z \in \overline{B \sigma B}(R)$ and $b \in B(R)$ we have $b z , z b \in \overline{B \sigma B}(R)$. 
\end{proposition}

\begin{proof}
By Lemma \ref{lem: schubert cells contained in affine subfunctors}, we may choose functions $X , Y :\nn \to \mP_\fin(\nn)$ that are adapted to $\sigma$. Note that for such $X$ and $Y$, $\GL_{X , Y}$ is stable under left and right multiplication by $B$. Consider the diagram of $k$-functors $\GL_{X , Y} \times B \overset{\pi}{\longrightarrow} \GL_{X , Y} \overset{m_\sigma}{\longleftarrow} B \times B$ where $\pi$ is the projection map. All three of these $k$-functors are representable, hence in particular $m_\sigma$ is quasi-compact and quasi-separated. 

Furthermore we claim that $\pi$ is flat. Indeed suppose that  $\GL_{X , Y} \cong h^S$ and $B \cong h^T$. Then $\GL_{X , Y} \times B \cong h^S \times h^T \cong h^S \times_{h^k} \times h^T \cong h^{S \otimes_k T}$. Therefore $\pi$ is flat if and only if the homomorphism $S \to S \otimes_k T$ given by $s \mapsto s \otimes 1$ is flat. But by Corollary \ref{cor: B represented by a flat algebra}, $k \to T$ is flat, hence by Lemma \ref{lem: flatness is preserved by pushouts}, $S \to S \otimes_k T$ is also flat. 

We now consider the pullback of $\GL_{X , Y} \times B \overset{\pi}{\longrightarrow} \GL_{X , Y} \overset{m_\sigma}{\longleftarrow} B \times B$ shown below.

\begin{center}
\begin{tikzcd}
(\GL_{X , Y} \times B) \times_{\GL_{X , Y}} (B \times B) \arrow[r, "f"] \arrow[d , "g" '] &  B \times B \arrow[d , "m_\sigma"] \\
\GL_{X , Y} \times B \arrow[r, "\pi" '] &  \GL_{X , Y}
\end{tikzcd}
\end{center}

By \cite{demazure-gabriel1980} I.2.6.15, we have that $\overline{\im g} = \pi^{-1}(\overline{\im m_\sigma})$. Now $\im m_\sigma$ is the subfunctor $B \sigma B \subseteq \GL_{X , Y}$. Furthermore, since pullbacks of $k$-functors can be computed component-wise, we have that $\im g$ is the subfunctor $B \sigma B \times B$ of $\GL_{X , Y} \times B$. Hence we have that $\overline{B \sigma B \times B} = \pi^{-1}(\overline{B \sigma B}) = \overline{B \sigma B} \times B$.

Now since $\GL_{X , Y}$ is stable under right multiplication by $B$, we consider the multiplication map $m : \GL_{X , Y} \times B \to \GL_{X , Y}$. Since $B \sigma B$ is stable under right multiplication by $B$, we have $B \sigma B \times B \subseteq m^{-1}(B  \sigma B)$, and hence we can calculate as follows. 
\begin{equation*}
\overline{B \sigma B} \times B = \overline{B \sigma B \times B} \subseteq \overline{ m^{-1}(B  \sigma B) } \subseteq m^{-1}( \overline{B \sigma B} )
\end{equation*}
The last equality follows because $m^{-1}( B \sigma B) \subseteq m^{-1}(\overline{ B \sigma B })$ and the latter is a closed subfunctor. Hence indeed $\overline{B \sigma B}$ is stable under right multiplication by $B$. A similar proof shows stability under left multiplication. 
\end{proof}


\begin{proposition}
Let $k$ be a \textbf{domain}. Suppose that $\sigma , \tau \in \Perm(\nn)$ are such that $\sigma < \tau$. If $\mathcal{Z}$ is a closed subfunctor of $\GL_\nn$ which contains $\tau$, is stable under left and right multiplication by $B$, and $\mZ \subseteq \GL_{X , Y}$ for some $X , Y : \nn \to \mP_\fin(\nn)$, then $\mathcal{Z}$ also contains $\sigma$. 
\end{proposition}

\begin{proof}
First we prove the result if $\sigma = \tau (p , q)$ for some $p , q \in \nn$ with $p < q$. We define $g = \sigma L_{q,p}(x) \in \End_\nn(k[x])$, where $L_{q,p}(x)$ is the elementary matrix which is the same as the identity matrix but with an $x$ in the $( q, p )$-entry, so that $g$ is obtained from $\sigma$ by adding $x$ times the $q$th column to the $p$th column. Then since $\sigma$ and $L_{q,p}(x)$ are invertible, so is $g$. By the Yoneda lemma, $g \in \GL_\nn(k[x])$ corresponds to a natural transformation $\eta: h^{k[x]} \to \GL_\nn$. Explicitly for any $k$-algebra $R$, the component $\eta_R : h^{k[x]}(R) \to \GL_\nn(R)$ sends a $k$-algebra homomorphism $\varphi : k[x] \to R$ to the matrix   
\begin{equation*}
[\varphi(g_{i, j})] = \sigma L_{q,p}(\varphi(x)) = 
\begin{blockarray}{cccccc}
&  & p &  & q &  \\
\begin{block}{c[ccccc]}
  & \ddots &  &  &  &  \\
 \sigma(p)  &  & 1 &  &  &  \\
   &  &  &  &  &  \\
 \sigma(q)  &  & \varphi(x) &  & 1 &  \\
   &  &  &  &  & \ddots \\
\end{block}
\end{blockarray} \in \GL_\nn(R)
\end{equation*}
Note that by Lemma \ref{lem: composing with a transposition in the bruhat order} the fact that $\tau (p , q) < \tau$ implies that indeed $\sigma(p) = \tau(q) < \tau(p) = \sigma(q)$. We claim that if $r := \varphi(x)$ is a unit of $R$, then $\sigma L_{q,p}(r) = b \tau c$ for some $b , c \in B_\nn(R)$. Indeed let $D_{\sigma(q)}(r^{-1})$ be the elementary matrix which is the same as the identity matrix but with $r^{-1}$ in the $(\sigma(q), \sigma(q))$-entry. Then $D_{\sigma(q)}(r^{-1})$ is upper triangular and we have that 
\begin{equation*}
D_{\sigma(q)}(r^{-1}) \sigma L_{q,p}(r) = 
\begin{blockarray}{cccccc}
&  & p &  & q &  \\
\begin{block}{c[ccccc]}
  & \ddots &  &  &  &  \\
 \sigma(p)  &  & 1 &  &  0 &  \\
   &  &  &  &  &  \\
 \sigma(q)  &  & 1 &  & r^{-1} &  \\
   &  &  &  &  & \ddots \\
\end{block}
\end{blockarray}
\end{equation*}
is obtained from $\sigma L_{q,p}(r)$ by scaling the $\sigma(q)$th row by $r^{-1}$. Since $p < q$, the elementary matrix $L_{p , q}(-r^{-1})$ is upper triangular as well, and 
\begin{equation*}
D_{\sigma(q)}(r^{-1}) \sigma L_{q,p}(r) L_{p , q}(-r^{-1}) = 
\begin{blockarray}{cccccc}
&  & p &  & q &  \\
\begin{block}{c[ccccc]}
  & \ddots &  &  &  &  \\
 \sigma(p)  &  & 1 &  & -r^{-1} &  \\
   &  &  &  &  &  \\
 \sigma(q)  &  & 1 &  & 0 &  \\
   &  &  &  &  & \ddots \\
\end{block}
\end{blockarray}
\end{equation*}
is obtained from the previous matrix by adding $-r^{-1}$ times the $p$th column to the $q$th column. Next, $D_{\sigma(p)}(-r)$ is diagonal, hence also upper triangular and 
\begin{equation*}
D_{\sigma(p)}(-r) D_{\sigma(q)}(r^{-1}) \sigma L_{q,p}(r) L_{p , q}(-r^{-1}) = 
\begin{blockarray}{cccccc}
&  & p &  & q &  \\
\begin{block}{c[ccccc]}
  & \ddots &  &  &  &  \\
 \sigma(p)  &  & -r &  & 1 &  \\
   &  &  &  &  &  \\
 \sigma(q)  &  & 1 &  & 0 &  \\
   &  &  &  &  & \ddots \\
\end{block}
\end{blockarray}
\end{equation*}
is obtained from the previous matrix by scaling the $\sigma(p)$th row by $-r$. Finally, recall that $\sigma(q) > \sigma(p)$, hence $L_{\sigma(p), \sigma(q) }(r)$ is upper triangular, and 
\begin{equation*}
L_{\sigma(p), \sigma(q) }(r)D_{\sigma(p)}(-r) D_{\sigma(q)}(r^{-1}) \sigma L_{q,p}(r) L_{p , q}(-r^{-1})  = 
\begin{blockarray}{cccccc}
&  & p &  & q &  \\
\begin{block}{c[ccccc]}
  & \ddots &  &  &  &  \\
 \sigma(p)  &  & 0 &  & 1 &  \\
   &  &  &  &  &  \\
 \sigma(q)  &  & 1 &  & 0 &  \\
   &  &  &  &  & \ddots \\
\end{block}
\end{blockarray} = \tau
\end{equation*}
is obtained from the previous matrix by adding $r$ times the $\sigma(q)$th row to the $\sigma(p)$th row. Hence $b := (L_{\sigma(p), \sigma(q) }(r)D_{\sigma(p)}(-r) D_{\sigma(q)}(r^{-1}))^{-1}$ and $c := L_{p , q}(-r^{-1})^{-1}$ are upper triangular and $\sigma L_{q,p}(r) = b \tau c$. Note that the latter is in $\mZ(R)$ since by hypothesis $\mZ$ is stable under left and right multiplication by $B$. 
 
Therefore the closed subfunctor $\eta^{-1}(\mZ) \subseteq h^{k[x]}$ contains the subfunctor $R \mapsto \{ \varphi : k[x] \to R \mid \varphi(x) \in R^\times \}$ which is precisely the open subfunctor $\mD(x)$. Since $k$ is a domain, $k[x]$ is also a domain, therefore by Lemma \ref{lem: closure of basic affine open is the whole spectrum for domains} it follows that $\eta^{-1}(\mZ)$ must be all of $h^{k[x]}$. Hence in particular for any $k$-algebra $R$, if $\psi : k[x] \to R$ is the unique map which sends $x \mapsto 0$, then we have that $\sigma = \eta_R(\psi) \in \mZ(R)$, as desired. 
 
Next, if there are transpositions $t_1,\ldots, t_n$ such that $\sigma = \tau t_1 \ldots t_n < \ldots < \tau t_1 t_2 < \tau t_1 < \tau$, we prove that $\sigma \in \mZ$ by induction on $n$. It follows from what we just proved that $\tau t_1 \in \mZ$. Now suppose that $\tau t_1 t_2 \ldots t_i \in \mZ$. Since $\mZ$ is stable under left and right multiplication by $B$, it follows that $B \tau t_1 t_2 \ldots t_i B \subseteq \mZ$. Then again the case we just proved yields that $\tau t_1 t_2 \ldots t_{i + 1} \in \mZ$.

Finally by Proposition \ref{prop: going infinitely far down in the bruhat order} the only remaining option is that there exists a sequence of transpositions $t_1, t_2, \ldots $ such that $\tau t_1 \ldots t_{i + 1} < \tau t_1 \ldots t_{i}$ for all $i \geq 0$, and the sequence $(\tau t_1 \ldots t_n)_{n = 0}^\infty$ converges to $\sigma$. By the previous case $\tau t_1 \ldots t_n \in\mZ$ for all $n \geq 0$. Let $\tau_n = \tau t_1 \ldots t_n$. By hypothesis, $\mZ \subseteq \GL_{X , Y}$ for some $X$ and $Y$, and possibly by enlarging $X_{i}$ for each $i$ if necessary, we can assume that $\sigma \in \GL_{X ,Y}$ as well. Now $\GL_{X , Y}$ is isomorphic to $h^A$ where $A = k[x_{i , j} , y_{p , q} \mid i \in X_j , p \in Y_q ] / \mathfrak{a}$ and $\fa$ is defined in the proof of Proposition \ref{prop: closed subfunctor version of GL is an ind-affine ind-scheme}. Explicitly, for any $k$-algebra $R$, this isomorphism sends a matrix $g \in \GL_{X , Y}(R)$ to the unique ring homomorphism $A \to R$ given by $x_{ i , j } \mapsto g_{i , j}$ and $y_{i , j} \mapsto (g^{-1})_{i , j}$. Under this isomorphism, $\mZ$ corresponds to a closed subfunctor of $h^A$, which by Proposition \ref{prop: properties of closed subfunctors} must be of the form $V(\fb)$ for some ideal $\fb \subseteq A$, and each $\tau_n \in \mZ$, as well as $\sigma \in \GL_{X , Y}$ corresponds to some element $\psi_n \in V(\fb)$ and $\varphi \in h^A$ respectively. We wish to show that $\psi \in V(\fb)$. To that end, let $R$ be a $k$-algebra and $a \in \fb$. Then $a = f + \mathfrak{a}$ where $f$ is some polynomial in $k[x_{i , j} , y_{p , q} \mid i \in X_j , p \in Y_q]$. Since $f$ is a polynomial, there exists some $m \in \nn$ such that all terms of $f$ are products of the variables $x_{i, j}$ and $y_{p , q}$ with $j , q \leq m$. Since $(\tau_n)$ converges to $\sigma$, by Corollary \ref{cor: inverse of a convergent sequence is convergent} we have that $(\tau^{-1}_n)$ converges to $\sigma^{-1}$ as well, hence there exists some $N \in \nn$ such that for all $n \geq N$, we have $\tau_n(\ell) = \sigma(\ell)$ and $\tau_n^{-1}(\ell) = \sigma^{-1}(\ell)$ for all $\ell \leq m$. Of course this implies $(\tau_n)_{i, j} = \sigma_{i, j}$ and $(\tau_n^{-1})_{i, j} = (\sigma^{-1})_{i, j}$ for all pairs $(i, j)$ with $j \leq m$. In particular, because $\tau_N \in \mZ$, we have that $\psi_N \in V(\fb)$, so $0 = \psi_N(a) = f(\tau_N, \tau_N^{-1})$. But $f(\tau_N, \tau_N^{-1}) = f(\sigma, \sigma^{-1})$ since $f$ only contains variables $x_{i , j}$ and $y_{i , j}$ with $j \leq m$. Therefore $\varphi(a) = f(\sigma, \sigma^{-1}) = 0$ as well. Since $a  \in \bf$ was arbitrary, it follows that $\varphi(\fb) = 0$, and hence that $\varphi \in V(\fb)$, which implies $\sigma \in \mZ$. 
\end{proof}

\begin{lemma}\label{lem: closure of basic affine open is the whole spectrum for domains}
Suppose that $R$ is a domain. If $f \in R$ is nonzero, then $\overline{\mD(f)} = h^R$

\end{lemma}

\begin{proof}
Suppose that $\mV(\fa)$ is a closed subfunctor containing $\mD(f)$. Then in particular we have that $\mD(f)(R_f) \subseteq \mV(\fa)(R_f)$. However the localization map $\loc_f : R \to R_f$ which sends $r \mapsto r/f$ certainly satisfies $\loc_f(f) \in R_f^\times$, therefore $\loc_f \in \mD(f)(R_f)$, and hence $\loc_f \in \mV(\fa)(R_f)$, so by definition we have $\loc_f(\fa) = 0$. This means that for any $a \in \fa$, we have $a / 1 = 0$ in $R_f$, which implies that $f^n a = 0$ in $R$ for some $n \in \zz_{\geq 0}$. Since $R$ is a domain and $f \neq 0$, this implies $a = 0$, hence $\fa = 0$ and $\mV(\fa) = h^R$. Since $\mV(\fa)$ was an arbitrary closed subfunctor containing $\mD(f)$, it follows that $\overline{\mD(f)} = h^R$. 

\end{proof}

\begin{corollary}\label{cor: containment of the closure of schubert cells}
Let $k$ be a \textbf{domain} and let $\sigma \in \Perm(\nn)$. Then $\overline{B \tau B} \subseteq \overline{B\sigma B}$ for all $\tau < \sigma$. 
\end{corollary}

Given finite subsets $P , Q \subseteq \nn$ with $|P| = |Q|$, we define a natural transformation $\det_{P , Q} : \GL_\nn \to \aaa^1$ by sending $g \in \GL_\nn(R) \mapsto \det g_{P , Q} \in R$. Furthermore, we define a subfunctor $Y_\sigma$ of $\GL_\nn$ by letting $Y_\sigma(R)$ be the set of all $g \in \GL_\nn(R)$ such that for all $p , q \in \nn$, all $\ell > r_{\geq p , q}(\sigma)$ and all $P \subseteq [1 , p], Q \subseteq [q , \infty)$ with $|P| = |Q| = \ell$, we have $\text{det}_{P , Q} (g) = 0$. It follows from Proposition \ref{prop: properties of closed subfunctors} that $Y_\sigma$ is a closed subfunctor.

\begin{proposition}
For any $\sigma \in \Perm(\nn)$, $Y_\sigma$ is an fppf sheaf.
\end{proposition}

\begin{proof}
First of all, given any $k$-algebra $R$ and any $R$-algebras $R_1, \ldots, R_n$ we must show that the map $Y_\sigma(R_1 \times \ldots \times R_n) \to Y_\sigma(R_1) \times \ldots \times Y_\sigma(R_n)$ induced by the projection maps $\pi_\ell : R_1 \times \ldots \times R_n \to R_\ell$ is bijective. Injectivity follows from the fact that $Y_\sigma$ is a subfunctor of $\GL_\nn$ which is an fppf-sheaf (or alternatively is an obvious direct calculation). For surjectivity, given $(g_1, \ldots, g_n) \in Y_\sigma(R_1) \times \ldots \times Y_\sigma(R_n)$, since $\GL_\nn$ is an fppf-sheaf, there exists a unique $g \in \GL_\nn( R_1 \times \ldots \times R_n)$, namely given by $g_{i , j} = ( (g_1)_{i, j} , \ldots, (g_n)_{i , j})$, which is sent to $(g_1, \ldots, g_n)$ under the map $Y_\sigma(R_1 \times \ldots \times R_n) \to Y_\sigma(R_1) \times \ldots \times Y_\sigma(R_n)$. But notice that for any finite subsets $P, Q \subseteq \nn$ with $|P| = |Q|$, we have $\text{det}_{P , Q} g = ( \text{det}_{P , Q} g_1, \ldots, \text{det}_{P , Q} g_n)$, so the vanishing of certain minors of $g_1, \ldots, g_n$, implies that the same minors of $g$ vanish as well, whence $g \in Y_\sigma(R_1 \times \ldots \times R_n)$. 

Second of all, given a $k$-algebra $R$ and an fppf-$R$-algebra $\varphi :  R \to S$, we use Proposition \ref{prop: characterization of equalizers in set} to show that the sequence $Y_\sigma(R) \to Y_\sigma(S) \rightrightarrows Y_\sigma(S \otimes_R S)$ is an exact sequence of sets. 

(EQ1). Injectivity of the map $Y_\sigma(R) \to Y_\sigma(S)$ follows from the fact that $\GL_\nn$ is an fppf-sheaf and $Y_\sigma$ is a subfunctor of $\GL_\nn$. 

(EQ2). Suppose that $h \in Y_\sigma(S)$ is sent to the same matrix under the two maps $Y_\sigma(S) \rightrightarrows Y_\sigma(S \otimes_R S)$. Again since $\GL_\nn$ is an fppf-sheaf, there exists $g \in \GL_\nn(R)$ such that $[\varphi(g_{i , j})] = [h_{i , j}]$. if we have that $\text{det}_{P , Q} (h) = 0$, then we compute
\begin{equation*}
\varphi(\text{det}_{P , Q} [g_{i , j}]) = \text{det}_{P , Q} ( [\varphi (g_{i , j} ) ] ) = \text{det}_{P , Q} (h) = 0. 
\end{equation*}

However recall (Proposition \ref{prop: equivalent characterizations of faithful flatness}) that $\varphi$ is injective, therefore $\text{det}_{P , Q} [g_{i , j}] = 0$. Hence $h \in Y_\sigma(S)$ implies $g \in Y_\sigma(R)$. 
\end{proof}

\begin{lemma}\label{lem: Y sigma stable under B mult}
$Y_\sigma$ is stable under left and right multiplication by $B$. 
\end{lemma}

\begin{proof}
Let $R$ be a $k$-algebra, $g \in Y_\sigma(R)$ and $b , c \in B(R)$. We wish to show that $h := b g c \in Y_\sigma$. Given $p , q \in \nn$, define $\iota_q : \Span(e_i \mid i \leq q) \to R^{\oplus \nn}$ to be the inclusion map, and $\pi_p : R^{\oplus \nn} \to \Span(e_i \mid i \geq p)$ the projection map associated to the direct sum decomposition $R^{\oplus \nn} =  \Span(e_i \mid i < p) \oplus \Span(e_i \mid i \geq p)$. Then one can see that the matrix of the $R$-linear map $\pi_p \circ h \circ \iota_q : \Span(e_i \mid i \leq q) \to \Span(e_i \mid i \geq p)$ with respect to the standard bases is the lower left submatrix $h_{\geq p , q}$ of $h$. Because $b$ and $c$ are upper triangular, we have that $\pi_p \circ b g c \circ \iota_q = b' \pi_p \circ g \circ \iota_q c'$ where $b'$ and $c'$ are also upper triangular, and indeed submatrices of $b$ and $c$ respectively.

This implies that the diagonal entries of $b'$ and $c'$ are units, so by Proposition \ref{prop: invertible upper triangulars}, they are both isomorphisms. Now let $\ell > r_{\geq p , q}(\sigma)$. Then by Proposition \ref{prop: minors vs wedge} all $\ell \times \ell$ minors of $\pi_p h  \iota_q$ vanish if and only if $\bigwedge^\ell( \pi_p h \iota_q ) = 0$. However we compute 
\begin{equation*}
\bigwedge^\ell( \pi_p  h  \iota_q ) = \bigwedge^\ell( \pi_p b g c \iota_q ) = \bigwedge^\ell( b' \pi_p g \iota_q c' ) \overset{(1)}{=} \bigwedge^\ell( b' ) \bigwedge^\ell ( \pi_p g \iota_q) \bigwedge^\ell(c'). 
\end{equation*}
Here (1) follows because $\bigwedge^\ell : \bMod_R \to \bMod_R$ is a functor, as does the fact that since $b'$ and $c'$ are isomorphisms, so are $\bigwedge^\ell( b' )$ and $\bigwedge^\ell( c' )$. Therefore $\bigwedge^\ell( b' ) \bigwedge^\ell ( \pi_p g \iota_q) \bigwedge^\ell(c') = 0$ if and only if $\bigwedge^\ell ( \pi_p g \iota_q) = 0$, which holds if and only if all $\ell \times \ell$ minors of $\pi_p g \iota_q$ vanish, which holds by definition of $\ell$. Therefore indeed all $\ell \times \ell$ minors of $\pi_p h  \iota_q$ vanish, so $h \in Y_\sigma$, as desired. 
\end{proof}

\begin{proposition}\label{prop: closed version of a schubert cell contains closure of the schubert cell}
Let $k$ be a \textbf{domain}. For any $\sigma , \tau \in \Perm(\nn)$ we have that $B \tau B \subseteq Y_\sigma$ if and only if $\tau \leq \sigma$. 
\end{proposition}

\begin{proof}
Suppose that $\tau \leq \sigma$. Since $\sigma \in Y_\sigma$, by Lemma \ref{lem: Y sigma stable under B mult} we have that $B \sigma B \subseteq Y_\sigma$. Then because $Y_\sigma$ is a closed subfunctor, this implies $\overline{B \sigma B} \subseteq Y_\sigma$, and hence by Corollary \ref{cor: containment of the closure of schubert cells}, we obtain $B \tau B \subseteq Y_\sigma$.  

Conversely, suppose that $B \tau B \subseteq Y_\sigma$. Then for any $k$-algebra $R$, we have that $\tau \in Y_\sigma(R)$, which implies that for any $p, q \in \nn$ and any $\ell > r_{\geq p , q}(\sigma)$, we have that all $\ell \times \ell$ minors of $\pi_p  \tau  \iota_q$ vanish. This implies that $r_{\geq p , q}(\tau) \leq r_{\geq p , q}(\sigma)$ and hence $r_{p , q}(\tau) \geq r_{p , q}(\sigma)$. Since $p$ and $q$ were arbitrary, by Proposition \ref{prop: region criterion for bruhat order} it follows that $\tau \leq \sigma$. 
\end{proof}

\begin{proposition}\label{prop: minors vs wedge}
Suppose that $f : R^n \to R^{\oplus \nn}$ is an $R$-linear map. Then all $\ell \times \ell$ minors of $f$ vanish if and only if $\bigwedge^\ell f = 0$. 
\end{proposition}

\begin{proof}
Denote the standard bases of $R^n$ and $R^{\oplus \nn}$ by $\{ d_i \mid 1 \leq i \leq n \}$ and $\{ e_j \mid j \in \nn \}$ respectively. Then a basis for $\bigwedge^\ell(R^n)$ is $\{ d_{i_1} \wedge \ldots \wedge d_{i_\ell} \mid 1 \leq i_1 < \ldots < i_\ell \leq n \}$ and a basis for $\bigwedge^\ell(R^{\oplus \nn})$ is $\{ e_{j_1} \wedge \ldots \wedge e_{j_\ell} \mid 1 \leq j_1 < \ldots < j_\ell \}$. In particular these exterior powers are both free, so we can compute the matrix of $\bigwedge^\ell f$ with respect to them. The entry of this matrix in the $(j_1, \ldots, j_\ell)  , (i_1, \ldots, i_\ell)$ position is precisely the $\ell \times \ell$ minor $\det_{\{ j_1, \ldots, j_\ell \}, \{ i_1, \ldots, i_\ell \}} f$. It therefore follows that all $\ell \times \ell$ minors of $f$ vanish if and only if $\bigwedge^\ell f = 0$. 
\end{proof}

\begin{proposition}\label{prop: y subfunctor is covered by schubert cells}
Let $k$ be a \textbf{domain}. Let $\sigma \in \Perm(\nn)$. Then $Y_\sigma$ is covered by the subfunctors $B \tau B$ for $\tau \leq \sigma$. 
\end{proposition}

\begin{proof}
We have just shown (Proposition \ref{prop: closed version of a schubert cell contains closure of the schubert cell}) that $B \tau B \subseteq Y_\sigma$ for all $\tau \leq \sigma$. Now suppose that $A$ is a $k$-algebra which is a field and let $g \in Y_\sigma(A)$. Since $A$ is a field, by the Bruhat decomposition we have that $g = b \tau c$ for $b , c \in B(A)$ and $\tau \in \Perm(\nn)$. By Lemma \ref{lem: Y sigma stable under B mult}, it follows that $\tau \in Y_\sigma(A)$ hence by the proof of Proposition \ref{prop: closed version of a schubert cell contains closure of the schubert cell}, it follows that $\tau \leq \sigma$. Therefore $g \in B \tau B(A)$ as desired. 
\end{proof}

\begin{theorem}
Let $k$ be a \textbf{domain} and let $\sigma \in \Perm(\nn)$. Then we have that $\overline{B \sigma B}$ is covered by the disjoint subfunctors $B \tau B$ for $\tau \leq \sigma$. 
\end{theorem}


\subsection{Closure in the Quotient}

We now turn our attention to the closure of the Schubert cells in the quotient $\GL_\nn/B$. Suppose that $k$ is a \textbf{domain}. We call the closure $\overline{\sh(B \sigma B) / B} \subseteq \GL_\nn/B$ a \emph{Schubert subfunctor} (in the finite-dimensional case it is called a \emph{Schubert variety}). Suppose that $\mX$ is a closed subfunctor of $\GL_\nn/B$ containing $\sh(B \sigma B) / B$. Then $\pi^{-1}(\mX)$ is a closed subfunctor of $\GL_\nn$ which contains $B \sigma B$, where $\pi : \GL_\nn \to \GL_\nn/B$ denotes the fppf quotient map. Hence by Corollary \ref{cor: containment of the closure of schubert cells}, if $\tau \leq \sigma$ then we have that $\overline{\sh(B \tau B)} = \overline{B \tau B} \subseteq \pi^{-1}(\mX)$. It therefore follows that $\sh(B \tau B) / B$ is contained in $\sh( \mX)$. But the latter is equal $\mX$ because $\GL_\nn/B$ is an fppf sheaf and closed subfunctors of fppf sheaves are again fppf sheaves (Corollary \ref{cor: closed subfunctors of sheaves are sheaves}). Since $\mX$ was an arbitrary closed subfunctor containing $\sh(B \sigma B) / B$, it follows that if $\tau \leq \sigma$ then $\sh(B \tau B) / B \subseteq \overline{\sh(B \sigma B) / B}$. The proof of the following proposition is trivial. 

\begin{proposition}\label{ref: quotient of y subfunctor is closed}
$Y_\sigma / B$ is a closed subfunctor of $\GL_\nn/B$. 
\end{proposition}

It follows from Proposition \ref{ref: quotient of y subfunctor is closed} that $\overline{\sh(B \sigma B) / B} \subseteq Y_\sigma / B$. The proof of the following result is similar to that of Proposition \ref{prop: y subfunctor is covered by schubert cells}. It is an easy application of the Bruhat decomposition. 

\begin{proposition}
Let $k$ be a \textbf{domain}. Let $\sigma \in \Perm(\nn)$. Then $Y_\sigma / B$ is covered by the subfunctors $B \tau B / B$ for $\tau \leq \sigma$. 
\end{proposition}

As a consequence of this we obtain our closure result. 

\begin{theorem}
Let $k$ be a \textbf{domain} and let $\sigma \in \Perm(\nn)$. Then we have that the Schubert subfunctor $\overline{\sh(B \sigma B) / B}$ is covered by the disjoint subfunctors $\sh(B \tau B) / B$ for $\tau \leq \sigma$. 
\end{theorem}



\bibliographystyle{alpha}

\bibliography{AG_paper}

\begin{thebibliography}{{Sta}21}

\bibitem[BLP11]{bautista-liu-paquette2011}
Raymundo Bautista, Shiping Liu, and Charles Paquette.
\newblock Representation theory of an infinite quiver.
\newblock {\em arXiv preprint arXiv:1109.3176}, 2011.

\bibitem[Che94]{chevalley1994}
Claude Chevalley.
\newblock Sur les d{\'e}compositions cellulaires des espaces g/b.
\newblock {\em Algebraic Groups and their Generalizations: Classical Methods,
  American Mathematical Society}, pages 1--23, 1994.

\bibitem[DG70]{demazure-gabriel1970}
Michel Demazure and Pierre Gabriel.
\newblock {\em Groupes alg{\'e}briques. Tome I. G{\'e}om{\'e}trie
  alg{\'e}brique g{\'e}n{\'e}ralit{\'e}s. Groupes commutatifs}.
\newblock North-Holland, 1970.

\bibitem[DG80]{demazure-gabriel1980}
Michel Demazure and Peter Gabriel.
\newblock {\em Introduction to algebraic geometry and algebraic groups}.
\newblock Elsevier, 1980.

\bibitem[DP04]{dimitrov-penkov2004}
Ivan Dimitrov and Ivan Penkov.
\newblock Ind-varieties of generalized flags as homogeneous spaces for
  classical ind-groups.
\newblock {\em International Mathematics Research Notices},
  2004(55):2935--2953, 2004.

\bibitem[Dri06]{drinfeld2006}
Vladimir Drinfeld.
\newblock Infinite-dimensional vector bundles in algebraic geometry.
\newblock In {\em The unity of mathematics}, pages 263--304. Springer, 2006.

\bibitem[Ehr34]{ehresmann1934}
Charles Ehresmann.
\newblock {\em Sur la topologie de certains espaces homogenes}.
\newblock PhD thesis, Facult{\'e} des sciences de Paris, 1934.

\bibitem[FP15]{fresse-penkov2015}
Lucas Fresse and Ivan Penkov.
\newblock Schubert decompositions for ind-varieties of generalized flags, 2015.

\bibitem[Ful97]{fulton1997}
William Fulton.
\newblock {\em Young tableaux: with applications to representation theory and
  geometry}.
\newblock Number~35. Cambridge University Press, 1997.

\bibitem[Gro62]{grothendieck1962}
Alexandre Grothendieck.
\newblock {\em Fondements de la g{\'e}om{\'e}trie alg{\'e}brique: extraits du
  S{\'e}minaire Bourbaki, 1957-1962}.
\newblock Secr{\'e}tariat math{\'e}matique, 1962.

\bibitem[GZ12]{gabriel-zisman2012}
Peter Gabriel and Michel Zisman.
\newblock {\em Calculus of fractions and homotopy theory}, volume~35.
\newblock Springer Science \& Business Media, 2012.

\bibitem[IP17]{ignatyev-penkov2017}
Mikhail~V. Ignatyev and Ivan Penkov.
\newblock Ind-varieties of generalized flags: a survey of results, 2017.

\bibitem[Jan03]{jantzen2003}
Jens~Carsten Jantzen.
\newblock {\em Representations of algebraic groups}, volume 107.
\newblock American Mathematical Soc., 2003.

\bibitem[Kap58]{kaplansky1958}
Irving Kaplansky.
\newblock Projective modules.
\newblock {\em Annals of Mathematics}, pages 372--377, 1958.

\bibitem[Kap18]{kaplansky2018}
Irving Kaplansky.
\newblock {\em Infinite abelian groups}.
\newblock Courier Dover Publications, 2018.

\bibitem[Kar00]{karpenko2000}
Nikita~A. Karpenko.
\newblock Cohomology of relative cellular spaces and of isotropic flag
  varieties.
\newblock {\em Algebra i Analiz}, 12(1):3--69, 2000.

\bibitem[Lus10]{lusztig2010}
G~Lusztig.
\newblock Bruhat decomposition and applications.
\newblock {\em arXiv preprint arXiv:1006.5004}, 2010.

\bibitem[ML13]{maclane2013}
Saunders Mac~Lane.
\newblock {\em Categories for the working mathematician}, volume~5.
\newblock Springer Science \& Business Media, 2013.

\bibitem[NS21]{nagpal-snowden2021}
Rohit Nagpal and Andrew Snowden.
\newblock Symmetric subvarieties of infinite affine space, 2021.

\bibitem[OPV06]{onn-prasad-vaserstein2006}
Uri Onn, Amritanshu Prasad, and Leonid Vaserstein.
\newblock A note on bruhat decomposition ofgl(n) over local principal ideal
  rings.
\newblock {\em Communications in Algebra}, 34(11):4119–4130, Nov 2006.

\bibitem[Per10]{perry2010}
Alexander Perry.
\newblock Faithfully flat descent for projectivity of modules.
\newblock {\em arXiv preprint arXiv:1011.0038}, 2010.

\bibitem[RG71]{raynaud-gruson1971}
Michel Raynaud and Laurent Gruson.
\newblock Criteres de platitude et de projectivit{\'e}.
\newblock {\em Inventiones mathematicae}, 13(1):1--89, 1971.

\bibitem[Rin16]{ringel2016}
Claus~Michael Ringel.
\newblock Representation theory of {Dynkin} quivers. {Three} contributions.
\newblock {\em Frontiers of Mathematics in China}, 11(4):765--814, 2016.

\bibitem[{Sta}21]{stacks-project}
The {Stacks project authors}.
\newblock The stacks project.
\newblock \url{https://stacks.math.columbia.edu}, 2021.

\end{thebibliography}

\end{document}